\documentclass[11pt,reqno]{amsart}

\usepackage{lineno,hyperref}
\modulolinenumbers[5]
\usepackage{amsmath}
\usepackage{amssymb}
\usepackage{mathrsfs}
\usepackage{amsthm}
\usepackage{bm}
\usepackage{bbm}
\usepackage{graphicx}
\usepackage{booktabs}
\usepackage{float}
\usepackage{subfig}
\usepackage{geometry}
\usepackage{color}
\newtheorem{theorem}{Theorem}[section]
\newtheorem{lemma}[theorem]{Lemma}
\newtheorem{proposition}[theorem]{Proposition}

\newtheorem{assumption}{Assumption}
\theoremstyle{remark}
\newtheorem{remark}[theorem]{Remark}

\newcommand{\ud}{\mathrm d}
\newcommand{\D}{\mathbb D}
\newcommand{\OO}{\mathcal O}

\newcommand{\E}{\mathbb E}

\numberwithin{equation}{section}
\allowdisplaybreaks[4]

\begin{document}

\title[Convergence of Density Approximations ]{Density convergence of a fully discrete finite difference method for  stochastic Cahn--Hilliard equation}

\author{Jialin Hong}
\address{LSEC, ICMSEC, Academy of Mathematics and Systems Science, Chinese Academy of Sciences, and School of Mathematical Sciences, University of Chinese Academy of Sciences, Beijing 100049, China}
\email{hjl@lsec.cc.ac.cn}

\author{Diancong Jin}
\address{School of Mathematics and Statistics, Huazhong University of Science and Technology,  and Hubei Key Laboratory of Engineering Modeling and Scientific Computing, Huazhong University of Science and Technology, Wuhan 430074, China}
\email{jindc@hust.edu.cn}

\author{Derui Sheng}
\address{LSEC, ICMSEC, Academy of Mathematics and Systems Science, Chinese Academy of Sciences, and School of Mathematical Sciences, University of Chinese Academy of Sciences, Beijing 100049, China}
\email{sdr@lsec.cc.ac.cn}

\thanks{This work is supported by the National key R\&D Program of China under Grant No.\ 2020YFA0713701, National Natural Science Foundation of China (Nos.\ 11971470, 11871068, 12031020, 12022118, 12026428, 11926417), and the Fundamental Research Funds for the Central Universities 3004011142.}


\begin{abstract}
This paper focuses on investigating  the density convergence  of a fully discrete finite difference method when applied to numerically solve the stochastic Cahn--Hilliard equation driven by multiplicative space-time white noises. The main difficulty lies in the control of the drift coefficient
that is neither globally Lipschitz nor one-sided Lipschitz. To handle this difficulty, we propose a novel localization argument and derive the strong convergence rate of the numerical solution to estimate the total variation distance between the exact and numerical solutions. This along with the existence of the density of the numerical solution finally yields the convergence of density in $L^1(\mathbb{R})$ of the numerical solution.  Our results partially answer positively to the open problem emerged in [\textit{J. Cui and J. Hong, J. Differential Equations (2020)}] on computing the density of the exact solution numerically.
\end{abstract}

\maketitle

\section{Introduction}
The density of the exact solution of a stochastic system characterizes all relevant probabilistic information and has wide applications in the probability potential theory. 
When a numerical method is applied to the original system, 
a natural question is whether the numerical solution provides an effective approximation of the density of the exact solution, which has received much
attention recently.
For instance, for stochastic differential equations (SDEs) 
whose coefficients are smooth vector fields with bounded derivatives, \cite{BT96,HW96,KHA97} obtained the convergence of density of the numerical solution based on It\^o--Taylor type dicretizations under H\"ormander's condition. For stochastic Langevin equations with non-globally monotone coefficients, \cite{CHS19} used the splitting method to derive an approximation for the density of the exact solution. 
Relatively,  the research of approximations for densities of exact solutions of stochastic partial differential equations (SPDEs) is still at its infancy. And we are only aware of \cite{CCHS20}, where the authors investigated the existence and convergence of densities of numerical dicretizations for stochastic heat equations with additive noise.
Following this line of investigation, the present work makes further contributions on numerical approximations of densities of exact solutions of SPDEs with polynomial nonlinearity and multiplicative noises.

This paper is concerned with 
the following stochastic Cahn--Hilliard equation
\begin{gather}\label{CH}
	\begin{split}
		\partial_t u+\Delta^2u&=\Delta f(u)+\sigma(u)\dot{W},\quad \text{in}~[0,T]\times\OO
	\end{split}
\end{gather}
with the initial value $u(0,\cdot)=u_0$ and Dirichlet boundary conditions (DBCs) $u=\Delta u=0$ on $\partial \OO$. Here, $\OO:=(0,\pi)$, $T>0$, and $\dot{W}$ is the formal derivative of  a Brownian sheet $W=\{W(t,x),(t,x)\in[0,T]\times\OO\}$ defined on some complete probability space $(\Omega,\mathscr F,\{\mathscr F_t\}_{t\in[0,T]}, \mathbb P).$ The stochastic Cahn–Hilliard equation is  a model arising in non-equilibrium dynamics of metastable states \cite{CHC1961,CHC1958,CHC71,CHC1970,CHC1971}. For example, \eqref{CH} can describe the complicated phase separation and coarsening phenomena in a melted alloy that is quenched to a temperature at which only two different concentration phases can exist stably \cite{CH58,DN91}. 
The unknown quantity $u$ in \eqref{CH} represents the concentration, and $f(u)=u^3-u$ is the derivative of the double well potential.
For the  existence of a unique solution to \eqref{CH} 
under suitable assumptions, we refer to  \cite{AKM16,CC01,CH19} and references therein. 
Concerning the density of the exact solution, \cite{CC01} and \cite{CC02,CH20}  respectively proved the existence and strict positivity of the density $\{p_{t,x}\}_{(t,x)\in(0,T]\times\mathcal{O}}$ of $\{u(t,x)\}_{(t,x)\in(0,T]\times\mathcal{O}}$, under the non-degeneracy condition $\vert \sigma(\cdot)\vert >0$. From the practical point of view, 
numerically approximating the density $p_{t,x}$  is of prime importance in understanding intrinsic properties, beyond the existence, of the density, and we will 
resort to numerical methods to handle this problem.

Numerical methods have been successfully applied to solve the stochastic Cahn--Hilliard equation; see \cite{CCZZ18,LM11} for that of the linearized Cahn--Hilliard equation, \cite{CHS21,DN91,FKLL18,KLM11,QW20}  for the additive noise case, and \cite{CH19,CH20,ZL22,FLZ20} for the multiplicative noise case. 
Among them, the stochastic Cahn--Hilliard equation is interpreted as an SDEs in Hilbert spaces, which is discretized by
the finite element method or the spectral Galerkin method in space. 
In order to numerically approximating the density of the exact solution, we understand the exact solution $u:[0,T]\times \mathcal{O}\rightarrow L^2(\Omega)$ as a random field and apply the spatial
finite difference method (FDM) to discretize \eqref{CH}.
By introducing a uniform spatial 
stepsize $h=\pi/n$, $n\ge2$,  the spatial FDM of \eqref{CH} can be formulated into an $(n-1)$-dimensional SDE
\begin{align}\label{NCH}
	\ud U(t)+A_n^2U(t)\ud t= A_nF_n(U(t))\ud t+\sqrt{n/\pi}\Sigma_n(U(t)) \ud \beta_t.
\end{align}
Here  $A_n$ is  the matrix form of the discrete Dirichlet Laplacian, $F_n$ and $\Sigma_n$ are
respectively determined by $f$ and $\sigma$ (see \eqref{FnSigma}), and $\{\beta_t\}_{t\in[0,T]}$ is some $(n-1)$-dimensional Brownian motion related to $W$ (see subsection \ref{S31} for more details). 
Further, by denoting $\tau:=T/m$ $(m\in\mathbb{N}_+)$ the uniform time stepsize, we 
discretize \eqref{NCH} by the backward Euler scheme in time and obtain a fully discrete FDM method
\begin{align}\label{eq:Ui}
U^{i+1}-U^i+\tau A_n^2 U^{i+1}=\tau A_n F_n(U^{i+1})+\sqrt{n/\pi}\Sigma_n(U^i)(\beta_{t_{i+1}}-\beta_{t_{i}}),
\end{align}
where $t_i:=i\tau$ for $i\in \{0,1,\ldots,m\}$.
The $k$th component $U_k(t_i)$ (resp.\ $U^i_k$) of $U(t_i)$ (resp.\ $U^i$) approximates formally to $u(t_i,kh)$ for every $k\in\mathbb{Z}_{n-1}:=\{1,\ldots,n-1\}$.
 Taking advantage of the local weak monotonicity  of $A_nF_n$, we prove the well-posedness of  \eqref{NCH} and \eqref{eq:Ui}, and that for every $k\in\mathbb{Z}_{n-1}$, both $\{U_k(t)\}_{t\in(0,T]}$ and $\{U_k^i\}_{i=1,\ldots,m}$ admit densities under the non-degeneracy condition. 

Our first main result is the strong convergence rates of the spatial FDM and fully discrete FDM for \eqref{CH}, which will be used to derive the convergence of densities of numerical methods, and is of independent interest. 
Inspired by regularity estimates of original systems in \cite{AKM16,CC01},
we first use the interpolation approach to establish in Proposition \ref{Unl2-pro}  a uniform moment bound for the numerical solution $U(t)$ in the discrete Sobolev norm $\Vert (-A_n)^{\frac{1}{2}}\cdot\Vert _{l^2_n}$. To overcome the difficulty that the drift coefficient is neither globally Lipschitz  nor one-sided Lipschitz, we introduce an auxiliary process $\tilde U(t)$ (see \eqref{anx_u}), and focus mainly on estimating the error $E(t):=\tilde U(t)- U(t)$.
With the aid of the one-sided Lipschitz property of $A_nF_n$ in the discrete negative Sobolev norm $\Vert (-A_n)^{-\frac{1}{2}}\cdot\Vert _{l^2_n}$,
we are able to estimate $(-A_n)^{-\frac{1}{2}}E(t)$ in Proposition \ref{H-1E}, and meanwhile the linear part $-A_n^2E(t)$ leads to a upper bound for the 
$L^4(\Omega;L^2(0,T;\Vert (-A_n)^{\frac{1}{2}}\cdot\Vert _{l^2_n}))$-norm of
$E(t)$. Further,
building on the local Lipschitz continuity of $F_n$ in the norm $\Vert (-A_n)^{\frac{1}{2}}\cdot\Vert _{l_n^2}$, we attain the strong convergence order $1$ of the spatial FDM. By essentially exploiting the discrete analogue of previous arguments, we also show that the fully discrete FDM converges strongly to the spatial FDM with order nearly $\frac38$ in time. The above convergence orders are optimal in the sense that they coincide with the spatial and temporal H\"older continuity exponents of the exact solution, respectively.

Our second main result is the convergence of density in $L^1(\mathbb{R})$ of the numerical solutions for \eqref{CH}, which is realized by a localization argument to deal with the non-globally monotone coefficient $\Delta f$. Let us illustrate our idea by taking the spatial semi-discrete numerical solution for example. First, we establish a  criterion for reducing the total variation distance of random variables to that of their localizations in Proposition \ref{prop:localization}. Based on this criterion, the estimate of the total variation distance between $u$ and $u^n$ boils down to estimating that between $u_R$ and $u^n_R$. Here, $u_R$ is the localization of $u$ and solves the localized stochastic Cahn--Hilliard equation
\begin{equation*}
			\partial_t u_R+\Delta^2u_R=\Delta f_R(u_R)+\sigma(u_R)\dot{W},\quad R\ge1,
\end{equation*}
where $f_R=fK_R$ with $K_R$ being a smooth cut-off function supported on $[-R-1,R+1]$. In addition, 
 for any fixed  $ R\ge 1$, define the localization  $u^n_R$ of $u^n$ as the spatial FDM numerical solution of $u_R$.
Second, in order to control the total variation distance between $u^n_R$ and $u_R$, we apply 
a criterion for the convergence in total variation distance provided by \cite{NP13}, whose prerequisites
contain the negative moment estimate of the Malliavin derivative $Du_R(t,x)$  and the convergence of $u_R^n(t,x)$ in the Malliavin--Sobolev space $\D^{1,2}$. These are accomplished by making full use of the globally Lipschitz condition of $f_R$ and the strong regularizing effect of the linear part. Finally, together with the existence of density of the spatial FDM, we obtain that the density of the spatial FDM converges in $L^1(\mathbb{R})$ to that of the exact solution.
In a similar manner, we also show that the density of the fully discrete FDM converges in $L^1(\mathbb{R})$ to that of the exact solution.

We summarize main contributions of this work as follows.
\begin{itemize}
	\item[$\bullet$] We give the optimal strong convergence rate of a fully discrete FDM for stochastic Cahn--Hilliard equations with polynomial nonlinearity and multiplicative noise. 
	
	\item[$\bullet$] 
	We are the first to give the convergence of density for numerical approximations of SPDEs with polynomial nonlinearity. 
	The results on the existence and convergence of density of the numerical solutions for stochastic Cahn--Hilliard equations partially respond positively to an open problem on computing the density of the exact solution numerically proposed in \cite[Section 5]{CH20}.

	\item[$\bullet$] 
	We propose a criterion for reducing the  total variation distance of  random variables to that of their localizations. And it is successfully applied to derive the density convergence of a fully discrete numerical method for \eqref{CH}.	
We believe that this localization argument is also available for other SPDEs with non-globally Lipschitz coefficients such as stochastic Allen--Cahn equations.
\end{itemize}

The rest of this paper is organized as follows. Section \ref{S2} is devoted to introducing  the mild solution, the spatial and fully discrete FDMs for \eqref{CH}. In subsection \ref{S31}, we also prove the existence of densities of the numerical solutions.
The regularity estimates of the numerical solutions are presented in Section \ref{S4}.
The strong convergence rate of the spatial FDM and the fully discrete FDM are proved in Sections \ref{S5} and \ref{S4-5}, respectively. Finally, Section \ref{S6} is reserved for the convergence of 
  densities in $L^1(\mathbb{R})$ of the numerical solutions.
\section{Preliminaries}\label{S2}

Let $\mathcal{C}^\alpha(\OO)$ be the space of $\alpha$-H\"older continuous functions on $\OO$ for $\alpha\in(0,1)$, and the space of $\alpha$ times continuously differentiable functions on $\OO$ for $\alpha\in\mathbb N$. For $d\in\mathbb{N}_+$, we denote by $\Vert \cdot\Vert $ and $\langle\cdot,\cdot\rangle$ the Euclidean norm and inner product of $\mathbb{R}^d$, respectively. 
Given a measurable space $(\mathbb{M},\mathcal M,\mathbf{m})$ and
a Banach space $(H,\Vert \cdot\Vert _H)$, let $L^p(\mathbb{M};H)$ be the space of  measurable functions $g:\mathbb{M}\to H$ endowed with the usual norm 
$\|g\|_{L^p(\mathbb{M};H)}:=\left(\int_\mathbb{M}\|g\|_{H}^p\ud \mathbf{m}\right)^{\frac1p}$.
Especially, we write $L^p(\mathbb M):=L^p(\mathbb M;\mathbb{R})$ for short.
For $N\in\mathbb{N}_+$, denote $\mathbb{Z}_{N}:=\{1,\ldots,N\}$ and $\mathbb{Z}_{N}^0:=\{0,1,\ldots,N\}$. 
We use $C$ to denote a generic positive constant that may change from one place to another and depend on several parameters but never on the stepsize $h$. 

Given a random field $v=\{v(t,x),(t,x)\in[0,T]\times \OO\}$ and a kernel $S:(0,T]\times\OO\times\OO\to \mathbb{R}$, we denote $S*f(v)$ and $S\diamond \sigma(v)$ the deterministic and stochastic convolutions, respectively, namely for any $(t,x)\in[0,T]\times\OO$,
\begin{gather}\label{eq:SDC}
S*f(v)(t,x):=\int_0^t\int_\OO S_{t-s}(x,y)f(v(s,y))\ud y\ud s,\\\label{eq:SSC}
S\diamond\sigma(v)(t,x):=\int_0^t\int_\OO S_{t-s}(x,y)\sigma(v(s,y))W(\ud s,\ud y).
\end{gather}
If for any $p\ge1$, there exists some constant $C_p$ such that $\|v(t,x)\|_{L^p(\Omega)}\le C_p$ for all $(t,x)\in[0,T]\times\OO$, then it can be verified that for any $0\le s\le t\le T$ and $x,y\in\OO$,
\begin{align}\notag\label{eq:SD}
\|S* f(v)(t,x)&-S*f(v)(s,y)\|_{L^p(\Omega)} \le C\int_0^s\int_\OO |S_{t-r}(x,z)-S_{s-r}(x,z)|\ud z\ud r\\
& +C\int_s^t\int_\OO |S_{t-r}(x,z)|\ud z\ud r+C\int_0^s\int_\OO |S_{s-r}(x,z)-S_{s-r}(y,z)|\ud z\ud r,
\end{align}
and 
\begin{align}\notag\label{eq:SS}
\|S\diamond \sigma(v)&(t,x)-S\diamond \sigma(v)(t,y)\|^2_{L^p(\Omega)}\le C\int_0^s\int_{\mathcal{O}}\vert S_{s-r}(x,z)-S_{s-r}(y,z)\vert ^2\mathrm{d} z\mathrm{d} r\\
&\quad+C\int_0^s\int_{\mathcal{O}}\vert S_{t-r}(x,z)-S_{s-r}(x,z)\vert ^2\mathrm{d} z\mathrm{d} r+C\int_s^t\int_{\mathcal{O}}\vert S_{t-r}(x,z)\vert ^2\mathrm{d} z\mathrm{d} r.
\end{align}

\subsection{Mild solution}
The physical importance of the Dirichlet problem lies in that
  it governs the propagation of a solidification front into an ambient medium which is at rest relative to the front \cite{DN91}; see for instance \cite{CH20,CM04,EL92} for the study of Cahn--Hilliard equations with DBCs. In this case,
the Green function associated to $\partial_t+\Delta^2$  is  given by
$G_t(x,y)=\sum_{j=1}^\infty e^{-\lambda_j^2t}\phi_j(x)\phi_j(y), t\in[0,T],x,y\in\mathcal O$,
where $\lambda_j=-j^2$, $\phi_j(x)=\sqrt{2/\pi}\sin(jx)$, $j\ge1$. 
It is known that $\{\phi_j\}_{j\ge1}$ forms an orthonormal basis of $L^2(\mathcal O)$.
As pointed out in \cite[p.19]{CM04}, there exist $C,c>0$ such that
\begin{align}\label{Gtxy}
\vert G_t(x,y)\vert \le \frac{C}{t^{1/4}}\exp\Big(-c\frac{\vert x-y\vert ^{4/3}}{\vert t\vert ^{1/3}}\Big),\\\label{DGtxy}
\vert \Delta G_t(x,y)\vert \le  \frac{C}{t^{3/4}}\exp\Big(-c\frac{\vert x-y\vert ^{4/3}}{\vert t\vert ^{1/3}}\Big),
\end{align}
which corresponds to \cite[formula (1.2)]{CM04} with $a=b=0$ and $a=2,b=0$, respectively.

\textit{Without further explanations,
we always assume in the text that $u_0:\OO\rightarrow \mathbb{R}$ is nonrandom and continuous, $f(x)=x^3-x$, and
$\sigma:\mathbb{R}\rightarrow\mathbb{R}$ is bounded and satisfies the globally Lipschitz condition.} These assumptions ensure that \eqref{CH}  admits a
unique mild solution $u=\{u(t,x),(t,x)\in[0,T]\times\OO\}$
 given by (cf. \cite{CC01,CH20,CH22})
  \begin{align*}
u(t,x)&=\mathbb{G}_tu_0(x)+
\int_0^t\int_\mathcal{O} \Delta G_{t-s}(x,y)f(u(s,y))\mathrm{d} y\mathrm{d} s\\
&\quad+\int_0^t\int_\mathcal{O} G_{t-s}(x,y)\sigma(u(s,y))W(\mathrm{d} s,\mathrm{d} y),\quad(t,x)\in[0,T]\times\mathcal{O}.
\end{align*}
Hereafter, $\mathbb{G}_tv(x):=\int_{\mathcal O}G_t(x,y)v(y)\mathrm{d} y$ for $v\in \mathcal{C}(\mathcal{O})$. 
   Moreover, as shown in \cite[Proposition 5.2]{CH22}, the exact solution to \eqref{CH} satisfies
 \begin{align}\label{eq:bound-e}
\sup_{t\in[0,T]}\E\left[\sup_{x\in\mathcal{O}}\vert u(t,x)\vert^p\right]\le C(u_0,T,p).
 \end{align}

Similar to \cite[Lemma 1.8]{CC01}, we have the following regularity estimate of $G$.
 \begin{lemma}\label{Greg}
For $\alpha\in(0,1)$, there exists $C=C_\alpha$ such that for $x,y\in\mathcal{O}$ and $t>s$,
\begin{gather*}
\int_0^t\int_{\mathcal{O}}\vert G_{t-r}(x,z)-G_{t-r}(y,z)\vert ^2\mathrm{d} z\mathrm{d} r\le C\vert x-y\vert ^2,\\
\int_0^s\int_{\mathcal{O}}\vert G_{t-r}(x,z)-G_{s-r}(x,z)\vert ^2\mathrm{d} z\mathrm{d} r+\int_s^t\int_{\mathcal{O}}\vert G_{t-r}(x,z)\vert ^2\mathrm{d} z\mathrm{d} r\le C\vert t-s\vert ^{\frac{3}{4}\alpha},\\
\int_0^t\int_{\mathcal O}\vert \Delta G_{t-r}(x,z)-\Delta G_{t-r}(y,z)\vert \mathrm{d} z\mathrm{d} r\le C\vert x-y\vert,\\
\int_0^s\int_{\mathcal O}\vert \Delta G_{t-r}(x,z)-\Delta G_{s-r}(x,z)\vert \mathrm{d} z\mathrm{d} r+\int_s^t\int_{\mathcal O}\vert \Delta G_{t-r}(x,z)\vert \mathrm{d} z\mathrm{d} r\le C\vert t-s\vert ^{\frac{3\alpha}{8}}.
\end{gather*}
\end{lemma}
Then we are able to investigate the H\"older continuity of $u$.

\begin{lemma}\label{Holder-exact}
Let $u_0\in\mathcal C^2(\mathcal O)$ and $\alpha\in(0,1)$. Then for $p\ge1$, there exists some constant $C=C(\alpha,p,T)$ such that
\begin{align*}
\Vert u(t,x)-u(s,y)\Vert _{L^p(\Omega)}\le C(\vert t-s\vert ^{\frac{3\alpha}{8}}+\vert x-y\vert )\quad\forall~(t,x),(s,y)\in[0,T]\times\mathcal{O}.
\end{align*}
\end{lemma}
\begin{proof}
Without loss of generality, let $p\ge2$ and $t>s$.
Using \cite[Lemma 2.3]{CC01}
and the assumption $u_0\in\mathcal C^{2}(\mathcal{O})$, we get
$$\vert \mathbb{G}_tu_0(x)-\mathbb{G}_tu_0(y)\vert +\vert \mathbb{G}_tu_0(x)-\mathbb{G}_su_0(x)\vert \le C(\vert t-s\vert ^{\frac{1}{2}}+\vert x-y\vert ).$$
Then by \eqref{eq:SD} with $S=\Delta G$ and \eqref{eq:SS} with $S=G$, as well as Lemma \ref{Greg}, we finish the proof.
\end{proof}

\subsection{Spatial FDM}\label{S31}
In this part, we introduce the spatial FDM  for \eqref{CH} and present the regularity estimates of the spatial discrete Green function.
Given a function $w$ defined on the mesh $\{0,h,2h,\ldots,\pi\}$ with $h=\frac{\pi}{n}$,
define the difference operator 
\begin{align*}
\delta_h w_i:&=\frac{w_{i-1}-2w_i+w_{i+1}}{h^2},\quad i\in \mathbb{Z}_{n-1},
\end{align*}
where $w_i:=w(ih)$.  Notice that
$$\delta_h^2 w_i=\frac{w_{i-2}-4w_{i-1}+6w_i-4w_{i+1}+w_{i+2}}{h^4},\quad i\in \mathbb{Z}_{n-1}.$$
The compatibility conditions  $u_0(0)=u_0(\pi)=0$ and $u_0^{\prime\prime}(0)=u_0^{\prime\prime}(\pi)=0$ are direct results of DBCs and the initial condition. One can approximate $u(t,kh)$ via $\{u^n(t,kh)\}_{n\ge2}$, where $u^n(0,kh)=u_0(kh)$ and
\begin{align}\label{unt}
&\quad\ \ud u^n(t,kh)+\delta_h^2u^n(t,kh)\ud t\\\notag
&=\delta_h f(u^n(t,kh))\ud t+h^{-1}\sigma(u^n(t,kh))\ud(W(t,(k+1)h)-W(t,kh))
\end{align}
for $t\in[0,T]$ and $k\in\mathbb{Z}_{n-1}$,
under the boundary conditions 
\begin{gather*}
u^n(t,0)=u^n(t,\pi)=0,\\
u^n(t,-h)+u^n(t,h)=u^n(t,(n-1)h)+u^n(t,(n+1)h)=0
\end{gather*}
for $t\in(0,T]$.
Define $\Pi_n$ as an interpolation operator which gives the polygonal interpolation of a function defined on the spatial grid points, i.e.,
$$\Pi_n(\varphi)(x)=\varphi(\kappa_n(x))+n\pi^{-1}(x-\kappa_n(x))(\varphi(\kappa_n(x)+h)-\varphi(\kappa_n(x))),\quad x\in\OO,$$
where $\kappa_n(y)=h\lfloor y/h\rfloor$  with $\lfloor\cdot\rfloor$ being the floor function.
Then we define $u^n(t,x):=\Pi_n(u^n(t,\cdot))(x)$.
To solve \eqref{unt}, introduce $$U(t)=(U_1(t),\ldots,U_{n-1}(t))^\top,\quad\beta_t=(\beta^1_t,\ldots,\beta^{n-1}_t)^{\top}$$ with
$U_k(t):=u^n(t,kh)$ and $\beta^k_t:=\sqrt{n/\pi}(W(t,(k+1)h)-W(t,kh))$ for $k\in\mathbb{Z}_{n-1}$, where the explicit dependence of $U(t)$ and $\beta_t$ on $n$ is omitted.
Let $A_n\in\mathbb{R}^{(n-1)\times(n-1)}$ be the matrix form of the discrete Dirichlet Laplacian, i.e., \begin{align*}
A_n:=\frac{n^2}{\pi^2}\left[\begin{array}{cccccc}-2 & 1 & 0 & \cdots & \cdots & 0 \\1 & -2 & 1 & \ddots & \ddots & \vdots \\0 & 1 & -2 & \ddots & \ddots & \vdots \\\vdots & 0 & \ddots & \ddots & \ddots & 0 \\\vdots & \ddots & \ddots & 1 & -2 & 1 \\0 & \cdots & \cdots & 0 & 1 & -2\end{array}\right].
\end{align*} Then
\eqref{unt} can be rewritten into the $(n-1)$-dimensional SDE \eqref{NCH}
with the initial value $U(0)=(u_0(h),\ldots,u_0((n-1)h))^\top$ and the coefficients 
\begin{equation}\label{FnSigma}
\begin{split}
F_n(U(t))&=(f(U_1(t)),\ldots,f(U_{n-1}(t)))^\top,\\
\Sigma_n(U(t))&=\textrm{diag}(\sigma(U_1(t)),\ldots,\sigma(U_{n-1}(t))).
\end{split}
\end{equation}

\begin{lemma}
Eq.\ \eqref{NCH} has a
unique strong solution $U=\{U(t),t\in[0,T]\}$.
\end{lemma}
\begin{proof}
Since $f(a)=a^3-a$, there is some $c_0>0$ such that for $a,b\in\mathbb{R}$,
\begin{gather}\label{fbfa}
(f(b)-f(a))(a-b)=-(a^2+b^2+ab)(a-b)^2+(a-b)^2\le(a-b)^2,\\\label{fb-fa}
\vert f(b)-f(a)\vert \le c_0(1+a^2+b^2)\vert b-a\vert . 
\end{gather}  
Based on  \eqref{fbfa} and \eqref{fb-fa}, it can be verified that there exist  $K_n(R), K_n>0$ such that for all $R\in(0,\infty)$, $x,y\in\mathbb{R}^{n-1}$ with $\Vert x\Vert  ,\Vert y\Vert  \le R$,
\begin{gather}\label{lwm}
\langle x-y,A_nF_n(x)-A_nF_n(y)\rangle \le K_n(R)\Vert x-y\Vert  ^2~ \text{(local weak monotonicity)},\\
\label{wc}
\langle x,A_nF_n(x)\rangle \le K_n(1+\Vert x\Vert  ^2) ~\text{(weak coercivity)}.
\end{gather}
By virtue of \cite[Theorem 3.1.1]{PR07}, \eqref{lwm}, \eqref{wc} and the Lipschitz continuity of $\sigma$, we obtain that
\eqref{NCH} admits a unique solution $\{U(t),t\in[0,T]\}$, which is a.s. continuous and $\{\mathscr F_t\}$-adapted.
\end{proof}
By virtue of  \cite[Theorem 5]{CC01} and \cite[Remark 5.3($ii$)]{CM04}, we know that if  $\sigma(x)\neq0$ for any $x\in\mathbb{R}$, then for any $(t,x)\in(0,T]\times\OO$, the exact solution $u(t,x)$ to \eqref{CH} admits a density.
As a numerical counterpart,
Theorem \ref{Density-n} implies the existence of density of $u^n(t,kh)=U_k(t)$
for every $k\in\mathbb{Z}_{n-1}$ and $t\in(0,T]$. 
\begin{theorem}\label{Density-n}
Let $\sigma$ be continuously differentiable and $\sigma(x)\neq0$  for any $x\in\mathbb{R}$. Then
for any $t\in(0,T]$, the law of $U(t)$ is absolutely continuous with respect to the Lebesgue measure on $\mathbb{R}^{n-1}$.
\end{theorem}
\begin{proof}
 Introduce $Z(t):=(-A_n)^{-\frac12}U(t)$ for $t\in[0,T]$. By \eqref{NCH}, 
\begin{align*}
\ud Z(t)+\!A_n^2Z(t)\ud t=\!-(-A_n)^{\frac12}F_n((-A_n)^{\frac12}Z(t))\ud t+\!\sqrt{\frac{n}{\pi}}(-A_n)^{-\frac12}\Sigma_n((-A_n)^{\frac12}Z(t))\ud \beta_t.
\end{align*}
 For $t\in(0,T]$, to prove the absolute continuity of the law of $U(t)$, it suffices to show that the law of $Z(t)$ is  absolutely continuous. 
 Next, we apply \cite[Theorem 5.2]{IRS19} to show  the absolute continuity of the law of $Z(t)$,  where the following conditions ($i$) and ($ii$) are required.
\begin{itemize}
\item[($i$)]  Assumption 3.1 of \cite{IRS19}, which mainly contains properties $(1)$-$(3)$ below.

\begin{itemize}
\item[$(1)$]  $\mathbb R^{n-1}\ni x\mapsto\tilde b(x):=-(-A_n)^{\frac12}F_n((-A_n)^{\frac12}x)$ satisfies the one-sided Lipschitz condition; 

\item[$(2)$]  $\mathbb R^{n-1}\ni x\mapsto\tilde\sigma(x):=\sqrt{n/\pi}(-A_n)^{-\frac12}\Sigma_n((-A_n)^{\frac12}x)$ satisfies the global Lipschitz condition; 

\item[$(3)$]  both $\tilde b$ and $\tilde \sigma$ are continuously differentiable.
\end{itemize}
\item[($ii$)]  For any $\mathbb{R}^{n-1}\ni z\neq 0$ and $s\le t$,
$$z^\top \Sigma_n(U(s))\Sigma_n(U(s))^\top z>\lambda(s,t)\vert z\vert ^2\ge 0,\quad a.s.,$$
for some function $\lambda: \mathbb{R}_+^2\rightarrow \mathbb{R}$ with $\int_0^t\lambda(s,t)\ud s>0$.
\end{itemize}
We first prove property (1). In view of \eqref{fbfa}, we have that for any $x,y\in\mathbb R^{n-1}$,
\begin{equation*}
-\langle x-y,F_n(x)-F_n(y)\rangle=\sum_{k=1}^{n-1}(x_k-y_k)(f(y_k)-f(x_k))\le \Vert x-y\Vert ^2.
\end{equation*}
Hence, it follows from the symmetry of $(-A_n)^{\frac12}$ that
\begin{align*}
 &\quad\ \langle x-y,-(-A_n)^{\frac12}F_n((-A_n)^{\frac12}x)+(-A_n)^{\frac12}F_n((-A_n)^{\frac12}y)\rangle\\
 &= -\langle (-A_n)^{\frac12}x-(-A_n)^{\frac12}y,F_n((-A_n)^{\frac12}x)-F_n((-A_n)^{\frac12}y)\rangle\\
 &\le \Vert (-A_n)^{\frac12}(x-y)\Vert ^2\le (n-1)^2\Vert x-y\Vert ^2,
 \end{align*}
 which yields the desired property (1).
Similarly, using  the Lipschitz continuity of $\sigma$ and the continuous differentiability of $f$ and $\sigma$, one could see that properties (2) and (3) are fulfilled. Besides, since $\sigma(\cdot)\neq 0$, the square matrix $\Sigma_n(U(s))\Sigma_n(U(s))^\top$ has positive minimum eigenvalue, which is denoted by $\lambda_{min}(s,\omega)$. Then
 the property ($ii$) follows immediately by choosing $\lambda(s,t)=
\frac{1}{2}\lambda_{min}(s,\omega)>0$ for $s\le t$. 
\end{proof}

Making use of the variation of constants formula, we obtain from \eqref{NCH} that
\begin{align}\label{Unt}
U(t)&=\exp(-A_n^2t)U(0)+\int_0^tA_n\exp(-A_n^2(t-s))F_n(U(s))\mathrm{d} s\\\notag
&\quad+\sqrt{\frac{n}{\pi}}\int_0^t\exp(-A_n^2(t-s))\Sigma_n(U(s))\mathrm{d} \beta_s,\quad t\in[0,T].
\end{align}
For $j\in\mathbb{Z}_{n-1}$, $e_j=(e_j(1),\ldots,e_j(n-1))^\top$ given by
\begin{equation}\label{ejk}
e_j(k)
=\sqrt{\pi/n}\phi_j(kh)=\sqrt{2/n}\sin(jkh),\quad k\in\mathbb{Z}_{n-1},
\end{equation} 
is an eigenvector of $A_n$  associated with the eigenvalue $\lambda_{j,n}=-j^2c_{j,n}$, where $c_{j,n}:=\sin^2(\frac{j}{2n}\pi)/(\frac{j}{2n}\pi)^2$ satisfies $\frac{4}{\pi^2}\le c_{j,n}\le 1.$ The vectors $\{e_i\}_{i=1}^{n-1}$  form an orthonormal basis of $\mathbb{R}^{n-1}$ (see e.g., \cite{GI98}). In particular,
\begin{align*}
\langle e_i,e_j\rangle=\frac{\pi}{n}\sum_{k=1}^{n-1}\phi_i(kh)\phi_j(kh)=\int_{\mathcal{O}}\phi_i(\kappa_n(y))\phi_j(\kappa_n(y))\mathrm{d} y=\delta_{ij}.
\end{align*} 
It is verified that
$1-\frac{\sin a}{a}\le \frac{1}{6}a^2$ for all $a\in[0,\frac{\pi}{2})$, which indicates that for $j\in\mathbb{Z}_{n-1}$,
\begin{align}\label{cjn}
0\le 1-c_{j,n}=\Big(1+\sin(\frac{j}{2n}\pi)/(\frac{j}{2n}\pi)\Big)\Big(1-\sin(\frac{j}{2n}\pi)/(\frac{j}{2n}\pi)\Big)\le \frac{\pi^2j^2}{12n^2}.
\end{align}

Introduce the discrete kernel
$$G^n_t(x,y)=\sum_{j=1}^{n-1}\exp(-\lambda_{j,n}^2t)\phi_{j,n}(x)\phi_j(\kappa_n(y)),$$
where $\phi_{j,n}:=\Pi_n(\phi_j)$. Define the discrete Dirichlet Laplacian $\Delta_n$ by $\Delta_n w(y)=0$ for $y\in[0,h)$, and 
\begin{align}\label{Deltaw}
\Delta_n w(y)=\frac{n^2}{\pi^2}\left(w\big(\kappa_n(y)+\frac{\pi}{n}\big)-2w\big(\kappa_n(y)\big)+w\big(\kappa_n(y)-\frac{\pi}{n}\big)\right),
\end{align}
for $y\in[h,\pi)$, where $w:\mathcal{O}\rightarrow \mathbb{R}$ with $w(0)=w(\pi)=0$.
Since $\Delta_n \phi_j(\kappa_n(y))=\lambda_{j,n}\phi_j(\kappa_n(y))$, it follows that
\begin{align*}
\Delta_n G^n_{t}(x,y)=\sum_{j=1}^{n-1}\lambda_{j,n}\exp(-\lambda_{j,n}^2t)\phi_{j,n}(x)\phi_j(\kappa_n(y)).
\end{align*}
Similar to \cite[Section 2]{GI98}, based on \eqref{Unt}, the diagonalization of the matrix $A_n$, \eqref{ejk} and $u^n(t,kh)=\sum_{j=1}^{n-1}\langle U(t),e_j\rangle e_j(k)$, one  has 
\begin{align}\label{unR0}\notag
u^n(t,x)=&\int_{\mathcal O}G^n_t(x,y)u_0(\kappa_n(y))\mathrm{d} y+\int_0^t\int_{\mathcal O}\Delta_nG^n_{t-s}(x,y)f(u^n(s,\kappa_n(y)))\mathrm{d} y\mathrm{d} s\\
&+\int_0^t\int_{\mathcal O}G^n_{t-s}(x,y)\sigma(u^n(s,\kappa_n(y)))W(\mathrm{d} s,\mathrm{d} y),\quad(t,x)\in[0,T]\times\mathcal{O}.
\end{align}

We have
the following regularity 
estimate of $G^n$, whose proof is analogous to that of Lemma \ref{Greg} and thus is omitted.
\begin{lemma}\label{Gn-regularity}
Let $\alpha\in(0,1)$. Then for any $x,y\in\mathcal O$ and $0\le s<t\le T$,
\begin{gather*}
\int_0^s\int_{\mathcal O}\vert G^n_{t-r}(x,z)-G^n_{s-r}(y,z)\vert ^2\ud z\ud r\le C_\alpha(\vert x-y\vert ^2+\vert t-s\vert ^{\frac{3}{4}\alpha}),\\
\int_s^t\int_{\mathcal O}\vert G^n_{t-r}(x,z)\vert ^2\ud z\ud r\le C\vert t-s\vert ^{\frac{3}{4}},\\
\int_0^s\int_{\mathcal O}\vert \Delta_nG^n_{t-r}(x,z)-\Delta_nG^n_{s-r}(y,z)\vert \mathrm{d} z\mathrm{d} r\le C_\alpha(\vert x-y\vert +\vert t-s\vert ^{\frac{3\alpha}{8}}),\\
\int_s^t\int_{\mathcal O}\vert \Delta_n G^n_{t-r}(x,z)\vert \mathrm{d} z\mathrm{d} r\le C_\alpha\vert t-s\vert ^{\frac{3\alpha}{8}}.
\end{gather*}
\end{lemma}

\subsection{Fully discrete FDM}
We utilize the implicit Euler method with the time stepsize $\tau=T/m$ to further discretize \eqref{NCH} and then obtain the  fully discrete numerical method \eqref{eq:Ui} with $U^0=U(0)$. In order to illustrate that \eqref{eq:Ui} is uniquely solvable,
 by introducing $Z^{i}:=(-A_n)^{-\frac12}U^i$, one can see that \eqref{eq:Ui} is equivalent to
\begin{align}\label{eq:Z}
Z^{i+1}-Z^{i}=\tau A_n^2Z^{i+1} + \tau\tilde{b}(Z^{i+1})+\tilde\sigma(Z^i)( \beta_{t_{i+1}}-\beta_{t_i}),
\end{align}
where $\tilde{b}$ is one-sided Lipschitz continuous and $\tilde \sigma$ is globally Lipschitz continuous (see the proof of Theorem \ref{Density-n} for the expressions of $\tilde{b}$ and $\tilde \sigma$). By means of \cite[Lemma 3.1]{MS13}, we obtain the unique solvability of \eqref{eq:Z}, which implies that \eqref{eq:Ui} is uniquely solvable. 
In virtue of \eqref{eq:Ui}, we have
\begin{align*}
U^{i}
&=(I+\tau A_n^2)^{-i}U^0+\tau\sum_{l=0}^{i-1}(I+\tau A_n^2)^{-(i-l)}A_n F_n(U^{l+1})\\
&\quad\ +\sum_{l=0}^{i-1}\sqrt{n/\pi}(I+\tau A_n^2)^{-(i-l)}\Sigma_n(U^l)(\beta_{t_{l+1}}-\beta_{t_{l}}),\quad i\in \mathbb{Z}_{m}^0,
\end{align*}
where $\sum_{l=0}^{-1}$ is viewed as $0$ by convention.
Denote $\eta_{\tau}(t):=\tau\lfloor\frac{t}{\tau}\rfloor$, i.e., $\eta_{\tau}(t)=t_i$ for $t\in[t_i,t_{i+1})$.
Notice that for any $i\in\mathbb{Z}_m^0$, $U^i=(U^i_1,\ldots,U^i_{n-1})^\top$ is the temporal numerical approximation of the spatial semi-discrete numerical solution $U(t_i)=(u^n(t_i,h),\ldots,u^n(t_i,(n-1)h))^\top$. Hence we denote
$u^{n,\tau}(t_i,kh):=U^i_k$ for $k\in\mathbb{Z}_{n-1}$ and $u^{n,\tau}(t_i,kh):=0$ for $k\in\{0,n\}$,
in view of the DBCs.
In addition, we define 
 $u^{n,\tau}(t_i,x):=\Pi_n(u^{n,\tau}(t_i,\cdot))(x)$ as
 the fully discrete numerical solution of $u(t_i,x)$ for every $i\in\mathbb{Z}_m^0$ and $x\in\mathcal{O}$. 
Introduce the fully discrete Green function
\begin{align*}
G^{n,\tau}_t(x,y):=\sum_{j=1}^{n-1}(1+\tau\lambda_{j,n}^2)^{-\lfloor\frac{t}{\tau}\rfloor}\phi_{j,n}(x)\phi_j(\kappa_n(y)),
\end{align*}
and then by \eqref{Deltaw},
\begin{align*}
\Delta_nG^{n,\tau}_t(x,y)=\sum_{j=1}^{n-1}\lambda_{j,n}(1+\tau\lambda_{j,n}^2)^{-\lfloor\frac{t}{\tau}\rfloor}\phi_{j,n}(x)\phi_j(\kappa_n(y)).
\end{align*}
Analogously to \eqref{unR0}, one has that for $i\in\mathbb{Z}_m^0$ and $x\in\mathcal{O}$,
\begin{align}\label{eq:untau}
u^{n,\tau}(t_i,x)&=\int_{\mathcal O}G^{n,\tau}_{t_i}(x,y)u_0(\kappa_n(y))\mathrm{d} y\\\notag
&\quad +\int_0^{t_i}\int_{\mathcal O}\Delta_nG^{n,\tau}_{t_i-s+\tau}(x,y)f(u^{n,\tau}(\eta_\tau(s)+\tau,\kappa_n(y)))\mathrm{d} y\mathrm{d} s\\\notag
&\quad +\int_0^{t_i}\int_{\mathcal O}G^{n,\tau}_{t_i-s+\tau}(x,y)\sigma(u^{n,\tau}(\eta_\tau(s),\kappa_n(y)))W(\mathrm{d} s,\mathrm{d} y).
\end{align}
By the polygonal interpolation in time,  we define $\{u^{n,\tau}(t,x)\}_{(t,x)\in[0,T]\times\OO}$ by 
\begin{gather*}
	u^{n,\tau}(t,x):=u^{n,\tau}(t_{i},x)+\frac{t-t_{i}}{\tau}\left(u^{n,\tau}(t_{i+1},x)-u^{n,\tau}(t_{i},x)\right),~ t\in[t_{i},t_{i+1}],i\in\mathbb{Z}_{m-1}^0.
\end{gather*}

For $R\ge1$, let $K_R:\mathbb{R}\rightarrow\mathbb{R}$ be an even smooth cut-off function satisfying
\begin{align}\label{KR}
K_R(x)=1,\quad \text{if}~\vert x\vert <R;\qquad K_R(x)=0,\quad \text{if}~\vert x\vert \ge R+1,
\end{align}
and $\vert K_R\vert \le 1$, $\vert K_R^\prime\vert \le 2$. We are now ready to present the existence of the density of the fully discrete numerical solution.

\begin{theorem}\label{Density-n-dis}
		Let $\sigma$ be continuously differentiable and $\sigma(x)\neq0$  for any $x\in\mathbb{R}$. Then for sufficiently small $\tau>0$,
		the law of $\{U^i\}_{i\in\mathbb{Z}_m}$ is absolutely continuous with respect to the Lebesgue measure on $\mathbb{R}^{n-1}$.
	\end{theorem}
	\begin{proof}
		As in the proof of Theorem \ref{Density-n}, we only need to show that for any $i\in \mathbb{Z}_m$, the law of $Z^i=(-A_n)^{-\frac12}U^i$ is absolutely continuous. 
		For every $R\ge1$, define $\{Z_R^{i}\}_{i\in\mathbb{Z}_m}$ recursively by
		\begin{align*}
			Z_R^{i}-Z_R^{i-1}=\tau A_n^2Z_R^{i} + \tau \tilde{b}_R(Z_R^{i})+\tilde\sigma(Z_R^{i-1})( \beta_{t_{i}}-\beta_{t_{i-1}}),\quad i\in\mathbb{Z}_m,
		\end{align*}
		and $Z_R^{0}=Z^0$. Here 
	 $\tilde b_R(x)=\tilde b(x)K_R(\|x\|)$ is globally Lipschitz continuous. 
		
		Fix $i\in\mathbb{Z}_m$. Then  $Z_R^{i}=Z^{i}$ on the set $\{\omega\in\Omega:\sup_{i\in\mathbb{Z}_m^0}\|Z^{i}\|\le R\}$ whose probability converges to $1$ as $R\to\infty$. By means of the globally Lipschitz continuity of $\tilde b_R$ and $\tilde\sigma$, one can prove that each component of $Z_R^{i}$ belongs to $\mathbb{D}^{1,2}$.
		Together with \cite[Theorem 2.1.2]{DN06}, once we prove that the Malliavin covariance matrix $\gamma_{i}:=\int_0^{T} D_rZ^{i}(D_rZ^{i})^\top \ud r$
		of $Z^{i}$ is invertible a.s., it will follow that the law of $Z^i$ is absolutely continuous (see Appendix \ref{secA1} for the more details about the Malliavin derivative $D$ and the Malliavin--Sobolev space $\mathbb{D}^{1,2}$).
		Taking the Malliavin derivative on both sides of \eqref{eq:Z}, it holds that
		for a.s.\ $r\in[t_{i-1},t_i)$,
		\begin{align*}
			D_rZ^{i}= \tau(A_n^2+\nabla\tilde{b}(Z^{i}))D_rZ^{i}+\tilde\sigma(Z^{i-1}).
		\end{align*}
		By the one-sided Lipschitz continuity of $\tilde b:\mathbb{R}^{n-1}\to\mathbb{R}^{n-1}$, there exists some constant $K_n$ depending on $n$ such that $y^\top(A_n^2+\nabla\tilde{b}(Z^{i}))y\le K_n\|y\|^2.$
		Hence for any $\tau\in(0,1/K_n)$,
		\begin{align*}
			\|y\|^2-\tau y^\top(A_n^2+\nabla\tilde{b}(Z^{i}))y\ge (1-\tau K_n)\|y\|^2\quad \forall~y\in\mathbb{R}^{n-1}.
		\end{align*}
		This implies that the matrix $I-\tau(A_n^2+\nabla\tilde{b}(Z^{i}))$ is invertible for sufficiently small $\tau>0$. Hence it follows from the invertible of $\tilde \sigma$ that  for any $y\in \mathbb{R}^{n-1}$ with $\|y\|\neq0$,
		$$y^\top \gamma_{i} y\ge \int_{t_{i-1}}^{t_{i}}y^\top D_rZ^{i}(D_rZ^{i})^\top y\ud r>0,$$
		which means that the Malliavin covariance matrix $\gamma_{i}$ of $Z^{i}$ is invertible a.s. 
	\end{proof}

\section{Discrete $H^1$-regularity}\label{S4}
We introduce the discrete $L^2$-inner product and the discrete $L^p$-norm ($1\le p\le\infty$), respectively, as
\begin{align*}
\langle a,b\rangle_{l^2_n}=\frac{\pi}{n}\sum_{i=1}^{n-1}a_ib_i,
\qquad\Vert a\Vert _{l^p_n}=\begin{cases}\Big(\frac{\pi}{n}\sum\limits_{i=1}^{n-1}\vert a_i\vert ^p\Big)^{\frac{1}{p}},\quad&1\le p<\infty,\\
\sup\limits_{1\le j\le n-1}\vert a_i\vert ,\quad&p=\infty,
\end{cases}
\end{align*}
for vectors $a=(a_1,\ldots,a_{n-1})^\top$ and $b=(b_1,\ldots,b_{n-1})^\top$.

This section presents the discrete $H^1$-regularity of the  numerical solutions, which is crucial for the strong convergence analysis of the numerical solutions in Sections \ref{S5}--\ref{S4-5}. We begin with the discrete versions of embedding and interpolation theorems.

\begin{lemma}\label{lem:intepolation}
Let $2\le p\le\infty$, $n\ge2$, and $t>0$. Then for any $a\in l_n^\infty$,
\begin{align}\label{l2H1}
\Vert a\Vert _{l^\infty_n}&\le \sqrt{\pi}\Vert (-A_n)^{\frac{1}{2}}a\Vert _{l^2_n},\\\label{l6h2}
\Vert a\Vert _{l_n^6}&\le C\Vert A_na\Vert _{l^2_n}^{\frac{1}{6}}\Vert a\Vert ^{\frac{5}{6}}_{l_n^2},\\\label{interpolation}
\Vert e^{-A_n^2t}a\Vert _{l_n^p}&\le Ct^{-\frac{1}{4}(\frac{1}{2}-\frac{1}{p})}\Vert a\Vert _{l_n^2},
\end{align}
where $C>0$ is a constant independent of $a$, $n$ and $t>0$.
\end{lemma}
\begin{proof}

Let $a=(a_1,\ldots,a_{n-1})^\top$ and $a_0=a_n=0.$

($i$)~ It follows from the definition of $A_n$ that
\begin{align}\label{n2B}
\Vert (-A_n)^{\frac{1}{2}}a\Vert ^2_{l^2_n}=\langle -A_na,a\rangle_{l^2_n}=\frac{n}{\pi}\sum_{j=1}^{n}\vert a_{j}-a_{j-1}\vert ^2.
\end{align}
Hence, by the triangle and Cauchy--Schwarz inequalities, for  $k\in\mathbb{Z}_{n-1},$
\begin{align*}
\vert a_k\vert \le\sum_{j=1}^{k}\vert a_j-a_{j-1}\vert \le\Big(\sum_{j=1}^{k}\vert a_j-a_{j-1}\vert ^2\frac{n}{\pi}\Big)^{\frac{1}{2}}\Big(\sum_{j=1}^{k}\frac{\pi}{n}\Big)^{\frac{1}{2}}\le\sqrt{\pi}\Vert (-A_n)^{\frac{1}{2}}a\Vert _{l^2_n}.
\end{align*}

($ii$) By the Cauchy--Schwarz inequality and \eqref{n2B}, 
\begin{align*}
\vert a_k^2\vert &\le\sum_{j=1}^{k}\vert a_j^2-a_{j-1}^2\vert \le\Big(\frac{n}{\pi}\sum_{j=1}^{k}\vert a_j-a_{j-1}\vert ^2\Big)^{\frac{1}{2}}\Big(\frac{\pi}{n}\sum_{j=1}^{k}\vert a_j+a_{j-1}\vert ^2\Big)^{\frac{1}{2}}\\
&\le\Vert (-A_n)^{\frac{1}{2}}a\Vert _{l^2_n}\Big(\frac{4\pi}{n}\sum_{j=1}^{n-1}\vert a_j\vert ^2\Big)^{\frac{1}{2}}=2\Vert (-A_n)^{\frac{1}{2}}a\Vert _{l^2_n}\Vert a\Vert _{l_n^2}\quad\forall~ k\in\mathbb{Z}_{n-1},
\end{align*}
from which we deduce $\Vert a\Vert _{l_n^\infty}\le C\Vert (-A_n)^{\frac{1}{2}}a\Vert ^{\frac{1}{2}}_{l^2_n}\Vert a\Vert ^{\frac{1}{2}}_{l_n^2}$. This gives
\begin{align*}
\Vert a\Vert _{l_n^6}^6\le\frac{\pi}{n}\sum_{j=1}^{n-1}\vert a_j\vert ^2\Vert a\Vert _{l_n^\infty}^4\le C\Vert (-A_n)^{\frac{1}{2}}a\Vert _{l^2_n}^2\Vert a\Vert ^4_{l_n^2},
\end{align*}
which together with $\Vert (-A_n)^{\frac{1}{2}}a\Vert _{l^2_n}^2=\langle - A_na,a\rangle_{l^2_n}\le\Vert A_na\Vert _{l^2_n}\Vert a\Vert _{l^2_n}$ yields \eqref{l6h2}.

(iii)~
Since $\lambda_{j,n}\le -\frac{4}{\pi^2}j^2$, $\sum_{j=1}^{n-1}e^{-2\lambda_{j,n}^2t}\le t^{-\frac{1}{4}}\int_0^\infty e^{-\frac{32}{\pi^4}z^4}\ud z=:C_0^2t^{-\frac{1}{4}},$
where $C_0^2:=\int_0^\infty \exp(-\frac{32}{\pi^4}z^4)\ud z<\infty$. By the Cauchy--Schwarz inequality and \eqref{ejk},
\begin{align*}
\Vert e^{-A_n^2t}a\Vert _{l_n^\infty}=&\sup_{1\le k\le n-1}\Big\vert \sum_{j=1}^{n-1}e^{-\lambda_{j,n}^2t}\langle a,e_j\rangle  e_j(k)\Big\vert \\
&\le\Big(\frac{\pi}{n}\sum_{j=1}^{n-1}\vert \langle a,e_j\rangle\vert ^2\Big)^{\frac{1}{2}}\Big(\frac{n}{\pi}\sum_{j=1}^{n-1}e^{-2\lambda_{j,n}^2t}\frac{2}{n}\Big)^{\frac{1}{2}}\le \Vert a\Vert _{l_n^2}C_0t^{-\frac{1}{8}},
\end{align*}
where in the last step, we have used the fact that $\{e_j\}_{j=1}^{n-1}$ forms an orthonormal basis of $\mathbb{R}^{n-1}$.
This proves \eqref{interpolation} for $p=\infty$. Besides,
the Parseval identity leads to
\begin{align*}
\Vert e^{-A_n^2t}a\Vert _{l_n^2}^2=\frac{\pi}{n}\Vert e^{-A_n^2t}a\Vert  ^2=\frac{\pi}{n}\sum_{j=1}^{n-1}e^{-2\lambda_{j,n}^2t}\vert \langle a,e_j\rangle \vert ^2\le\frac{\pi}{n}\sum_{j=1}^{n-1}\vert \langle a,e_j\rangle \vert ^2=\Vert a\Vert _{l_n^2}^2,
\end{align*}
which implies \eqref{interpolation} for $p=2$. By the Riesz--Thorin interpolation theorem (see e.g., \cite[Theorem 1.3.4]{GL14}), we obtain
\eqref{interpolation} for $p\in(2,\infty)$.
\end{proof}

\subsection{Spatial FDM}\label{S3.2}
Introduce $O(t):=\sqrt{\frac{n}{\pi}}\int_0^t\exp(-A_n^2(t-s))\Sigma_n(U(s))\ud \beta_s$ and $V(t):=U(t)-O(t)$ for $t\in[0,T]$.
 Then
$V$ solves
\begin{align}\label{vn}
\dot V(t)=-A_n^2V(t)+A_nF_n(U(t)),\qquad V(0)=U(0),
\end{align}
where $\cdot$ denotes the derivative with respect to time. In this part, we give the $H^1$-regularity estimate of the spatial semi-discrete numerical solution $U(t)$ by dealing with $O(t)$ and $V(t)$ separately.

By the elementary identity
\begin{align}\label{eq:convolution}
\int_s^t(t-r)^{\alpha-1}(r-s)^{-\alpha}\ud r=\frac{\pi}{\sin(\pi\alpha)}\quad\forall~0\le s\le r\le t,\quad\text{with}~\alpha\in(0,1),
\end{align} 
 we shall use the factorization method to write
\begin{align*}
(-A_n)^{\frac{1}{2}}O(t)=\frac{\sin(\pi\alpha)}{\pi}\int_0^t\exp(-A_n^2(t-r))(t-r)^{\alpha-1}Y(r)\ud r,
\end{align*}
where
$Y(r):=\sqrt{n/\pi}\int_0^r\exp(-A_n^2(r-s))(-A_n)^{\frac{1}{2}}(r-s)^{-\alpha}\Sigma_n(U(s))\ud \beta_s$. Then by $\|\exp(-A_n^2(t-r))\|_{2}\le 1$ and the H\"older inequality, 
for any $p>\frac{1}{\alpha}$,
\begin{align*}\|(-A_n)^{\frac{1}{2}}O(t)\|^p\le C(\alpha,T,p)\int_0^t\|Y(r)\|^p\ud r\quad\forall~t\in[0,T].
\end{align*}
Hereafter, $\Vert \cdot\Vert _{2}$ and $\Vert \cdot\Vert _{\mathrm F}$ denote the Euclidean and Frobenius norms of matrices, respectively.

Notice that  $\Vert \Sigma_n(U(s))e_l\Vert ^2=\sum_{k=1}^{n-1}\vert \sigma(U_k(s))e_l(k)\vert ^2\le C,$
thanks to the boundedness of $\sigma$
and  $\vert e_l(k)\vert \le \sqrt{2/n}$ for all $l,k\in\mathbb{Z}_{n-1}$. 
For any $\alpha_1>0$,
\begin{align}\label{ex<x}
e^{-x}\le C_{\alpha_1}x^{-\alpha_1}\quad\forall~ x >0.
\end{align}
Hence, by the symmetry of $\Sigma_n(U(s))$ and \eqref{ex<x} with $\alpha_1=\frac{3}{4}+\epsilon$, $0<\epsilon\ll 1$, 
\begin{align}\label{ASigmaUF}
&\quad\Vert (-A_n)^{\frac{1}{2}}\exp(-A_n^2(r-s))\Sigma_n(U(s))\Vert ^2_{\mathrm F}\\\notag
&=\sum_{k,l=1}^{n-1}\langle(-A_n)^{\frac{1}{2}}\exp(-A_n^2(r-s))\Sigma_n(U(s))e_k,e_l\rangle^2\\\notag
&=\sum_{l=1}^{n-1}(-\lambda_{l,n})\exp(-2\lambda^2_{l,n}(r-s))\Vert \Sigma_n(U(s))e_l\Vert ^2\\\notag
&\le C_\epsilon\sum_{l=1}^{n-1}(-\lambda_{l,n})^{-\frac{1}{2}-2\epsilon}(r-s)^{-\frac{3}{4}-\epsilon}\le C_\epsilon(r-s)^{-\frac{3}{4}-\epsilon}
\end{align}
for any $0\le s< r\le T$ since
 $-\lambda_{l,n}\ge\frac{4}{\pi^2} l^2$. 
As a consequence, by the Burkholder inequality and choosing $\alpha\in(0,\frac18-\frac\epsilon2)$, we derive that for any $p>\frac{1}{\alpha}$,
\begin{align*}
\E\left[\Vert Y(r)\Vert ^p_{l_n^2}\right]&=\E\bigg[\Big\Vert \int_0^r(-A_n)^{\frac{1}{2}}\exp(-A_n^2(r-s))(r-s)^{-\alpha}\Sigma_n(U(s))\ud \beta_s\Big\Vert ^p\bigg]\\
&\le C(p)\E\bigg[\Big\vert \int_0^r(r-s)^{-2\alpha}\Vert(-A_n)^{\frac{1}{2}}\exp(-A_n^2(r-s))\Sigma_n(U(s))\Vert ^2_{\mathrm F}\ud s\Big\vert ^{\frac{p}{2}}\bigg]\\
&\le C(p,T).
\end{align*}
Therefore,  we have that for any $p\ge1$,
\begin{align}\label{Onl2}
\E\bigg[\sup_{t\in[0,T]}\Vert (-A_n)^{\frac{1}{2}}O(t)\Vert ^p_{l_n^2}\bigg]\le C_{p,T}\int_0^T\E\left[\|Y(r)\|_{l_n^2}^p\right]\ud r \le C(T,p).
\end{align}

Taking \eqref{l2H1} into account, it further yields that for any $p\ge1$,
\begin{align}\label{linfty-w}
\mathbb E\bigg[\sup_{t\in[0,T]}\Vert O(t)\Vert _{l^\infty_n}^p\bigg]
\le C(p,T).
\end{align}
Taking advantage of the special form of $F_n$ and \eqref{linfty-w}, we are able to show that $V$ is bounded in $L^\infty(0,T;L^{p}(\Omega;l_n^2))$ for $p\ge2$. 

\begin{lemma} \label{uniforml2}
Let $u_0\in\mathcal{C}^1(\OO)$ and $q\ge1$. Then for any $t\in[0,T]$,
\begin{align*}
\mathbb E\left[\Vert V(t)\Vert ^{2q}_{l^2_n}\right]+\mathbb E\left[\Big\vert \int_0^t\Vert  A_nV(s)\Vert _{l^2_n}^2\ud s\Big\vert ^q\right]\le C(q,T).
\end{align*}
\end{lemma}

\begin{proof}
Due to \eqref{n2B} and $u_0\in\mathcal{C}^1(\OO)$, 
$$\Vert (-A_n)^{\frac{1}{2}}U(0)\Vert _{l_n^2}^2=\frac{n}{\pi}\sum_{j=1}^{n}\left\vert u_0(jh)-u_0((j-1)h)\right\vert ^2\le n\sum_{j=1}^n\frac{C}{n^2}\leq C,$$
which together with the symmetry of $A_n$ and $\frac{4}{\pi^2}\le\lambda_{min}(-A_n)\le\lambda_{max}(-A_n)\le (n-1)^2$ implies
\begin{equation}\label{u0-H1}
\Vert (-A_n)^{-\frac{1}{2}}V(0)\Vert _{l^2_n}\le C\Vert V(0)\Vert _{l^2_n}\le C\Vert (-A_n)^{\frac{1}{2}}V(0)\Vert _{l^2_n}\le C.
\end{equation}
Here, $\lambda_{min}(-A_n)$ and $\lambda_{max}(-A_n)$ denote the minimal and maximal eigenvalues of $-A_n$, respectively.
Notice that for $i\in\{1,2,3\}$ and $\delta\in(0,1)$,  $a^{4-i}b^i\le \delta a^{4}+C_\delta b^4$ for any $a,b\in\mathbb{R}$, from which it follows that for any $\epsilon\in(0,1)$,
\begin{align}\label{eq:FnUV}
&\quad\langle A_n F_n(U(t)),(-A_n)^{-1}V(t)\rangle_{l^2_n}=-\langle F_n(U(t)),V(t)\rangle_{l^2_n}\\\notag
&=-\frac{\pi}{n}\sum_{j=1}^{n-1}\left[(V_j(t)+O_j(t))^3-(V_j(t)+O_j(t))\right]V_j(t)\\\notag
&\le-(1-\epsilon)\Vert V(t)\Vert _{l_n^4}^4+C(\epsilon)(\Vert O(t)\Vert _{l_n^4}^4+1).
\end{align}
Taking the inner product $\langle\cdot,(-A_n)^{-1}V(t)\rangle_{l^2_n}$ on both sides of \eqref{vn} gives
\begin{align*}
&\frac{1}{2} \frac{\ud}{\ud t}\langle V(t),(-A_n)^{-1}V(t)\rangle_{l^2_n}=\langle A_nV(t),V(t)\rangle_{l^2_n}+\langle A_n F_n(U(t)),(-A_n)^{-1}V(t)\rangle_{l^2_n}\\\notag
&\le -\Vert (-A_n)^{\frac{1}{2}}V(t)\Vert ^2_{l_n^2}-(1-\epsilon)\Vert V(t)\Vert _{l_n^4}^4+C(\epsilon)(\Vert O(t)\Vert _{l_n^4}^4+1).
\end{align*}
Then integrating with respect to time and using  \eqref{u0-H1}, we get that for $t\in(0,T]$,
\begin{align}\label{vh-1}
&\quad\ \Vert (-A_n)^{-\frac{1}{2}}V(t)\Vert _{l^2_n}^2+\int_0^t2\Vert (-A_n)^{\frac{1}{2}}V(s)\Vert ^2_{l_n^2}\ud s+\int_0^t(2-2\epsilon)\Vert V(s)\Vert _{l_n^4}^4\ud s\\\notag
&\le\Vert (-A_n)^{-\frac{1}{2}}V(0)\Vert _{l^2_n}^2+C_\epsilon\int_0^t1+\Vert O(s)\Vert _{l_n^4}^4\ud s\le C_\epsilon+C_\epsilon\int_0^t\Vert O(s)\Vert _{l_n^4}^4\ud s.
\end{align}
By \eqref{vh-1} with $\epsilon=\frac12$, \eqref{linfty-w} and the H\"older inequality, we arrive at   
\begin{align}\label{vnh-1}
&\quad\ 
\mathbb E\left[\Big\vert \int_0^t\Vert (-A_n)^{\frac{1}{2}}V(s)\Vert ^2_{l_n^2}\ud s\Big\vert^p\right]+\mathbb E\left[\Big\vert \int_0^t\Vert V(s)\Vert _{l_n^4}^4\ud s\Big\vert ^p\right]\\\notag&\le C(p)+C(p,T)\mathbb E\int_0^t\Vert O(s)\Vert _{l_n^4}^{4p}\ud s\le C(p,T)\quad\forall~t\in(0,T].
\end{align}

Based on \eqref{vnh-1}, we proceed to estimate $\Vert V(t)\Vert ^2_{l^2_n}$.
Taking the inner product $\langle\cdot,V(t)\rangle_{l^2_n}$ on both sides of \eqref{vn}, it follows that
\begin{gather*}
\frac{1}{2}\frac{\ud}{\ud t}\Vert V(t)\Vert ^2_{l^2_n}+\Vert A_nV(t)\Vert _{l^2_n}^2=\langle F_n(U(t))-F_n(V(t)), A_nV(t)\rangle_{l^2_n}+\langle A_nF_n(V(t)), V(t)\rangle_{l^2_n}.
\end{gather*}
The inequality $ab\le \frac{\epsilon}{4}a^2+\frac{1}{\epsilon}b^2$ for $a,b\in\mathbb{R},$ and \eqref{fb-fa} give that for any $\epsilon\in(0,1)$,
\begin{align}\label{eq:FU-FV}\notag
\langle F_n(U(t))-F_n(V(t)), &A_nV(t)\rangle_{l^2_n}\le \frac{\epsilon}{4}\Vert A_nV(t)\Vert _{l^2_n}^2+\frac{1}{\epsilon}\Vert F_n(U(t))-F_n(V(t))\Vert _{l^2_n}^2\\
&\le \frac{\epsilon}{4}\Vert A_nV(t)\Vert _{l^2_n}^2+C(\epsilon)\left(1+\Vert V(t)\Vert _{l_n^4}^4+\Vert O(t)\Vert _{l_n^\infty}^4\right)\Vert O(t)\Vert _{l_n^\infty}^2.
\end{align}
By $a^3b\le \frac{3}{4}a^4+\frac{1}{4}b^4$ for $a,b\in\mathbb{R}$, one can verify $\langle x,A_n(x_1^3,\ldots,x_{n-1}^3)^\top\rangle \le 0$. Hence,
\begin{align*}
&\langle x,A_nF_n(x)\rangle 
=\langle x,A_n(x_1^3,\ldots,x_{n-1}^3)^\top\rangle -\langle x,A_nx\rangle 
\le -\langle x,A_nx\rangle,
\end{align*}
from which
we deduce that for $\epsilon\in(0,1)$,
\begin{align}\label{eq:FVV}\notag
\langle A_nF_n(V(t)), V(t)\rangle_{l^2_n}&\le\langle-A_nV(t), V(t)\rangle_{l_n^2}\le  \frac{\epsilon}{4}\Vert A_nV(t)\Vert _{l^2_n}^2+C(\epsilon)\Vert V(t)\Vert ^2_{l^2_n}\\
&\le  \frac{\epsilon}{4}\Vert  A_nV(t)\Vert _{l^2_n}^2+C(\epsilon)(\Vert V(t)\Vert ^4_{l^4_n}+1).
\end{align}
Combining the above estimates with \eqref{u0-H1} produces
\begin{align*}
&\quad\Vert V(t)\Vert ^2_{l^2_n}+(2-\epsilon)\int_0^t\Vert  A_nV(s)\Vert _{l^2_n}^2\ud s\\
&\le\Vert V(0)\Vert _{l^2_n}^2+ C(\epsilon)\int_0^t\left(1+\Vert V(s)\Vert _{l_n^4}^4+\Vert O(s)\Vert _{l_n^\infty}^4\right)\left(\Vert O(s)\Vert _{l_n^\infty}^2+1\right)\ud s\\
&\le  C+C(\epsilon)\Big(\sup_{t\in[0,T]}\Vert O(t)\Vert _{l_n^\infty}^2+1\Big)\int_0^t1+\Vert V(s)\Vert _{l_n^4}^4+\Vert O(s)\Vert _{l_n^\infty}^4\ud s.
\end{align*}
Taking $\epsilon=\frac{1}{2}$ in the above inequality, we obtain from the Young inequality, \eqref{linfty-w} with $p=8q$, and \eqref{vnh-1} with $p=2q$ that
\begin{align*}
&\quad\mathbb E\left[\Vert V(t)\Vert ^{2q}_{l^2_n}\right]+\mathbb E\bigg[\Big\vert \int_0^t\Vert A_nV(s)\Vert _{l^2_n}^2\ud s\Big\vert ^{q}\bigg]\\
&\le C+C\mathbb E\bigg[\Big(\sup_{t\in[0,T]}\Vert O(t)\Vert _{l_n^\infty}^{2q}+1\Big)\Big\vert \int_0^t1+\Vert V(s)\Vert _{l_n^4}^{4}+\Vert O(s)\Vert _{l_n^\infty}^{4}\ud s\Big\vert ^q\bigg]\\
&\le C+C\mathbb E\bigg[\sup_{t\in[0,T]}\Vert O(t)\Vert _{l_n^\infty}^{4q}+1\bigg]+C\mathbb E\left[
\Big\vert \int_0^t1+\Vert V(s)\Vert _{l_n^4}^{4}+\Vert O(s)\Vert _{l_n^\infty}^{4}\ud s\Big\vert ^{2q}\right]\le C,
\end{align*}
which completes the proof.
\end{proof}

Now we proceed to derive the discrete $H^1$-regularity estimate of the spatial semi-discrete numerical solution $U(t)$, which guarantees $U\in L^\infty(0,T;L^p(\Omega;l^\infty_n))$. 
\begin{proposition}\label{Unl2-pro}
Let  $u_0\in\mathcal{C}^1(\OO)$ and $p\ge1$. Then for any $t\in[0,T]$,
\begin{align}\label{Unl2}
\mathbb{E}\left[\Vert (-A_n)^{\frac{1}{2}}U(t)\Vert ^p_{l_n^2}\right]\le C(T,p).
\end{align} 
\end{proposition}
\begin{proof}
It follows from \eqref{l2H1}, \eqref{interpolation} with $p=2$ and \eqref{u0-H1} that
\begin{align*}
\Vert e^{-A_n^2t}V(0)\Vert _{l^6_n}&\le C\Vert e^{-A_n^2t}V(0)\Vert _{l^\infty_n}\le C\Vert (-A_n)^{\frac{1}{2}}e^{-A_n^2t}V(0)\Vert _{l^2_n}\le C.
\end{align*}
From the spectrum mapping theorem and the symmetry of $A_n$, we get 
\begin{equation}\label{smooth}
\Vert e^{-\frac{1}{2}A_n^2(t-s)}(-A_n)^{\gamma}\Vert _{2}=\max_{1\le j\le n-1}e^{-\frac{1}{2}\lambda_{j,n}^2(t-s)}(-\lambda_{j,n})^\gamma\le C(t-s)^{-\frac{\gamma}{2}}
\end{equation}
for any $\gamma>0$,
since $x\mapsto x^\gamma e^{-\frac{1}{2}x^2}$ is uniformly bounded on $[0,\infty)$.   
Applying the variation of constants formula to \eqref{vn}, we infer from \eqref{interpolation} with $p=6$ and \eqref{smooth} with $\gamma=1$ that
\begin{align}\label{VNt6}
\Vert V(t)\Vert _{l^6_n}&\le\Vert e^{-A_n^2t}V(0)\Vert _{l^6_n}+\int_0^t\Vert e^{-\frac{1}{2}A_n^2(t-s)}e^{-\frac{1}{2}A_n^2(t-s)}A_nF_n(U(s))\Vert _{l^6_n}\ud s\\\notag
&\le C+C\int_0^t(t-s)^{-\frac{7}{12}}\Vert F_n(U(s))\Vert _{l_n^2}\ud s.
\end{align}
We claim that for any $\theta\in[0,\frac{3}{4})$ and sufficiently large $p\ge 1$,
\begin{align}\label{theta-Fn}
\mathbb E\left[\Big\vert \int_0^t(t-s)^{-\theta}\Vert F_n(U(s))\Vert _{l_n^2}\ud s\Big\vert ^p\right]\le C(p,\theta,T).
\end{align}
Indeed, by the definition of $F_n$, \eqref{l6h2}, and the H\"older inequality,
\begin{align*}
&\quad\ \int_0^t(t-s)^{-\theta}\Vert F_n(U(s))\Vert _{l_n^2}\ud s\le C\int_0^t(t-s)^{-\theta}\left(1+\Vert V(s)\Vert ^3_{l_n^6}+\Vert O(s)\Vert ^3_{l_n^6}\right)\ud s\\
&\le C\int_0^t(t-s)^{-\theta}\Big(1+\Vert V(s)\Vert _{l_n^2}^{\frac{5}{2}}\Vert A_nV(s)\Vert _{l_n^2}^{\frac{1}{2}}+\Vert O(s)\Vert ^3_{l_n^\infty}\Big)\ud s\\
&\le C\Big(1+\sup_{t\in[0,T]}\Vert O(t)\Vert ^{3}_{l_n^\infty}\Big)+
\Big(\int_0^t(t-s)^{-\frac{4}{3}\theta}\Vert V(s)\Vert _{l_n^2}^{\frac{10}{3}}\ud s\Big)^{\frac{3}{4}}\Big(\int_0^T\Vert A_nV(s)\Vert _{l_n^2}^{2}\ud s\Big)^{\frac{1}{4}}\\
&\le C\Big(1+\sup_{t\in[0,T]}\Vert O(t)\Vert ^{3}_{l_n^\infty}\Big)+
C\Big(\int_0^T\Vert V(s)\Vert _{l_n^2}^{\frac{10p_2}{3}}\ud s\Big)^{\frac{3}{4p_2}}\Big(\int_0^T\Vert A_nV(s)\Vert _{l_n^2}^{2}\ud s\Big)^{\frac{1}{4}},
\end{align*}
where $p_2=\frac{p_1}{p_1-1}$ with $p_1=\frac{1}{2}+\frac{3}{8\theta}$.
Taking $p$th moment on both sides of the above inequality, then using the H\"older inequality, \eqref{linfty-w} and Lemma \ref{uniforml2}, we obtain \eqref{theta-Fn}. Furthermore, \eqref{theta-Fn} (with $\theta=\frac{7}{12}$) and \eqref{VNt6} give that for any $p\ge1$,
\begin{align}\label{Vn-l6}
\mathbb E\Big[ \Vert V(t)\Vert _{l^6_n}^p\Big]\le C(p,T)\quad\forall~t\in[0,T].
\end{align}

Due to \eqref{vn}, \eqref{u0-H1}, and \eqref{smooth} with $\gamma=\frac{3}{2}$, for any $t\in[0,T]$,
\begin{align}\label{AnVnt}\notag
\Vert (-A_n)^{\frac{1}{2}} V(t)\Vert _{l^2_n}&=\Vert (-A_n)^{\frac{1}{2}}  e^{-A_n^2t}V(0)\Vert _{l^2_n}+\int_0^t\Vert (-A_n)^{\frac{1}{2}}  e^{-A_n^2(t-s)}A_nF_n(U(s))\Vert _{l^2_n}\ud s\\\notag
&\le C+C\int_0^t(t-s)^{-\frac{3}{4}}\Vert F_n(U(s))\Vert _{l_n^2}\ud s\\
&\le C+C\int_0^t(t-s)^{-\frac{3}{4}}\left(1+\Vert V(s)\Vert ^3_{l_n^6}+\Vert O(s)\Vert ^3_{l_n^6}\right)\ud s.
\end{align}
Taking $p$th moment on both sides of \eqref{AnVnt}, and using \eqref{linfty-w}, \eqref{Vn-l6}, we have
\begin{align}\label{Vnl2}
\mathbb E\left[ \Vert (-A_n)^{\frac{1}{2}}V(t)\Vert _{l^2_n}^p\right]\le C(p,T)\quad\forall~t\in[0,T].
\end{align}
Due to $U=V+O$, \eqref{Unl2} follows from \eqref{Vnl2} and \eqref{Onl2}.
\end{proof}

Since $u^n(t,x)=\Pi_n(u^n(t,\cdot))(x)$ and $u^n(t,0)=u^n(t,\pi)=0$, it holds that
\begin{equation*}
\sup_{x\in\mathcal{O}}\vert u^n(t,x)\vert =\sup_{1\le j\le n-1}\vert U_j(t)\vert=\|U(t)\|_{l_n^\infty} \le\sqrt{\pi}\Vert (-A_n)^{\frac{1}{2}}U(t)\Vert _{l_n^2}, 
\end{equation*}
thanks to \eqref{l2H1}. As a consequence,
Proposition \ref{Unl2-pro} implies the moment boundedness of the spatial semi-discrete numerical solution $u^n(t,x)$, i.e., for any $p\ge1$,
\begin{align}\label{eq:u^nbound}
\sup_{t\in[0,T]}\E\bigg[\sup_{x\in\mathcal{O}}\vert u^n(t,x)\vert ^p\bigg]
\le C\sup_{t\in[0,T]}\E\bigg[\Vert (-A_n)^{\frac{1}{2}}U(t)\Vert _{l_n^2}^p\bigg]\le C(T,p).
\end{align}

\subsection{Fully discrete FDM}\label{S3.3}
Introduce 
$O^i:=\sqrt{\frac{n}{\pi}}\int_0^{t_i}(I+\tau A_n^2)^{-(i-\lfloor\frac{s}{\tau}\rfloor)}\Sigma_n(U^{\lfloor\frac{s}{\tau}\rfloor})\ud \beta_s$
and $V^i:=U^i-O^i$ for $i\in\mathbb{Z}_{m}^0$. Then for $i\in\mathbb{Z}_{m-1}^0$,
\begin{align}\label{eq:vn-dis}
V^{i+1}-V^i+\tau A_n^2 V^{i+1}=\tau A_n F_n(U^{i+1}),
\end{align}
with the initial value $V^0=U^0=U(0)$.
Similarly to the spatial semi-discrete case, this part gives the $H^1$-regularity estimate of the fully discrete numerical solution $U^i$ by estimating $O^i$ and $V^i$ separately.

We first use the factorization method to estimate $O^{i}$.
By \eqref{eq:convolution} with $t=t_{i}$, 
\begin{align}\notag
(-A_n)^{\frac12}O^{i}=\frac{\sin(\pi\alpha)}{\pi}\int_0^{t_i}(I+\tau A_n^2)^{-(i-\lfloor\frac{r}{\tau}\rfloor-1)}(t_{i}-r)^{\alpha-1}Y^\tau(r)\ud r,
\end{align}
where $Y^\tau(r):=\sqrt{n/\pi}\int_0^r(I+\tau A_n^2)^{-(\lfloor\frac{r}{\tau}\rfloor-\lfloor\frac{s}{\tau}\rfloor+1)}(-A_n)^{\frac12}(r-s)^{-\alpha}\Sigma_n(U^{\lfloor\frac{s}{\tau}\rfloor})\ud \beta_s.$
We have the smooth effect of $(I+\tau A_n^2)^{-\iota}$, namely for any $\gamma\in[0,2]$ and $\iota\ge1$,
\begin{align}\label{eq:smooth-dis}
 \| (-A_n)^{\gamma}(I+\tau A_n^2)^{-\iota}\|_2&=\max_{1\le k\le n-1}(-\lambda_{k,n})^{\gamma}(1+\tau\lambda_{k,n}^2)^{-\iota}\\\notag
&\le\max_{1\le k\le n-1}(-\lambda_{k,n})^{\gamma}(1+\tau\iota\lambda_{k,n}^2)^{-1}\le C\times(\iota \tau)^{-\frac{\gamma}{2}}.
\end{align}
Here we used Bernoulli's inequality
$(1+z)^{-\alpha}\le (1+\alpha z)^{-1}$ for $z>-1$ and $\alpha\ge1$.
Then proceeding as in \eqref{ASigmaUF}, it holds that for any $0<\epsilon\ll 1$ and $j\in\mathbb{Z}^0_{m}$,
\begin{align*}
\Vert (-A_n)^{\frac12}(I+\tau A_n^2)^{-\iota}\Sigma_n(U^j)\Vert ^2_{\mathrm F}\le C_\epsilon(\iota \tau)^{-\frac{3}{4}-\epsilon}\quad\forall~\iota\ge1.
\end{align*}
Furthermore, the Burkholder inequality gives that for any $\alpha\in(0,\frac18-\frac\epsilon2)$ and $p\ge1$,
\begin{align*}
\E\left[\Vert Y^\tau(r)\Vert ^p_{l_n^2}\right]&\le C(p)\Big\vert\int_{0}^{r}(r-s)^{-2\alpha}(\eta_\tau(r)-\eta_\tau(s)+\tau)^{-\frac{3}{4}-\epsilon}
\ud s\Big\vert ^{\frac{p}{2}}\le C(p,T),
\end{align*}
since $ \tau+
\eta_\tau(r)-\eta_\tau(s)\ge r-s$.
Hence similarly to \eqref{Onl2}, for any $p>\frac{1}{\alpha}$,
\begin{align}\label{eq:AnO2l2}
\E\Big[\sup_{i\in\mathbb{Z}_{m}}\|(-A_n)^{\frac12}O^{i}\|_{l_n^2}^p\Big]\le C_{p,T}\int_0^T\E\Big[\|Y^\tau(r)\|_{l_n^2}^p\Big]\ud r\le C,
\end{align}
which together with \eqref{l2H1} and the H\"older inequality further ensures
\begin{align}\label{eq:linfty-w-dis}
\mathbb E\Big[\sup_{i\in\mathbb{Z}_{m}}\Vert O^i\Vert _{l^\infty_n}^p\Big]
\le C(p,T)\quad\forall~p\ge1.
\end{align}

\begin{lemma} \label{uniforml2-dis}
Let $u_0\in\mathcal{C}^1(\OO)$ and $q\ge1$. Then for any $i\in\mathbb{Z}_{m}^0$,
\begin{align*}
\mathbb E\left[\Vert V^i\Vert ^{2q}_{l^2_n}\right]+\mathbb E\bigg[\Big\vert \sum_{j=0}^{i-1}\tau\Vert  A_nV^{j+1}\Vert _{l^2_n}^2\Big\vert ^q\bigg]\le C(q,T).
\end{align*}
\end{lemma}

\begin{proof}

In view of the relation $V^0=V(0)=U(0)$ and \eqref{u0-H1},
\begin{equation}\label{eq:u0-H1-dis}
\Vert (-A_n)^{-\frac{1}{2}}V^0\Vert _{l^2_n}\le C\Vert V^0\Vert _{l^2_n}\le C\Vert (-A_n)^{\frac{1}{2}}V^0\Vert _{l^2_n}\le C.
\end{equation}
Let us recall a fundamental identity 
\begin{align}\label{BE}\langle x-y,x\rangle=\frac12(\|x\|^2-\|y\|^2+\|x-y\|^2)\quad\forall~x,y\in\mathbb{R}^{d}, d\ge1.
\end{align}
Applying $\langle\cdot,(-A_n)^{-1}V^{i+1}\rangle_{l^2_n}$ on both sides of \eqref{eq:vn-dis}, it follows from \eqref{BE} and a similar estimate of \eqref{eq:FnUV} that
\begin{align*}
&\quad\ \frac12\|(-A_n)^{-\frac12}V^{i+1}\|_{l^2_n}^2-\frac12\|(-A_n)^{-\frac12}V^{i}\|_{l^2_n}^2+\frac12\|(-A_n)^{-\frac12}(V^{i+1}-V^i)\|_{l^2_n}^2\\
&=-\tau\|(-A_n)^{\frac12}V^{i+1}\|^2_{l^2_n}+\tau\langle A_nF_n(U^{i+1}),(-A_n)^{-1}V^{i+1}\rangle_{l^2_n}\\
&\le -\tau\|(-A_n)^{\frac12}V^{i+1}\|^2_{l^2_n}-\frac12\tau\Vert V^{i+1}\Vert _{l_n^4}^4+C\tau(\Vert O^{i+1}\Vert _{l_n^4}^4+1).
\end{align*}
Hence by \eqref{eq:u0-H1-dis}, 
\begin{align*}
\Vert (-A_n)^{-\frac{1}{2}}V^i\Vert _{l^2_n}^2+2\sum_{j=0}^{i-1}\tau\Vert (-A_n)^{\frac{1}{2}}V^{j+1}\Vert ^2_{l_n^2}+\sum_{j=0}^{i-1}\tau\Vert V^{j+1}\Vert _{l_n^4}^4\le C+C\tau\sum_{j=0}^{i-1}\Vert O^{j+1}\Vert _{l_n^4}^4.
\end{align*}
Furthermore, according to \eqref{eq:linfty-w-dis} and the H\"older inequality, we deduce that 
\begin{align}\label{eq:vnh-1-dis}
\E\bigg[\Big|\sum_{j=0}^{i-1}\tau\Vert V^{j+1}\Vert _{l_n^4}^4\Big|^p\bigg]\le C\quad\forall~p\ge1.
\end{align}

By \eqref{BE}, taking the inner product $\langle\cdot,V^{i+1}\rangle_{l^2_n}$ on both sides of \eqref{eq:vn-dis} gives
\begin{align*}
\frac12\|V^{i+1}\|_{l^2_n}^2-\frac12\|V^{i}\|_{l^2_n}^2+\frac12\|V^{i+1}-V^i\|_{l^2_n}^2=-\tau\|A_nV^{i+1}\|^2_{l^2_n}+\tau\langle F_n(U^{i+1}),A_nV^{i+1}\rangle_{l^2_n}.
\end{align*}
Similarly to \eqref{eq:FU-FV} and \eqref{eq:FVV} with $\epsilon=\frac12$,
\begin{align*}
 \langle F_n(U^{i+1}), A_nV^{i+1}\rangle_{l^2_n}&= \langle F_n(U^{i+1})-F_n(V^{i+1}),A_nV^{i+1}\rangle_{l^2_n}+\langle F_n(V^{i+1}),A_nV^{i+1}\rangle_{l^2_n}
\\
&\le \frac{1}{4}\Vert A_nV^{i+1}\Vert _{l^2_n}^2+C\left(1+\Vert V^{i+1}\Vert _{l_n^4}^4+\Vert O^{i+1}\Vert _{l_n^\infty}^4\right)\left(\Vert O^{i+1}\Vert _{l_n^\infty}^2+1\right).
\end{align*}
Collecting the above estimates yields
\begin{align*}
&\quad\ \Vert V^i\Vert ^2_{l^2_n}+\tau\sum_{j=0}^{i-1}\Vert  A_nV^{j+1}\Vert _{l^2_n}^2\\
&\le  \Vert V^0\Vert _{l^2_n}^2+C\Big(\sup_{i\in\mathbb{Z}_{m}}\Vert O^i\Vert _{l_n^\infty}^2+1\Big)\sum_{j=0}^{i-1}\tau(1+\Vert V^{j+1}\Vert _{l_n^4}^4+\Vert O^{j+1}\Vert _{l_n^\infty}^4).
\end{align*}
Taking $q$th moments on both sides of the above inequality and using \eqref{eq:u0-H1-dis}, \eqref{eq:linfty-w-dis} and \eqref{eq:vnh-1-dis} finally yield the desired result.
\end{proof}

\begin{proposition}\label{eq:Unl2-pro-dis}
Let  $u_0\in\mathcal{C}^1(\OO)$ and $p\ge1$. Then for any $i\in\mathbb{Z}_{m}^0$,
\begin{align*}
\mathbb{E}\left[\big\Vert (-A_n)^{\frac{1}{2}}U^i\big\Vert ^p_{l_n^2}\right]\le C(T,p).
\end{align*} 
\end{proposition}
\begin{proof}

According to \eqref{eq:vn-dis}, for any $i\in\mathbb{Z}_{m}^0$,
\begin{align*}
V^{i}
&=(I+\tau A_n^2)^{-i}V^0+\tau\sum_{j=0}^{i-1}(I+\tau A_n^2)^{-(i-j)}A_n F_n(U^{j+1}).
\end{align*}
By virtue of the fact that
$\sum_{j=1}^{n-1}(1+\tau\lambda_{j,n}^2)^{-2\iota}\le\sum_{j=1}^{n-1}(1+2\tau\iota\lambda_{j,n}^2)^{-1}\le C(\iota\tau)^{-\frac14}$, similarly to \eqref{interpolation}, it can be verified that for any $2\le p\le \infty$,
\begin{align}\label{eq:interpo-dis}\Vert (I+\tau A_n^2)^{-\iota}a\Vert _{l_n^p}\le C(p)(\tau\iota)^{-\frac{1}{4}(\frac{1}{2}-\frac{1}{p})}\Vert a\Vert _{l_n^2}\quad\forall~a\in l_n^\infty.
\end{align}

Applying \eqref{eq:smooth-dis}, \eqref{eq:interpo-dis}  and $\Vert (I+\tau A_n^2)^{-i}V^0\Vert _{l^6_n}\le \Vert V^0\Vert _{l^6_n}\le C
$, we arrive at
\begin{align*}
\Vert V^{i}\Vert _{l^6_n}&\le C+\tau\sum_{j=0}^{i-1}\Vert(I+\tau A_n^2)^{-\frac12(i-j)}A_n(I+\tau A_n^2)^{-\frac12(i-j)} F_n(U^{j+1})\Vert _{l^6_n}\\
&\le C+C\tau\sum_{j=0}^{i-1}t_{i-j}^{-\frac{1}{12}}\Vert A_n(I+\tau A_n^2)^{-\frac12(i-j)}\|_2\| F_n(U^{j+1})\Vert _{l^2_n}\\
&\le C+C\tau\sum_{j=0}^{i-1}t_{i-j}^{-\frac{7}{12}}\| F_n(U^{j+1})\Vert _{l^2_n}.
\end{align*}
Proceeding as in \eqref{theta-Fn}, one can show that for $\theta\in[0,\frac{3}{4})$ and sufficiently large $p\ge 1$,
\begin{align*}
\mathbb E\bigg[\Big\vert \tau\sum_{j=0}^{i-1}t_{i-j}^{-\theta}\| F_n(U^{j+1})\Vert _{l^2_n}\Big\vert ^p\bigg]\le C(p,\theta,T),
\end{align*}
based on Lemma \ref{uniforml2-dis}. In particular, the case $\theta=\frac{7}{12}$ yields that for any $p\ge1$,
\begin{align}\label{Vn-l6-dis}
\mathbb E\Big[ \Vert V^i\Vert _{l^6_n}^p\Big]\le C(p,T)\quad\forall~t\in[0,T].
\end{align}
Then analogously to \eqref{AnVnt}, by \eqref{eq:u0-H1-dis} and \eqref{eq:smooth-dis},
\begin{align}\label{eq:Anvi-dis}
\Vert (-A_n)^{\frac{1}{2}} V^{i}\Vert _{l^2_n}&\le\Vert (-A_n)^{\frac{1}{2}} (I+\tau A_n^2)^{-i}V^0\Vert _{l^2_n}\\\notag
&\quad+\tau\sum_{j=0}^{i-1}\Vert (-A_n)^{\frac{1}{2}} (I+\tau A_n^2)^{-(i-j)}A_n F_n(U^{j+1})\Vert _{l^2_n}\\\notag
&\le C+C\tau\sum_{j=0}^{i-1}(t_i-t_j)^{-\frac{3}{4}}(1+\Vert V^{j+1}\Vert _{l_n^6}^3+\Vert O^{j+1}\Vert _{l_n^6}^3).
\end{align}
Then one can use \eqref{Vn-l6-dis} and \eqref{eq:linfty-w-dis} to obtain that
$$\mathbb{E}[\Vert (-A_n)^{\frac{1}{2}}V^i\Vert ^p_{l_n^2}]\le C(T,p)\quad\forall~i\in\mathbb{Z}_{m}^0,$$
which along with  \eqref{eq:AnO2l2} and $U^i=V^i+O^i$ completes the proof.
\end{proof}
In a similar manner to \eqref{eq:u^nbound}, Proposition \ref{eq:u^nbound-dis} guarantees that for any $p\ge1$,
\begin{align}\label{eq:u^nbound-dis}
\sup_{i\in\mathbb{Z}_{m}^0}\E\Big[\sup_{x\in\mathcal{O}}\vert u^{n,\tau}(t_i,x)\vert ^p\Big]\le C(T,p).
\end{align}

\section{Strong convergence analysis (I)}\label{S5}
In this section, we study the strong convergence rate of the numerical solution of the spatial FDM.
For $n\ge2$, denote by $\mathbb U(t):=(u(t,h),\ldots,u(t,(n-1)h))^\top$ the exact solution to \eqref{CH} on grid points, where the explicit dependence of $\mathbb U(t)$ on $n$ is omitted. We introduce the following auxiliary process $\{\tilde U(t),t\in[0,T]\}$ by
\begin{align}\label{anx_u}
\ud \tilde U(t)+A_n^2\tilde U(t)\ud t= A_nF_n(\mathbb U(t))\ud t+\sqrt{n/\pi}\Sigma_n(\mathbb U(t))\ud \beta_t,\quad t\in(0,T]
\end{align}
with initial value $\tilde U(0)=U(0)$. Let $\tilde u^n=\{\tilde u^n(t,x),(t,x)\in[0,T]\times\OO\}$  satisfy 
\begin{align*}
\tilde u^n(t,x)=&\int_{\mathcal O}G^n_t(x,y)u_0(\kappa_n(y))\ud y+\int_0^t\int_{\mathcal O}\Delta_nG^n_{t-s}(x,y)f(u(s,\kappa_n(y)))\ud y\ud s\\
&+\int_0^t\int_{\mathcal O}G^n_{t-s}(x,y)\sigma(u(s,\kappa_n(y)))W(\ud s,\ud y).
\end{align*}
Then $\tilde U_k(t)=\tilde u^n(t,kh)$ for $k\in\mathbb{Z}_{n-1}$ and $t\in[0,T]$.

\subsection{Error estimate between $\tilde u^n$ and $u$}\label{S4.1}
This part deals with the error between the exact solution $u$ and the auxiliary process $\tilde u$, which will rely on the following  estimates of the discrete Green function.

\begin{lemma}\label{lem:Gn-G}
 There exists $C=C(T)$ such that for any $(t,x)\in(0,T]\times\mathcal{O}$,
\begin{align}\label{DeltaGnG}
\int_0^t\int_{\mathcal O}\vert \Delta_nG_{s}^n(x,y)-\Delta G_{s}(x,y)\vert \mathrm{d} y\mathrm{d} s&\le Cn^{-1},
\\\label{Gn-G}
\int_0^t\int_{\mathcal O}\vert G^{n}_s(x,y)-G_s(x,y)\vert ^2\mathrm{d} y\mathrm{d} s&\le Cn^{-2}.
\end{align}
\end{lemma}
\begin{proof}
Due to the H\"older inequality,
it suffices to prove that for $\mu\in\{0,1\}$,
\begin{equation*}
\int_0^{T}\left(\int_{\mathcal{O}}|(-\Delta_n)^\mu G_{s}^n(x,y)-(-\Delta)^\mu G_{s}(x,y)|^2\ud y\right)^{1-\frac{\mu}{2}}\ud s\le Cn^{-(2-\mu)}.
\end{equation*}
In the remainder of the proof, we always assume $\mu\in\{0,1\}$ and $0<\epsilon\ll 1$. Denote
\begin{align*}
M_{1,\mu}^{s,x,y}&:=\sum_{j=1}^{n-1}(-\lambda_{j,n})^\mu e^{-\lambda_{j,n}^2s}\phi_{j,n}(x)\left(\phi_j(\kappa_n(y))-\phi_j(y)\right),\\
 M_{2,\mu}^{s,x,y}&:=\sum_{j=1}^{n-1}\left((-\lambda_{j,n})^\mu e^{-\lambda_{j,n}^2s}\phi_{j,n}(x)-(-\lambda_j)^{\mu}e^{-\lambda_{j}^2s}\phi_{j}(x)\right)\phi_j(y),\\
M_{3,\mu}^{s,x,y}&:=\sum_{j=n}^{\infty}(-\lambda_j)^{\mu}e^{-\lambda_{j}^2s}\phi_{j}(x)\phi_j(y),\quad s\in[0,T], x,y\in\OO.
\end{align*}
Since $\{\phi_j\}_{j\ge1}$ is an orthonormal basis of $L^2(\mathcal{O})$, it holds that 
\begin{align}\label{eq.Delta-Delta}
&\quad\int_\mathcal{O}\vert (-\Delta_n)^\mu G_{s}^n(x,y)-(-\Delta)^\mu G_{s}(x,y)\vert^2\mathrm{d} y\le 
3 \sum_{i=1}^3\int_\mathcal{O}\vert M_{i,\mu}^{s,x,y}\vert ^2\mathrm{d} y.
\end{align}
By virtue of $|\phi_j(x)|\le 1$ and \eqref{ex<x} with $\alpha_1=\frac{2}{2-\mu}(1-\epsilon)>\mu+\frac{1}{4}$, 
\begin{align*}
&\quad\ \int_0^{T}\Big(\int_\mathcal{O}\vert M_{3,\mu}^{s,x,y}\vert ^2\mathrm{d} y\Big)^{1-\frac{\mu}{2}}\mathrm{d} s\le 
\int_0^{T}\Big(\sum_{j=n}^{\infty}j^{4\mu}e^{-2j^4 s}\Big)^{1-\frac{\mu}{2}}\mathrm{d} s\\
&\le C\int_0^T\Big(\sum_{j=n}^{\infty}j^{4\mu-4\alpha_1}s^{-\alpha_1}\Big)^{1-\frac{\mu}{2}}\mathrm{d} s\le C_\epsilon n^{-4[1-(\mu+\frac14)(1-\frac{\mu}{2})-\epsilon]}\le Cn^{-(2-\mu)},
\end{align*}
thanks to $4[1-(\mu+\frac14)(1-\frac{\mu}{2})]>2-\mu$ for all $\mu\in[0,1]$.

We recall the following inequality in the proof of \cite[Lemma 3.2]{GI98}: 
\begin{align}\label{eq.key}
\int_\mathcal{O}\vert w(y)-w(\kappa_n(y))\vert ^2\mathrm{d} y\le Cn^{-2}\int_\mathcal{O}\Big\vert \frac{\mathrm{d}}{\mathrm{d} y}w(y)\Big\vert ^2\mathrm{d} y,\quad\text{for}~w\in \mathcal{C}^1(\mathcal{O}).
\end{align}
By \eqref{eq.key} and \eqref{ex<x}, for any $\rho>0$,
\begin{align*}
\int_\mathcal{O}\vert M_{1,\mu}^{s,x,y}\vert ^2\mathrm{d} y&\le  Cn^{-2}\int_\mathcal{O}\vert \sum_{j=1}^{n-1}(-\lambda_{j,n})^{\mu}e^{-\lambda_{j,n}^2s}\phi_{j,n}(x)j\cos(jy)\vert ^2\mathrm{d} y\\
&\le Cn^{-2}\sum_{j=1}^{n-1}j^{4\mu+2}e^{-2\lambda_{j,n}^2s}|\phi_{j,n}(x)|^2\le Cn^{-2}\sum_{j=1}^{n-1}j^{4\mu+2-4\rho}s^{-\rho}.
\end{align*}
Observe that for any $\alpha\in(0,1]$, 
\begin{align}\label{1-ex}
1-e^{-x}\le C_{\alpha}x^{\alpha},\quad x\ge 0.
\end{align}
In view of \eqref{cjn}, 
$\vert \lambda_j-\lambda_{j,n}\vert \le Cj^4/n^2$.
Thus, $|(-\lambda_j)^{\mu}-(-\lambda_{j,n})^{\mu}|\le \mu(-\lambda_{j,n})^{\mu-1}|\lambda_j-\lambda_{j,n}|\le Cj^{2\mu+2}/n^2$
and 
$\lambda_j^2-\lambda_{j,n}^2=\vert \lambda_j-\lambda_{j,n}\vert \vert \lambda_j+\lambda_{j,n}\vert \le Cj^6/n^2$, which along with \eqref{1-ex} yields that for $\rho_1>0$,
\begin{align*}
 \vert e^{-\lambda_{j,n}^2s}-e^{-\lambda_{j}^2s}\vert 
 \le e^{-\lambda_{j,n}^2s} j^6n^{-2}s\le Cn^{-2}j^{6-2\rho_1}s^{1-\frac{\rho_1}{2}}.
\end{align*}
Besides, it can be verified that $\vert \phi_{j,n}(x)-\phi_j(x)\vert \le Cj/n.$ Therefore, for $\rho,\rho_1>0$,
\begin{align*}
&\quad\int_\mathcal{O}\vert M_{2,\mu}^{s,x,y}\vert ^2\mathrm{d} y=\sum_{j=1}^{n-1}\left\vert (-\lambda_{j,n})^{\mu}e^{-\lambda_{j,n}^2s}\phi_{j,n}(x)-(-\lambda_j)^{\mu}e^{-\lambda_{j}^2s}\phi_{j}(x)\right\vert ^2\\
&\le 3\sum_{j=1}^{n-1}\vert (-\lambda_{j,n})^\mu-(-\lambda_j)^\mu\vert ^2e^{-2\lambda_{j,n}^2s}+3\sum_{j=1}^{n-1}\lambda_j^{2\mu}\vert e^{-\lambda_{j,n}^2s}-e^{-\lambda_{j}^2s}\vert ^2\\
&\quad+3\sum_{j=1}^{n-1}\lambda_j^{2\mu}e^{-2\lambda_{j}^2s}\vert \phi_{j,n}(x)-\phi_j(x)\vert ^2\\
&\le Cn^{-4}\sum_{j=1}^{n-1}j^{4\mu+4-4\rho}s^{-\rho}+Cn^{-4}\sum_{j=1}^{n-1}{j^{12+4\mu-4\rho_1}}s^{2-\rho_1}+Cn^{-2}\sum_{j=1}^{n-1}j^{4\mu+2-4\rho}s^{-\rho}\\
&\le Cn^{-2}\sum_{j=1}^{n-1}j^{4\mu+2-4\rho}s^{-\rho},
\end{align*}
where in the last step we set $\rho_1=\rho+2$. 
By choosing 
$\rho\in(\mu+\frac34,\frac{2}{2-\mu})$, we have
\begin{align*}
\int_0^{T}\Big(\int_\mathcal{O}\vert M_{i,\mu}^{s,x,y}\vert ^2\mathrm{d} y\Big)^{1-\frac{\mu}{2}}\mathrm{d} s
\le C\int_0^T(n^{-2}s^{-\rho})^{1-\frac{\mu}{2}}\mathrm{d} s
\le Cn^{-(2-\mu)},\quad i=1,2.
\end{align*}
Finally inserting the estimates on $\{M_{i,\mu}^{s,x,y}\}_{i=1,2,3}$ into \eqref{eq.Delta-Delta} finishes the proof.
\end{proof}

By virtue of Lemmas \ref{Holder-exact} and \ref{lem:Gn-G}, we now estimate the error between $u$ and $\tilde u$.
\begin{proposition}\label{utilde-uh-1}
Let $u_0\in\mathcal{C}^3(\mathcal{O})$.
 Then for any $p\ge1$, there exists some constant $C=C(p,T)$ such that for any $(t,x)\in[0,T]\times\mathcal{O}$,
\begin{align*}
\Vert \tilde u^n(t,x)-u(t,x)\Vert _{L^p(\Omega)}\le Cn^{-1}.
\end{align*}
\end{proposition}
\begin{proof}
For fixed $(t,x)\in[0,T]\times\OO$, $u^n(t,x)-u(t,x)=\sum_{j=1}^5\mathcal{I}_j$, where
\begin{align*}
&\mathcal{I}_1:=\int_{\mathcal O}G^n_t(x,y)u_0(\kappa_n(y))\mathrm{d} y-\mathbb{G}_tu_0(x),\\
&\mathcal{I}_2:=\int_0^t\int_{\mathcal O}[G_{t-s}^n(x,y)-G_{t-s}(x,y)]\sigma(u(s,\kappa_n(y)))W(\mathrm{d} s,\mathrm{d} y),\\
&\mathcal{I}_3:=\int_0^t\int_{\mathcal O}G_{t-s}(x,y)[\sigma(u(s,\kappa_n(y)))-\sigma(u(s,y))]W(\mathrm{d} s,\mathrm{d} y),\\
&\mathcal{I}_4:=\int_0^t\int_{\mathcal O}[ \Delta_nG_{t-s}^n(x,y)-\Delta G_{t-s}(x,y)]  f(u(s,\kappa_n(y)))\mathrm{d} y\mathrm{d} s,\\
&\mathcal{I}_5:=\int_0^t\int_{\mathcal O} \Delta G_{t-s}(x,y) \left[ f(u(s,\kappa_n(y)))-f(u(s,y))\right]\mathrm{d} y\mathrm{d} s.
\end{align*}

 Following the proof of \cite[Lemma 2.3]{CC01}, we use the PDE satisfied by $G$ to write
$\mathbb{G}_tu_0(x)=u_0(x)-\int_0^t\int_\mathcal{O}\Delta G_r(x,z)u_0^{\prime\prime}(z)\mathrm{d} z\mathrm{d} r.$ 
As a numerical counterpart, 
\begin{align}\label{eq.u1tilde}\notag
&\quad\int_{\mathcal O}G^n_t(x,y)u_0(\kappa_n(y))\mathrm{d} y-\tilde u^n(0,x)=\int_{\mathcal O}\int_0^t\frac{\partial}{\partial r}G^n_r(x,z)u_0(\kappa_n(z))\mathrm{d} z\mathrm{d} r\\&=-\int_0^t\int_\mathcal{O} \Delta_n^2 G^n_r(x,z) u_0(\kappa_n(z))\mathrm{d} z\mathrm{d} r=
-\int_0^t\int_\mathcal{O} \Delta_n G^n_r(x,z)\Delta_n u_0(z)\mathrm{d} z\mathrm{d} r,
\end{align}
where $\tilde u^n(0,x)=\Pi_n(u_0)(x)$ and in the last step we have used the fact that $$\int_\mathcal{O}\Delta_n v(z)w(\kappa_n(z))\mathrm{d} z=\int_\mathcal{O} v(\kappa_n(z))\Delta_n w(z)\mathrm{d} z,$$
for $v,w:\mathcal{O}\rightarrow \mathbb{R}$ with $v=w=0$ on $\partial \mathcal{O}$.
 In particular, when $u_0\in\mathcal C^1(\mathcal{O})$, 
\begin{align}\label{un0Holder}
\vert \tilde u^n(0,x)-\tilde u^n(0,y)\vert \le C\vert x-y\vert ,\quad x,y\in \mathcal{O}.
\end{align}
By $u_0\in \mathcal{C}^3(\mathcal{O})$, \eqref{Deltaw} and the Taylor expansion, there exist $\theta_1,\theta_2\in(0,1)$ such that for $z\in[h,\pi)$,
$$ \vert u^{\prime\prime}_0(z)-\Delta_nu_0(z)\vert =\vert u^{\prime\prime}_0(z)-\frac{1}{2}u_0^{\prime\prime}(\kappa_n(z)+\theta_1\frac{\pi}{n})-\frac{1}{2}u_0^{\prime\prime}(\kappa_n(z)-\theta_2\frac{\pi}{n})\vert \le Cn^{-1},$$
and for $z\in[0,h)$, $\vert u^{\prime\prime}_0(z)-\Delta_nu_0(z)\vert =\vert u^{\prime\prime}_0(z)\vert =\vert u^{\prime\prime}_0(z)-u^{\prime\prime}_0(0)\vert \le Cn^{-1}.$
Therefore, using  \eqref{DGtxy} and \eqref{DeltaGnG}, a direct calculation gives 
\begin{align*}
|\mathcal{I}_1|&\le Cn^{-1}+\int_0^t\int_\mathcal{O}\vert \Delta_n G^n_r(x,z)-\Delta G_r(x,z)\vert \vert u^{\prime\prime}_0(z)\vert \mathrm{d} z\mathrm{d} r\\
&\quad+\int_0^t\int_\mathcal{O}\vert \Delta G_r(x,z)\vert \vert u^{\prime\prime}_0(z)-\Delta_nu_0(z)\vert \mathrm{d} z\mathrm{d} r\le Cn^{-1}.
\end{align*}

Then we apply the Burkholder inequality, the boundedness and Lipschitz continuity of $\sigma$, \eqref{Gn-G}, \eqref{Gtxy}, and Lemma \ref{Holder-exact}  to obtain 
\begin{align*}
\Vert\mathcal{I}_2+\mathcal{I}_3
\Vert^2_{L^p(\Omega)}&\le C\int_0^t\int_{\mathcal O}\vert G_{t-s}^n(x,y)-G_{t-s}(x,y)\vert ^2 \mathrm{d} y\mathrm{d} s\\
&\quad +C\int_0^t\int_{\mathcal O}G^2_{t-s}(x,y)\Vert u(s,\kappa_n(y))-u(s,y)\Vert ^2_{L^p(\Omega)}\mathrm{d} y\mathrm{d} s\le Cn^{-2}.
\end{align*}

It follows from $\vert f(x)\vert \le C(1+\vert x\vert ^3)$ and  \eqref{eq:bound-e} that $\Vert f(u(t,x))\Vert _{L^p(\Omega)}$ is uniformly bounded for all $(t,x)\in[0,T]\times\OO$. This together with \eqref{DeltaGnG} indicates
\begin{align*}
\|\mathcal{I}_4\|_{L^p(\Omega)}&\le C\int_0^t\int_{\mathcal O}\vert \Delta_nG_{t-s}^n(x,y)-\Delta G_{t-s}(x,y)\vert \Vert f(u(s,\kappa_n(y)))\Vert_{L^{p}(\Omega)}\mathrm{d} y\mathrm{d} s\le Cn^{-1}.
\end{align*}
In addition, making use of
\eqref{fb-fa}, the H\"older inequality, Lemma \ref{Holder-exact} and \eqref{eq:bound-e} yields 
\begin{align*}
&\quad\ \Vert f(u(s,\kappa_n(y)))-f(u(s,y))\Vert _{L^p(\Omega)}\\&\le C\Vert u(s,\kappa_n(y))-u(s,y)\Vert _{L^{3p}(\Omega)}\big(1+\Vert u(s,\kappa_n(y))\Vert ^2_{L^{3p}(\Omega)}+\Vert u(s,y)\Vert ^2_{L^{3p}(\Omega)}\big)
\le Cn^{-1}
\end{align*}
for all $(s,y)\in[0,T]\times\mathcal{O}$.
Hence taking advantage of \eqref{DGtxy}, we have 
$\|\mathcal{I}_5\|_{L^p(\Omega)}
\le Cn^{-1}.$
Gathering the above estimates finally completes the proof.

\end{proof}

The next result reveals that $u^n$ has the same H\"older continuity exponent as $u$.
\begin{lemma}\label{u2u3}
Let  $u_0\in\mathcal{C}^2(\OO)$. Then for any $\alpha\in(0,1)$ and $p\ge1$, there exists some constant $C=C(p,T,\alpha)$ such that for any $0\le s<t\le T$, $x,y\in\mathcal O$ and $n\ge2$,
\begin{align*}
&\Vert u^n(t,x)-u^n(s,y)\Vert _{L^p(\Omega)}\le C(\vert t-s\vert ^{\frac{3\alpha}{8}}+\vert x-y\vert ).
\end{align*}
\end{lemma}
\begin{proof}

Since $u_0\in \mathcal C^2(\mathcal{O})$, $\vert \Delta_n u_0(z)\vert \le C$ for $z\in\mathcal{O}$. Hence it follows from \eqref{eq.u1tilde}, \eqref{un0Holder} and Lemma \ref{Gn-regularity} that 
\begin{align*}
&\quad\ \left\vert \int_{\mathcal O}G^n_t(x,z)u_0(\kappa_n(z))\mathrm{d} z-\int_{\mathcal O}G^n_s(y,z)u_0(\kappa_n(z))\mathrm{d} z\right\vert \\
&\le C\vert x-y\vert +
\int_0^s\int_\mathcal{O} \vert \Delta_n G^n_r(x,z)-\Delta_n G^n_r(y,z)\vert \vert \Delta_n u_0(z)\vert \mathrm{d} z\mathrm{d} r\\
&\quad +\int_s^t\int_\mathcal{O}| \Delta_n G^n_r(x,z)||\Delta_n u_0(z)|\mathrm{d} z\mathrm{d} r\le C(\vert x-y\vert+\vert t-s\vert ^{\frac{3\alpha}{8}}).
\end{align*}
In virtue of Lemma \ref{Gn-regularity}, by further applying \eqref{eq:SD} with $S=\Delta_nG^n$ and \eqref{eq:SS} with $S=G^n$, we obtain the desired result.
\end{proof}

\subsection{Error estimate between $\tilde u^n$ and $u^n$}\label{S4.2}
This part carries out the error estimate between the auxiliary process $\tilde u^n$ and the numerical solution $u^n$. This will be accomplished by studying the moment estimates of $E(t):=\tilde U(t)-U(t)$, since
$\sup_{x\in\mathcal{O}}\vert\tilde u^n(t,x)-u^n(t,x)\vert=\Vert E(t)\Vert_{l^\infty_n}$.

\begin{proposition}\label{H-1E}
Let $u_0\in\mathcal{C}^3(\OO)$. Then 
there exists some constant $C=C(T)$ such that for any $t\in[0,T]$,
\begin{align}\label{En-1-4}
\E\left[\Vert (-A_n)^{-\frac{1}{2}}E(t)\Vert _{l_n^2}^{4}\right]+\E\left[\Big\vert \int_0^t\Vert (-A_n)^{\frac{1}{2}}E(s)\Vert _{l_n^2}^2\ud s\Big\vert ^2\right]\le Cn^{-4}.
\end{align}
\end{proposition}
\begin{proof}The proof is divided into two steps.

\textit{Step 1: We show that for any $p\ge2$, there exists $C=C(p,T)$ such that} 
\begin{align}\label{En-1}
\int_0^t\E\left[\Vert (-A_n)^{-\frac{1}{2}}E(s)\Vert _{l_n^2}^{p-2}\Vert (-A_n)^{\frac{1}{2}}E(s)\Vert _{l_n^2}^2\right]\ud s\le Cn^{-p}\quad\forall~t\in[0,T].
\end{align}

Subtracting  \eqref{NCH} from  \eqref{anx_u} leads to 
\begin{align}\label{Ent}
&\quad\ud E(t)+A_n^2 E(t)\ud t\\\notag
&=A_n\left\{F_n(\mathbb U(t))-F_n(U(t))\right\}\ud t+\sqrt{n/\pi}\left\{\Sigma_n(\mathbb U(t))-\Sigma_n(U(t))\right\}\ud \beta_t.
\end{align}
Applying It\^o's formula to $\Vert (-A_n)^{-\frac{1}{2}}E(t)\Vert ^{p}$ $(p\ge2)$ reveals
\begin{small}
\begin{align}\label{H-1error}
&\quad\ud \Vert (-A_n)^{-\frac{1}{2}}E(t)\Vert ^{p}=-p\Vert (-A_n)^{-\frac{1}{2}}E(t)\Vert ^{p-2}\Vert (-A_n)^{\frac{1}{2}}E(t)\Vert ^2\ud t\\\notag
&\quad+p\Vert (-A_n)^{-\frac{1}{2}}E(t)\Vert ^{p-2}\langle E(t),F_n(U(t))-F_n(\mathbb U(t))\rangle\ud t\\\notag
&\quad+\frac{n}{2\pi}p\Vert (-A_n)^{-\frac{1}{2}}E(t)\Vert ^{p-2}\Vert (-A_n)^{-\frac{1}{2}}\{\Sigma_n(\mathbb U(t))-\Sigma_n(U(t))\}\Vert _{\mathrm F}^2\ud t\\\notag
&\quad+\frac{n}{2\pi}p(p-2)\Vert (-A_n)^{-\frac{1}{2}}E(t)\Vert ^{p-4}\Vert E(t)^\top(-A_n)^{-1}\{\Sigma_n(\mathbb U(t))-\Sigma_n(U(t))\}\Vert ^2\\\notag
&\quad+\frac{p}{\sqrt{\pi}}\Vert (-A_n)^{-\frac{1}{2}}E(t)\Vert ^{p-2}\left\langle(-A_n)^{-1}E(t),\sqrt{n}\{\Sigma_n(\mathbb U(t))-\Sigma_n(U(t))\}\ud \beta_t\right\rangle.
\end{align}
\end{small}
For $t\in[0,T]$, we introduce 
\begin{align*}
J_1(t)&:=\langle E(t),F_n(U(t))-F_n(\mathbb U(t))\rangle,\\
J_2(t)&:=n\Vert (-A_n)^{-\frac{1}{2}}\{\Sigma_n(\mathbb U(t))-\Sigma_n(U(t))\}\Vert ^2_{\mathrm F},\\
J_3(t)&:=n\Vert E(t)^\top(-A_n)^{-1}\{\Sigma_n(\mathbb U(t))-\Sigma_n(U(t))\}\Vert ^2.
\end{align*}
Since $\Vert z^\top B_1\Vert \le \Vert B_1\Vert _{\mathrm F}\Vert z\Vert $ for any $z\in\mathbb{R}^{n-1}$ and $B_1\in\mathbb{R}^{(n-1)\times(n-1)}$, we have
\begin{align}\label{J3t}
J_3(t)&
\le J_2(t)\Vert (-A_n)^{-\frac{1}{2}}E(t)\Vert ^2.
\end{align}
Due to \eqref{H-1error}, \eqref{J3t}, and $E(0)=0$, it holds that for any $p\ge2$, 
\begin{align}
\label{ANTp}
&\quad\E\left[\Vert (-A_n)^{-\frac{1}{2}}E(t)\Vert ^{p}\right]+p\int_0^t\E\left[\Vert (-A_n)^{-\frac{1}{2}}E(s)\Vert ^{p-2}\Vert (-A_n)^{\frac{1}{2}}E(s)\Vert ^2\right]\ud s\\\notag
&\le \frac{1}{2}p^2\int_0^t\E\left[\Vert (-A_n)^{-\frac{1}{2}}E(s)\Vert ^{p-2}\left(J_1(s)+J_2(s)\right)\right]\ud s.
\end{align}

The Lipschitz continuity of $\sigma$ and $e_l(k)\le \sqrt{2/n}$ for $l,k\in\mathbb{Z}_{n-1}$ imply 
\begin{align}\label{Sigma-sigma}
&\quad\Vert \{\Sigma_n(\mathbb U(t))-\Sigma_n(U(t))\}e_l\Vert ^2=\sum_{k=1}^{n-1}\left\vert \left(\sigma(\mathbb U_k(t))-\sigma(U_k(t))\right)e_l(k)\right\vert ^2\\\notag
&\le\frac{C}{n}\sum_{k=1}^{n-1}\left\vert \mathbb U_k(t)-U_k(t)\right\vert ^2= \frac{C}{n}\left\Vert \mathbb U(t)-U(t)\right\Vert ^2\quad\forall~l\in\mathbb{Z}_{n-1}.
\end{align}
Analogous to \eqref{ASigmaUF}, by \eqref{Sigma-sigma} and the symmetry of $A_n$ and $\Sigma_n(\mathbb U(t))-\Sigma_n(U(t))$,
\begin{align}\label{I3F}
J_2(t)
&=n\sum_{l=1}^{n-1}(-\lambda_{l,n})^{-1}\Vert \{\Sigma_n(\mathbb U(t))-\Sigma_n(U(t))\}e_l\Vert ^2\\\notag
&\le C\left\Vert \mathbb U(t)-U(t)\right\Vert ^2\le C\Vert \mathbb U(t)-\tilde U(t)\Vert ^2+ C\Vert E(t)\Vert ^2,
\end{align}
due to $ \sum_{l=1}^{n-1}(-\lambda_{l,n})^{-1}\le C.$
By \eqref{fb-fa} and the H\"older inequality, for $1\le q<\infty$, 
\begin{align}\label{F-F}
& \Vert F_n(\mathbb U(t))-F_n(\tilde U(t))\Vert ^2\le C\sum_{k=1}^{n-1}\big(1+\mathbb U_k(t)^2+\tilde U_k(t)^2\big)^2\vert \mathbb U_k(t)-\tilde U_k(t)\vert ^2\\\notag
&\le Cn\big(1+\Vert \mathbb U(t)\Vert _{l^{4q^\prime}_n}^4+\Vert \tilde U(t)\Vert _{l_n^{4q^\prime}}^{4}\big)\Vert \mathbb U(t)-\tilde U(t)\Vert ^2_{l_n^{2q}},
\end{align}
where $q^\prime=q/(q-1)$. Denote $K(t):=\big(1+\Vert \mathbb U(t)\Vert _{l^{8}_n}^4+\Vert \tilde U(t)\Vert _{l_n^{8}}^{4}\big)\Vert \mathbb U(t)-\tilde U(t)\Vert ^2_{l_n^{4}}$ for $t\in[0,T]$. 
Utilizing \eqref{fbfa} and the Young inequality, we get
\begin{align*}
J_1(t)&=\langle E(t),F_n(U(t))-F_n(\tilde U(t))\rangle+\langle E(t),F_n(\tilde U(t))-F_n(\mathbb U(t))\rangle\\
&\le \frac{3}{2}\Vert E(t)\Vert ^2+\frac{1}{2}\Vert F_n(\mathbb U(t))-F_n(\tilde U(t))\Vert ^2,
\end{align*}
which, along with \eqref{F-F} (with $q=2$), \eqref{I3F} and $\Vert \cdot\Vert _{l_n^2}\le C\Vert \cdot\Vert _{l_n^4}$, indicates
\begin{align}\label{J12}
J_1(t)+J_2(t)&\le C\Vert E(t)\Vert ^2+CnK(t)\\\notag
&\le \epsilon_1\Vert (-A_n)^{\frac{1}{2}}E(t)\Vert ^2+C(\epsilon_1)\Vert (-A_n)^{-\frac{1}{2}}E(t)\Vert ^2+CnK(t)
\end{align}
for $\epsilon_1>0$.
By Proposition \ref{utilde-uh-1} and the H\"older inequality, for any $q\ge\theta\ge1$,
\begin{align}\label{Un-Unpp}
\E\left[\Vert \mathbb U(s)-\tilde U(s)\Vert ^q_{l^\theta_n}\right]&=\E\bigg[\Big(\frac{\pi}{n}\sum_{l=1}^{n-1}\vert u(s,lh)-\tilde u^n(s,lh)\vert ^\theta\Big)^{\frac{q}{\theta}}\bigg]\le C n^{-q}
\end{align}
for any $s\in[0,T]$.
Thanks to \eqref{eq:bound-e}, we  have that for any $q\ge\theta\ge1$ and $t\in[0,T]$,
\begin{align}\label{Unp}
\E\Big[\Vert \mathbb U(t)\Vert _{l_n^{\theta}}^{q}\Big]\le C\E\Big[\Vert \mathbb U(t)\Vert _{l_n^{\infty}}^{q}\Big]\le C\E\Big[\sup_{x\in\mathcal{O}}\vert u(t,x)\vert ^q\Big]
\le C(q,T).
\end{align}
A combination of \eqref{Un-Unpp} and \eqref{Unp} yields that for any $q\ge\theta\ge1$,
\begin{align*}
\E\big[\Vert \tilde U(s)\Vert _{l_n^{\theta}}^{q}\big]\le C(\theta,q,T)\quad \forall~s\in[0,T].
\end{align*}
The previous three estimates and the H\"older inequality ensure that for any $p\ge1$,
\begin{align}\label{Kt}\notag
\left\Vert K(s)\right\Vert _{L^p(\Omega)}
&\le C\big(1+\Vert \mathbb U(s)\Vert _{L^{8p}(\Omega;l_n^8)}^{4}+\Vert \tilde U(s)\Vert _{L^{8p}(\Omega;l_n^8)}^{4}\big)\Vert \mathbb U(s)-\tilde U(s)\Vert _{L^{4p}(\Omega;l^{4}_n)}^2\\
&\le C n^{-2}\quad \forall~s\in[0,T].
\end{align}
Since $\sqrt{\pi}\Vert \cdot\Vert =\sqrt{n}\Vert \cdot\Vert _{l_n^2}$, it follows from \eqref{ANTp} and \eqref{J12} with $\epsilon_1=1/p^{2}$ that
\begin{align*}
&\quad\E\left[\Vert (-A_n)^{-\frac{1}{2}}E(t)\Vert _{l_n^2}^{p}\right]+\big(p-\frac{1}{2}
\big)\int_0^t\E\left[\Vert (-A_n)^{-\frac{1}{2}}E(s)\Vert _{l_n^2}^{p-2}\Vert (-A_n)^{\frac{1}{2}}E(s)\Vert _{l_n^2}^2\right]\ud s\\
&\le C \int_0^t\E\left[\Vert (-A_n)^{-\frac{1}{2}}E(s)\Vert _{l_n^2}^{p}\right]\ud s+C\int_0^t\E\left[\Vert (-A_n)^{-\frac{1}{2}}E(s)\Vert _{l_n^2}^{p-2}K(s)\right]\ud s\\
&\le C \int_0^t\E\left[\Vert (-A_n)^{-\frac{1}{2}}E(s)\Vert _{l_n^2}^{p}\right]\ud s+C\int_0^t\E\left[\vert K(s)\vert ^{\frac{p}{2}}\right]\ud s.
\end{align*}
This together with the Gronwall lemma and \eqref{Kt} leads to $$\E\big[\Vert (-A_n)^{-\frac{1}{2}}E(t)\Vert _{l_n^2}^{p}\big]\le Cn^{-p}$$ for any $t\in[0,T]$, and consequently, we obtain
\eqref{En-1}.

\textit{Step 2: We prove \eqref{En-1-4}.}
Taking $p=2$ in \eqref{H-1error} and using \eqref{J12}, we deduce
\begin{align}\label{H-1errorl2}
\Vert (-A_n)^{-\frac{1}{2}}E(t)\Vert ^{2}+2\int_0^t\Vert (-A_n)^{\frac{1}{2}}E(s)\Vert ^2\ud s
\le\int_0^t 2J_1(s)+J_2(s)\ud s+nM(t)\\\notag
\le 2\epsilon_1\int_0^t \Vert (-A_n)^{\frac{1}{2}}E(s)\Vert ^2\ud s+\int_0^tC_{\epsilon_1}\Vert (-A_n)^{-\frac{1}{2}}E(s)\Vert ^2+CnK(s)\ud s
+nM(t)
\end{align}
for $\epsilon_1\in(0,1)$. Here
$M(t):=\frac{2}{\sqrt {n\pi}}\int_0^t\langle(-A_n)^{-1}E(s),\{\Sigma_n(\mathbb U(s))-\Sigma_n(U(s))\}\ud \beta_s\rangle$
is a martingale since \eqref{En-1} and the boundedness of $\sigma$ imply
$\E[\vert M(t)\vert ^2]<\infty$. Moreover, by the It\^o isometry, \eqref{J3t},  \eqref{I3F}, and the Young inequality,
\begin{align*}
\E[\vert M(t)\vert ^2]&= \frac{4}{\pi n}\int_0^t\E[\Vert \{\Sigma_n(\mathbb U(s))-\Sigma_n(U(s))\}(-A_n)^{-1}E(s)\Vert ^2]\ud s\\
&\le Cn^{-2}\int_0^t\E[J_3(s)]\ud s\le  Cn^{-2}\int_0^t\E[\Vert (-A_n)^{-\frac{1}{2}}E(s)\Vert ^2J_2(s)]\ud s\\
&\le  Cn^{-2}\int_0^t\E[\Vert \mathbb U(s)-\tilde U(s)\Vert ^4]\ud s+Cn^{-2}\int_0^t\E[\Vert (-A_n)^{-\frac{1}{2}}E(s)\Vert ^4]\ud s\\
&\quad+Cn^{-2}\int_0^t\E[\Vert (-A_n)^{-\frac{1}{2}}E(s)\Vert ^2\Vert (-A_n)^{\frac{1}{2}}E(s)\Vert ^2]\ud s.
\end{align*}
Then \eqref{H-1errorl2} with $\epsilon_1=\frac{1}{2}$ and the Young inequality give
\begin{align*}
&\quad \E\left[\Vert (-A_n)^{-\frac{1}{2}}E(t)\Vert _{l_n^2}^{4}\right]+\E\left[\Big\vert \int_0^t\Vert (-A_n)^{\frac{1}{2}}E(s)\Vert _{l_n^2}^2\ud s\Big\vert ^2\right]\\
&\le C\E\left[\Big\vert \int_0^t\Vert (-A_n)^{-\frac{1}{2}}E(s)\Vert _{l_n^2}^2\ud s\Big\vert ^2\right]+\E[\vert M(t)\vert ^2]+C\E\left[\Big\vert \int_0^tK(s)\ud s\Big\vert ^2\right]\\
&\le C\int_0^t\E[\Vert (-A_n)^{-\frac{1}{2}}E(s)\Vert _{l_n^2}^4]\ud s+ C\int_0^t\E[\Vert \mathbb U(s)-\tilde U(s)\Vert _{l_n^2}^4]\ud s\\
&\quad+C\int_0^t\E[\Vert (-A_n)^{-\frac{1}{2}}E(s)\Vert _{l_n^2}^2\Vert (-A_n)^{\frac{1}{2}}E(s)\Vert _{l_n^2}^2]\ud s+C\int_0^t\E[K(s)^2]\ud s\\
&\le  C\int_0^t\E[\Vert (-A_n)^{-\frac{1}{2}}E(s)\Vert _{l_n^2}^4]\ud s+Cn^{-4},
\end{align*}
where  we have used \eqref{En-1} with $p=4$,
\eqref{Un-Unpp}, and \eqref{Kt} in the last step. Thus,
\eqref{En-1-4} follows immediately from the Gronwall lemma.\end{proof}

Similar to Proposition \ref{Unl2-pro}, we have the following regularity estimate of $\tilde U(t)$. 
\begin{lemma}\label{tildeUnl2}
Let  $u_0\in\mathcal{C}^1(\OO)$ and $p\ge1$. Then
\begin{align*}
\E\big[\Vert (-A_n)^{\frac{1}{2}}\tilde U(t)\Vert ^p_{l_n^2}\big]&\le C(T,p)\quad\forall~t\in[0,T].
\end{align*} 
\end{lemma}
\begin{proof}
For $t\in[0,T]$, denote $\tilde O(t):=\sqrt{n/\pi}\int_0^t\exp(-A_n^2(t-s))\Sigma_n(\mathbb U(s))\ud \beta_s$
and $\tilde V(t):=\tilde U(t)-\tilde O(t)$. Then $\tilde V$ satisfies
$$\frac{\ud}{\ud t} \tilde V(t)=-A_n^2\tilde V(t)+A_nF_n(\mathbb U(t)),\quad t\in(0,T],$$ and $\tilde V(0)=\mathbb U(0).$
Analogously to \eqref{AnVnt}, we have that for $t\in[0,T]$,
\begin{align*}
\Vert (-A_n)^{\frac{1}{2}} \tilde V(t)\Vert _{l^2_n}
&\le C+C\int_0^t(t-s)^{-\frac{3}{4}}\Vert F_n(\mathbb U(s))\Vert _{l_n^2}\ud s\\
&\le C+C\int_0^t(t-s)^{-\frac{3}{4}}\left(1+\Vert \mathbb U(s)\Vert ^3_{l_n^6}\right)\ud s,
\end{align*}
which along with the Minkowski inequality and \eqref{Unp} yields that for $p\ge2$,
\begin{align*}
\E\big[\Vert (-A_n)^{\frac{1}{2}} \tilde V(t)\Vert _{l^2_n}^p\big]\le C\quad \forall~t\in[0,T].
\end{align*}
Since $\sigma$ is bounded, a similar argument of \eqref{Onl2} yields  
$\E[\Vert (-A_n)^{\frac{1}{2}}\tilde O(t)\Vert ^p_{l_n^2}]
\le C(T,p)$ for all $t\in[0,T].$
Combining the above estimates finally finishes the proof.
\end{proof}

Further, we present the local Lipschitz continuity of $F_n$ under the  $\Vert (-A_n)^{\frac{1}{2}}\cdot\Vert _{l_n^2}$-norm, which is crucial for the strong convergence analysis in Theorem \ref{main-strong}.

\begin{lemma}\label{FaFb}
For any $a,b\in\mathbb{R}^{n-1}$, 
\begin{align*}
&\quad\Vert (-A_n)^{\frac{1}{2}}(F_n(a)-F_n(b))\Vert _{l_n^2}\\
&\le C
\big(1+\Vert (-A_n)^{\frac{1}{2}}a\Vert ^2_{l_n^2}+\Vert (-A_n)^{\frac{1}{2}}b\Vert ^2_{l_n^2}\big)\Vert (-A_n)^{\frac{1}{2}}(a-b)\Vert _{l_n^2}.
\end{align*}
\end{lemma}
\begin{proof}
Let $a=(a_1,\ldots,a_{n-1})$, $b=(b_1,\ldots,b_{n-1})$, and $a_0=b_0=a_n=b_n=0$. Then by \eqref{n2B} and $f(x)=x^3-x$, 
\begin{align*}
&\Vert (-A_n)^{\frac{1}{2}}(F_n(a)-F_n(b))\Vert ^2_{l_n^2}=\frac{n}{\pi}\sum_{j=1}^n\vert f(a_j)-f(b_j)-f(a_{j-1})+f(b_{j-1})\vert ^2\\
&\le\frac{2n}{\pi}\sum_{j=1}^n\vert a_j^3-b_j^3-a_{j-1}^3+b_{j-1}^3\vert ^2+2\Vert (-A_n)^{\frac{1}{2}}(a-b)\Vert ^2_{l_n^2}.
\end{align*}
By noticing that
\begin{align*}
&\quad\vert a_j^2+b_j^2+a_jb_j-a_{j-1}^2-b_{j-1}^2-a_{j-1}b_{j-1}\vert ^2\\&\le 3(\vert a_j^2-a_{j-1}^2\vert ^2+\vert b_j^2-b_{j-1}^2\vert ^2+\vert a_jb_j-a_{j-1}b_{j-1}\vert ^2)\\
&\le 12(\vert a_j-a_{j-1}\vert ^2\Vert a\Vert ^2_{l_n^\infty}+\vert b_j-b_{j-1}\vert ^2\Vert b\Vert ^2_{l_n^\infty})+6(\Vert b\Vert ^2_{l_n^\infty}\vert a_j-a_{j-1}\vert ^2+\Vert a\Vert ^2_{l_n^\infty}\vert b_j-b_{j-1}\vert ^2)\\
&\le C\left(\vert a_j-a_{j-1}\vert ^2+\vert b_j-b_{j-1}\vert ^2\right)(\Vert a\Vert ^2_{l_n^\infty}+\Vert b\Vert ^2_{l_n^\infty}),
\end{align*}
we obtain from \eqref{n2B} that
\begin{align}\label{a3b3}
&\quad\frac{n}{\pi}\sum_{j=1}^n\vert a_j^3-b_j^3-a_{j-1}^3+b_{j-1}^3\vert ^2\\\notag
&\le\frac{2n}{\pi}\sum_{j=1}^n\vert a_j^2+b_j^2+a_jb_j-a_{j-1}^2-b_{j-1}^2-a_{j-1}b_{j-1}\vert ^2(a_j-b_j)^2\\\notag
&\quad+\frac{2n}{\pi}\sum_{j=1}^n\vert a_{j-1}^2+b_{j-1}^2+a_{j-1}b_{j-1}\vert ^2(a_j-b_j-a_{j-1}+b_{j-1})^2\\\notag
&\le C\Vert a-b\Vert ^2_{l_n^\infty}\frac{n}{\pi}\sum_{j=1}^n\left(\vert a_j-a_{j-1}\vert ^2+\vert b_j-b_{j-1}\vert ^2\right)(\Vert a\Vert ^2_{l_n^\infty}+\Vert b\Vert ^2_{l_n^\infty})\\\notag
&\quad+C\max_{1\le j\le n}(a_{j-1}^4+b_{j-1}^4)\Vert (-A_n)^{\frac{1}{2}}(a-b)\Vert ^2_{l_n^2}\\\notag
&\le C\left(\Vert (-A_n)^{\frac{1}{2}}a\Vert ^2_{l_n^2}+\Vert (-A_n)^{\frac{1}{2}}b\Vert ^2_{l_n^2}\right)(\Vert a\Vert ^2_{l_n^\infty}+\Vert b\Vert ^2_{l_n^\infty})\Vert a-b\Vert ^2_{l_n^\infty}\\\notag
&\quad+C(\Vert a\Vert _{l_n^\infty}^4+\Vert b\Vert _{l_n^\infty}^4)\Vert (-A_n)^{\frac{1}{2}}(a-b)\Vert ^2_{l_n^2}.
\end{align}
Finally, a combination of \eqref{l2H1} and \eqref{a3b3} completes the proof.
\end{proof}

Now we are ready to present the main result of this section on the strong convergence rate of the numerical solution associated with the spatial FDM.

\begin{theorem}\label{main-strong}
Let $u_0\in\mathcal{C}^3(\OO)$ and $\zeta\in[1,2)$. Then there exists some constant $C=C(\zeta,T)$ such that for any $(t,x)\in[0,T]\times\OO$,
\begin{align*}
\E\left[\vert u(t,x)-u^n(t,x)\vert ^{\zeta}\right]\le C n^{-\zeta}.
\end{align*}
\end{theorem}
\begin{proof}
Due to Proposition \ref{utilde-uh-1}, it remains to show that for any $(t,x)\in[0,T]\times\OO$,
\begin{align}\label{tildeU-Uinfty}
\E\left[\vert \tilde u^n(t,x)-u^n(t,x)\vert^{\zeta}\right]\le C n^{-\zeta}.
\end{align}
Notice that $\sup_{x\in\OO}\Vert \tilde u^n(t,x)-u^n(t,x)\Vert _{L^\zeta(\Omega)}\le \Vert \tilde U(t)-U(t)\Vert _{L^\zeta(\Omega,l_n^\infty)}=
\Vert E(t)\Vert _{L^\zeta(\Omega,l_n^\infty)}.$  In view of \eqref{l2H1}, a sufficient condition for \eqref{tildeU-Uinfty} is 
\begin{align}\label{tildeU-UH1}
\Vert (-A_n)^{\frac{1}{2}}E(t)\Vert _{L^\zeta(\Omega,l_n^2)}\le Cn^{-1}\quad\forall~t\in[0,T].
\end{align}

Fix $t\in(0,T]$. For $\mu\in[0,\frac{1}{2}]$ and $s\in[0,t]$, we denote
\begin{align*}
K_{\mu}(s)&:=e^{-A_n^2(t-s)}(-A_n)^{1+\mu}\left\{F_n(\mathbb U(s))-F_n(U(s))\right\},\\
L_{\mu}(s)&:=\sqrt{n}(-A_n)^{\mu}e^{-A_n^2(t-s)}\left\{\Sigma_n(\mathbb U(s))-\Sigma_n(U(s))\right\}.
\end{align*}
Then the variation of constants formula applied to \eqref{Ent} and  $E(0)=0$ produce
$$(-A_n)^{\mu}E(t)=-\int_0^tK_{\mu}(s)\ud s+\frac{1}{\sqrt{\pi}}\int_0^tL_{\mu}(s)\ud \beta_s.$$
By the Burkholder inequality, it holds that for any $p\ge 1/2$, 
\begin{align}\label{EnH1}
&\quad\E\left[\Vert (-A_n)^{\mu}E(t)\Vert ^{2p}\right]\\\notag
&\le C\E\left[\Big\vert \int_0^t\Vert K_{\mu}(s)\Vert \ud s\Big\vert ^{2p}\right]+C\E\left[\Big\vert \int_0^t\Vert L_{\mu}(s)\Vert _{\mathrm F}^2\ud s\Big\vert ^p\right].
\end{align}

Similarly to \eqref{ASigmaUF}, we have
$$
\Vert L_\mu(s)\Vert ^2_{\mathrm F}
=n\sum_{l=1}^{n-1}(-\lambda_{l,n})^{2\mu}e^{-2\lambda^2_{l,n}(t-s)}\Vert \{\Sigma_n(\mathbb U(s))-\Sigma_n(U(s))\}e_l\Vert ^2.
$$
Further, taking \eqref{Sigma-sigma} and \eqref{ex<x} with $\alpha_1=\frac{1}{4}+\mu+\epsilon$ into account, we arrive at
\begin{align}\label{J2mu}
\Vert L_{\mu}(s)\Vert ^2_{\mathrm F}&\le C \sum_{l=1}^{n-1}(-\lambda_{l,n})^{2\mu}e^{-\lambda^2_{l,n}(t-s)}\Vert \mathbb U(s)-U(s)\Vert ^2\\\notag
&\le C_\epsilon(t-s)^{-\frac{1}{4}-\mu-\epsilon}\Vert \mathbb U(s)-U(s)\Vert ^2\\\notag
&\le C_\epsilon(t-s)^{-\frac{1}{4}-\mu-\epsilon}\Vert \mathbb U(s)-\tilde U(s)\Vert ^2+C_\epsilon(t-s)^{-\frac{1}{4}-\mu-\epsilon}\Vert E(s)\Vert ^2,
\end{align}
where $0<\epsilon\ll 1$.
The remainder of the proof of \eqref{tildeU-UH1} is separated into two steps.

\textit{Step 1: In this step, we take $\mu=0$ and estimate $\Vert E(t)\Vert _{l_n^2}$.}

By \eqref{smooth}, \eqref{F-F} with $q=1$, and Lemma \ref{FaFb},
\begin{align*}
\Vert K_0(s)\Vert _{l_n^2}&\le \Vert e^{-A_n^2(t-s)}(-A_n)\{F_n(\tilde U(s))-F_n(U(s))\}\Vert _{l_n^2}\\
&\quad+\Vert e^{-A_n^2(t-s)}(-A_n)\{F_n(\mathbb U(s))-F_n(\tilde U(s))\}\Vert _{l_n^2}\\
&\le C(t-s)^{-\frac{1}{4}}\big(1+\Vert (-A_n)^{\frac{1}{2}}\tilde U(s)\Vert ^2_{l_n^2}+\Vert (-A_n)^{\frac{1}{2}}U(s)\Vert ^2_{l_n^2}\big)\Vert (-A_n)^{\frac{1}{2}}E(s)\Vert _{l_n^2}\\
&\quad+C(t-s)^{-\frac{1}{2}}\big(1+\Vert \mathbb U(t)\Vert _{l^\infty_n}^2+\Vert \tilde U(t)\Vert _{l_n^\infty}^2\big)\Vert \mathbb U(t)-\tilde U(t)\Vert _{l_n^2}.
\end{align*}
Plugging this inequality and \eqref{J2mu} (with $\mu=0$) into \eqref{EnH1} (with $p=1$), we obtain
\begin{align*}
&\quad\E\left[\Vert E(t)\Vert _{l_n^2}^2\right]\le C\E\left[\Big\vert \int_0^t\Vert K_0(s)\Vert _{l_n^2}\ud s\Big\vert ^2\right]+\frac{C}{n}\int_0^t\E\left[\Vert L_0(s)\Vert _{\mathrm F}^2\right]\ud s\\
&\le C\E\left[\Big\vert \int_0^t(t-s)^{-\frac{1}{4}}\Big(1+\Vert (-A_n)^{\frac{1}{2}}\tilde U(s)\Vert ^2_{l_n^2}+\Vert (-A_n)^{\frac{1}{2}}U(s)\Vert ^2_{l_n^2}\Big)\Vert (-A_n)^{\frac{1}{2}}E(s)\Vert _{l_n^2}\ud s\Big\vert ^2\right]\\
&\quad+C\E\left[\Big\vert \int_0^t(t-s)^{-\frac{1}{2}}\big(1+\Vert \mathbb U(s)\Vert _{l^\infty_n}^2+\Vert \tilde U(s)\Vert _{l_n^\infty}^2\big)\Vert \mathbb U(s)-\tilde U(s)\Vert _{l_n^2}\ud s\Big\vert ^2\right]\\
&\quad+C_\epsilon\int_0^t(t-s)^{-\frac{1}{4}-\epsilon}\E\big[\Vert \mathbb U(s)-\tilde U(s)\Vert _{l_n^2}^2\big]\ud s+C_\epsilon\int_0^t(t-s)^{-\frac{1}{4}-\epsilon}\E\big[\Vert E(s)\Vert _{l_n^2}^2\big]\ud s\\
&=:CErr_1+CErr_2+C_{\epsilon}Err_3+C_\epsilon\int_0^t(t-s)^{-\frac{1}{4}-\epsilon}\E\big[\Vert E(s)\Vert _{l_n^2}^2\big]\ud s.
\end{align*}
We proceed to estimate $Err_1$, $Err_2$ and $Err_3$.
By the H\"older inequality, Proposition \ref{H-1E}, Proposition \ref{Unl2-pro}, and Lemma \ref{tildeUnl2}, we obtain 
\begin{align*}
Err_1&\le Cn^{-2}\int_0^t(t-s)^{-\frac{1}{2}}(1+\Vert (-A_n)^{\frac{1}{2}}\tilde U(s)\Vert ^4_{L^8(\Omega;l_n^2)}+\Vert (-A_n)^{\frac{1}{2}}U(s)\Vert ^4_{L^8(\Omega;l_n^2)})\ud s\\
&\le Cn^{-2}.
\end{align*}
Besides,  the Minkowski inequality and \eqref{l2H1} imply that $\sqrt{Err_2}$ is bounded by
\begin{align*}
C\int_0^t(t-s)^{-\frac{1}{2}}\big(1+\Vert \mathbb U(s)\Vert _{L^8(\Omega;l^\infty_n)}^2+\Vert (-A_n)^{\frac{1}{2}}\tilde U(s)\Vert _{L^8(\Omega;l^2_n)}^2\big)\Vert \mathbb U(s)-\tilde U(s)\Vert _{L^4(\Omega;l_n^2)}\ud s,
\end{align*}
which together with \eqref{Unp}, Lemma \ref{tildeUnl2} and \eqref{Un-Unpp} leads to
$$\sqrt{Err_2}\le C\int_0^t(t-s)^{-\frac{1}{2}}n^{-1}\ud s\le C n^{-1}.$$
Similarly, \eqref{Un-Unpp} also gives  $\sqrt{Err_3}\le C n^{-1}.$
Gathering the above estimates yields
\begin{align*}
\E\left[\Vert E(t)\Vert _{l_n^2}^2\right]
&\le C n^{-2}+C_\epsilon\int_0^t(t-s)^{-\frac{1}{4}-\epsilon}\E\left[\Vert E(s)\Vert _{l_n^2}^2\right]\ud s.
\end{align*}
Then the Gronwall lemma with weak singularities (see e.g., \cite[Lemma 3.4]{GI98}) implies 
$\E[\Vert E(t)\Vert _{l_n^2}^2]
\le C n^{-2},$
which along with \eqref{Un-Unpp} with $\theta=q=2$ gives
\begin{equation}\label{EE2}
\E\left[\Vert \mathbb U(t)-U(t)\Vert _{l_n^2}^2\right]
\le C n^{-2}.
\end{equation}

\textit{Step 2: In this step, we take $\mu=\frac{1}{2}$ and estimate $\Vert (-A_n)^{\frac{1}{2}}E(t)\Vert _{l_n^2}$.} 

By \eqref{smooth} and a similar argument of  \eqref{F-F} with $q=1$,
\begin{align*}
\Vert K_{\frac{1}{2}}(s)\Vert _{l_n^2}&\le C(t-s)^{-\frac{3}{4}}\Vert F_n(\mathbb U(s))-F_n(U(s))\Vert _{l_n^2}\\
&\le C(t-s)^{-\frac{3}{4}}\left(1+\Vert \mathbb U(s)\Vert _{l^\infty_n}^2+\Vert U(s)\Vert _{l_n^\infty}^2\right)\Vert \mathbb U(s)-U(s)\Vert _{l_n^2}.
\end{align*}
Hence,  the Minkowski and H\"older inequalities, \eqref{eq:bound-e}, \eqref{EE2},  \eqref{l2H1} and Proposition \ref{Unl2-pro} give that for $q\in[1,2)$ and $r=\frac{2q}{2-q}$,
\begin{align*}
&\quad\Big\Vert \int_0^t\Vert K_{\frac{1}{2}}(s)\Vert _{l_n^2}\ud s\Big\Vert _{L^q(\Omega)}\\\notag
&\le C\int_0^t(t-s)^{-\frac{3}{4}}\left\Vert 1+\Vert \mathbb U(s)\Vert _{l^\infty_n}^2+\Vert (-A_n)^{\frac12}U(s)\Vert _{l_n^2}^2\right\Vert _{L^r(\Omega)}\Vert \mathbb U(s)-U(s)\Vert _{L^2(\Omega;l_n^2)}\ud s\\
&\le Cn^{-1}.
\end{align*}
By the second inequality of \eqref{J2mu} and \eqref{EE2}, we have that for $p\in[\frac{1}{2},1)$,
\begin{align*}
\Big\Vert \int_0^t\Vert L_{\frac{1}{2}}(s)\Vert _{\mathrm F}^2\ud s\Big\Vert _{L^p(\Omega)}
\le Cn\int_0^t(t-s)^{-\frac{3}{4}-\epsilon}\Vert \mathbb U(s)- U(s)\Vert _{L^2(\Omega;l_n^2)}^2\ud s
\le Cn^{-1}.
\end{align*}
Combining the above two inequalities and \eqref{EnH1} with $p\in[\frac{1}{2},1)$ yields
\begin{align*}
\E\big[\Vert (-A_n)^{\frac{1}{2}}E(t)\Vert _{l_n^2}^{2p}\big]\le C\E\left[\Big\vert \int_0^t\Vert K_{\frac{1}{2}}(s)\Vert _{l_n^2}\ud s\Big\vert ^{2p}\right]+\frac{C}{n^p}\E\left[\Big\vert \int_0^t\Vert L_{\frac{1}{2}}(s)\Vert _{\mathrm F}^2\ud s\Big\vert ^p\right]\le \frac{C}{n^{2p}},
\end{align*}
which proves \eqref{tildeU-UH1}. The proof is completed.
\end{proof}

\section{Strong convergence analysis (II)}\label{S4-5}
As in the semi-discrete case, we introduce the auxiliary sequence $\{\tilde U^{i}\}_{i\in\mathbb{Z}_{m}^0}$ by
\begin{align}\label{eq:tildeUi}
\tilde U^{i+1}-\tilde U^i+\tau A_n^2 \tilde U^{i+1}=\tau A_n F_n(U(t_{i+1}))+\sqrt{n/\pi}\Sigma_n(\tilde U^i)(\beta_{t_{i+1}}-\beta_{t_{i}})
\end{align}
for $i\in\mathbb{Z}_{m-1}^0$, 
where $\tilde U^0=U^0$. 
As in \eqref{eq:untau}, let
 \begin{align}\label{eq:tildeun}
\tilde{u}^{n,\tau}(t_i,x)
&=\int_{\mathcal O}G^{n,\tau}_{t_i}(x,y)u_0(\kappa_n(y))\mathrm{d} y\\\notag&\quad+\int_0^{t_i}\int_{\mathcal O}\Delta_nG^{n,\tau}_{t_i-s+\tau}(x,y)f(u^{n}(\eta_\tau(s)+\tau,\kappa_n(y)))\mathrm{d} y\mathrm{d} s\\\notag
&\quad+\int_0^{t_i}\int_{\mathcal O}G^{n,\tau}_{t_i-s+\tau}(x,y)\sigma(\tilde{u}^{n,\tau}(\eta_\tau(s),\kappa_n(y)))W(\mathrm{d} s,\mathrm{d} y)
\end{align}
so that $\tilde u^{n,\tau}(t_i,kh)=\tilde U^i_k$ is the $k$th component  of $\tilde U^i$ and $\tilde u^{n,\tau}(t_i,x)=\Pi_n(\tilde u^{n,\tau}(t_i,\cdot))(x)$.

\begin{lemma} \label{lem:tildeU}
Let  $u_0\in\mathcal{C}^1(\OO)$ and $p\ge1$. Then for any $i\in\mathbb{Z}_{m}^0$,
\begin{align*}
\mathbb{E}\left[\big\Vert (-A_n)^{\frac{1}{2}}\tilde U^i\big\Vert ^p_{l_n^2}\right]\le C(T,p).
\end{align*} 
\end{lemma}
\begin{proof}
By denoting $\tilde O^i:=\sqrt{\frac{n}{\pi}}\int_0^{t_i}(I+\tau A_n^2)^{-(i-\lfloor\frac{s}{\tau}\rfloor)}\Sigma_n(\tilde U^{\lfloor\frac{s}{\tau}\rfloor})\ud \beta_s$, we have that
$\tilde V^{i}:=\tilde U^{i}-\tilde O^{i}$ satisfies
$\tilde  V^{i+1}-\tilde V^i+\tau A_n^2 \tilde V^{i+1}=\tau A_n F_n(U(t_{i+1}))$ for $i\in\mathbb{Z}_{m-1}^0$
and $\tilde V^0=U(0)$.
Then
by repeating the derivation of \eqref{eq:AnO2l2}, one can prove 
$\E[\sup_{i\in\mathbb{Z}_{m}}\|(-A_n)^{\frac12}\tilde O^{i}\|_{l_n^2}^p]\le C.
$
Moreover, similar to \eqref{eq:Anvi-dis}, we also have
$$\Vert (-A_n)^{\frac{1}{2}} \tilde V^{i}\Vert _{l^2_n}
\le C+C\tau\sum_{j=0}^{i-1}(t_i-t_j)^{-\frac{3}{4}}(1+\Vert U(t_{j+1})\Vert _{l_n^6}^3),$$
which along with Proposition \ref{Unl2-pro} yields $\E[\|(-A_n)^{\frac12}\tilde V^{i}\|_{l_n^2}^p]\le C$ for all $i\in\mathbb{Z}_{m}$. 
\end{proof}

\subsection{Error estimate between $\tilde u^{n,\tau}$ and $u^n$}\label{S5.1}
This part estimates the error between the auxiliary process $\tilde u^{n,\tau}$ and the spatial semi-discrete numerical solution $u^n$. We begin with the following error analysis between the fully discrete Green function $G^{n,\tau}_{s+\tau}(x,y)$ and the semi-discrete Green function $G^{n}_s(x,y)$.
\begin{lemma}\label{eq:Gntau}
For any $0<\epsilon\ll1$, there exists $C=C(T,\epsilon)$ such that for any $x\in\OO$,
\begin{gather*}
\int_0^{t_i}\int_{\mathcal{O}}|G^{n,\tau}_{s+\tau}(x,y)-G^{n}_s(x,y)|^2\ud y\ud s\le C\tau^{\frac{3}{4}-\epsilon},\quad i\in\mathbb{Z}_{m},\\
\int_0^{t_i}\int_{\mathcal{O}}|\Delta_nG^{n,\tau}_{s+\tau}(x,y)-\Delta_nG^{n}_s(x,y)|\ud y\ud s\le C\tau^{\frac{3}{8}-\epsilon},\quad i\in\mathbb{Z}_{m}.
\end{gather*}
\end{lemma} 
\begin{proof}
By the H\"older inequality,
it suffices to show that for $\mu\in\{0,1\}$,
\begin{equation*}
\int_0^{T}\left(\int_{\mathcal{O}}|(-\Delta_n)^\mu G^{n,\tau}_{s+\tau}(x,y)-(-\Delta_n)^\mu G^{n}_s(x,y)|^2\ud y\right)^{1-\frac{\mu}{2}}\ud s\le C\tau^{1-(\mu+\frac14)(1-\frac{\mu}{2})-\epsilon}.
\end{equation*}
Using the orthogonality of $\{\phi_{j}\circ\kappa_n\}_{j=1}^{n-1}$ and $|\phi_{j,n}(x)|\le 1$, 
\begin{align*}
&\quad\ \int_{\mathcal{O}}|(-\Delta_n)^\mu G^{n,\tau}_{s+\tau}(x,y)-(-\Delta_n)^{\mu}G^{n}_s(x,y)|^2\ud y\\
&\le \sum_{j=1}^{n-1}(-\lambda_{j,n})^{2\mu}\big|(1+\tau\lambda_{j,n}^2)^{-\lfloor\frac{s+\tau}{\tau}\rfloor}-e^{-\lambda_{j,n}^2s}\big|^2\le J_1^\mu(s)+2J_2^\mu(s)+2J_3^\mu(s),
\end{align*}
where 
\begin{align*}
&J_1^\mu(s):=\sum_{j=1}^{\lfloor\tau^{-\frac14}\rfloor-1}(-\lambda_{j,n})^{2\mu}\big|(1+\tau\lambda_{j,n}^2)^{-\lfloor\frac{s+\tau}{\tau}\rfloor}-e^{-\lambda_{j,n}^2s}\big|^2\\
&J_2^\mu(s):=\sum_{j=\lfloor\tau^{-\frac14}\rfloor}^{n-1}(-\lambda_{j,n})^{2\mu}e^{-2\lambda_{j,n}^2s},\quad J_3^\mu(s):=\sum_{j=\lfloor\tau^{-\frac14}\rfloor}^{n-1}(-\lambda_{j,n})^{2\mu}(1+\tau\lambda_{j,n}^2)^{-2\lfloor\frac{s+\tau}{\tau}\rfloor}.
\end{align*}
Since $\frac{4}{\pi^2}j^2\le \lambda_{j,n}\le j^2$ and $j\ge \frac{j+1}{2}$ for all $j\ge1$,
\begin{align*}
J_2^\mu(s)&\le \sum_{j=\lfloor\tau^{-\frac14}\rfloor}^{n-1}j^{4\mu}e^{-\frac{32}{\pi^4}j^4s}\le \sum_{j=\lfloor\tau^{-\frac14}\rfloor}^{n-1}\int_{j}^{j+1} z^{4\mu}e^{-\frac{32}{\pi^4}(\frac{z}{2})^4s}\ud z\\
&\le\int_{\lfloor\tau^{-\frac14}\rfloor}^\infty  z^{4\mu}e^{-\frac{2}{\pi^4}z^4s}\ud z\le\int_{\frac12\tau^{-\frac14}}^\infty  z^{4\mu}e^{-\frac{2}{\pi^4}z^4s}\ud z,
\end{align*}
by supposing without loss of generality that $\frac12\tau^{-\frac14}\le \tau^{-\frac14}-1\le \lfloor\tau^{-\frac14}\rfloor$.
Likewise, 
\begin{align}\label{eq:J3:dis}
J_3^\mu(s)\le \int_{\frac12\tau^{-\frac14}}^\infty  z^{4\mu}\left(1+\frac{\tau}{\pi^4}z^4\right)^{-2\lfloor\frac{s+\tau}{\tau}\rfloor}\ud z.
\end{align}

By
taking the change of variables $\tilde z=z\tau^{\frac14}$ and $ \tilde s=s/\tau$ in turn, we obtain
\begin{align*}
\int_{0}^{T}\left|J_2^\mu(s)\right|^{1-\frac{\mu}{2}}\ud s&\le C\int_0^{T/\pi^{4}}\Big|\int_{\frac12\tau^{-\frac14}}^{\infty}z^{4\mu}e^{-2z^4s}\ud z\Big|^{1-\frac{\mu}{2}}\ud s\\
&\le C\tau^{1-(\frac14+\mu)(1-\frac{\mu}{2})}\int_0^\infty\Big|\int_{\frac12}^\infty \tilde z^{4\mu}e^{-2\tilde z^4 \tilde s}\ud \tilde z\Big|^{1-\frac{\mu}{2}}\ud\tilde s\le C\tau^{1-(\frac14+\mu)(1-\frac{\mu}{2})}.
\end{align*}
Here the last integral is finite since by \eqref{ex<x}, for any $\alpha_1\in(\mu+\frac14,\frac{2}{2-\mu})$ and $\alpha_2>\frac{2}{2-\mu}$,
\begin{align*}
&\quad\ \int_0^\infty\Big|\int_{\frac12}^\infty \tilde z^{4\mu}e^{-2\tilde z^4 t}\ud \tilde z\Big|^{1-\frac{\mu}{2}}\ud t\\
&\le C\int_0^1\Big|\int_{\frac12}^\infty \tilde z^{4\mu-4\alpha_1} t^{-\alpha_1}\ud \tilde z\Big|^{1-\frac{\mu}{2}}\ud t+C\int_1^\infty\Big|\int_{\frac12}^\infty \tilde z^{4\mu-4\alpha_2} t^{-\alpha_2}\ud \tilde z\Big|^{1-\frac{\mu}{2}}\ud t<\infty.
\end{align*}
In a similar manner, applying the change of variables $\tilde z=z\tau^{\frac14}$ to \eqref{eq:J3:dis} leads to
\begin{align*}
J_3^\mu(s)&\le C\tau^{-(\mu+\frac14)}\int_{\frac12}^\infty (1+\pi^{-4}\tilde z^4)^{-2\lfloor\frac{s}{\tau}\rfloor-2}\tilde z^{4\mu}\ud \tilde z\le C\tau^{-(\mu+\frac14)}(1+\frac{1}{16\pi^4})^{-2\lfloor\frac{s}{\tau}\rfloor},
\end{align*}
which implies that
\begin{equation*}
\int_0^{T}\left|J_3^\mu(s)\right|^{1-\frac{\mu}{2}}\ud s
\le C\tau^{-(\mu+\frac14)(1-\frac{\mu}{2})}\sum_{k=0}^\infty\int_{t_k}^{t_{k+1}}(1+\frac{1}{16\pi^4})^{-k(2-\mu)}\ud s
\le C\tau^{1-(\mu+\frac14)(1-\frac{\mu}{2})}.
\end{equation*}
In order to estimate $J_1^\mu(s)$, we notice that for any $1\le j\le\lfloor\tau^{-\frac14}\rfloor-1$,
\begin{align*}
&\quad\ \big|(1+\tau\lambda_{j,n}^2)^{-\lfloor\frac{s+\tau}{\tau}\rfloor}-e^{-\lambda_{j,n}^2s}\big|\\
&\le e^{-\lfloor\frac{s+\tau}{\tau}\rfloor\ln (1+\tau\lambda_{j,n}^2)} \big|1-e^{-\lfloor\frac{s+\tau}{\tau}\rfloor(\tau\lambda_{j,n}^2-\ln (1+\tau\lambda_{j,n}^2))}\big|+e^{-\lambda_{j,n}^2s}\big|e^{-(\lfloor\frac{s+\tau}{\tau}\rfloor\tau-s)\lambda_{j,n}^2}-1\big|\\
&\le e^{-s c_0\lambda_{j,n}^2} \big|1-e^{-\lfloor\frac{s+\tau}{\tau}\rfloor c_1\tau^2\lambda_{j,n}^4}\big|+Ce^{-\lambda_{j,n}^2s}\tau\lambda_{j,n}^2\\
&\le Ce^{-s c_0\lambda_{j,n}^2}(s+\tau)\tau\lambda_{j,n}^4
+Ce^{-\lambda_{j,n}^2s}\tau\lambda_{j,n}^2\le  Ce^{-s c_0\lambda_{j,n}^2}(s\lambda_{j,n}^2+1)\tau\lambda_{j,n}^2,
\end{align*}
where we used  \eqref{1-ex} and  the fact that there exist some constants $c_0\in(0,1)$ and $c_1>0$ such that $\ln(1+z)\ge c_0z$ and $0\le z-\ln(1+z)\le c_1z^2$ for all $z\in[0,1]$. Hence by virtue of \eqref{ex<x}, for $\alpha_3:=\frac{2}{2-\mu}(1-\epsilon)$ with $0<\epsilon\ll1$, 
\begin{align*}
&\quad\ (-\lambda_{j,n})^{2\mu}e^{-2s c_0\lambda_{j,n}^2}(s^2\lambda_{j,n}^4+1)\tau^2\lambda_{j,n}^4\\
&\le C(-\lambda_{j,n})^{2\mu}(s\lambda_{j,n}^2)^{-2-\alpha_3}(s^2\lambda_{j,n}^4)\tau^2\lambda_{j,n}^4+
C(-\lambda_{j,n})^{2\mu}(s\lambda_{j,n}^2)^{-\alpha_3}\tau^2\lambda_{j,n}^4\\
&\le Cj^{4\mu+8-4\alpha_3}s^{-\alpha_3}\tau^2,
\end{align*}
which indicates that for any $0<\epsilon\ll 1$,
\begin{align*}
\int_{0}^{T}\left|J_1^\mu(s)\right|^{1-\frac{\mu}{2}}\ud s&\le C\Big|\int_{0}^{\tau^{-\frac14}}j^{4\mu+8-4\alpha_3}\tau^2\ud j\Big|^{1-\frac{\mu}{2}}\le C\tau^{1-(\mu+\frac14)(1-\frac{\mu}{2})-\epsilon}.
\end{align*}
Finally, collecting the above estimates finishes the proof.
\end{proof}
\begin{proposition}\label{eq:un-untau}
Let $u_0\in\mathcal{C}^2(\mathcal{O})$ and $0<\epsilon\ll1$.
 Then for any $p\ge1$, there exists some constant $C=C(p,T,\epsilon)$ such that for any $i\in \mathbb{Z}_m$ and $x\in\OO$,
$$\|u^{n}(t_i,x)-\tilde{u}^{n,\tau}(t_i,x)\|_{L^p(\Omega)}\le C\tau^{\frac38-\epsilon}.$$
\end{proposition}
\begin{proof}
By \eqref{unR0} and \eqref{eq:tildeun}, $u^{n}(t_i,x)-\tilde{u}^{n,\tau}(t_i,x)=\sum_{j=1}^6 Y^\tau_j$, where
\begin{align*}
Y^\tau_1&:=\int_{\mathcal O}G^{n}_{t_i}(x,y)u_0(\kappa_n(y))\mathrm{d} y-\int_{\mathcal O}G^{n,\tau}_{t_i}(x,y)u_0(\kappa_n(y))\mathrm{d} y,\\
Y^\tau_2&:=\int_0^{t_i}\int_{\mathcal O}\Delta_nG^{n}_{{t_i}-s}(x,y)[f(u^{n}(s,\kappa_n(y)))-f(u^{n}(\eta_\tau(s)+\tau,\kappa_n(y)))]\mathrm{d} y\mathrm{d} s,\\
Y^\tau_3&:=\int_0^{t_i}\int_{\mathcal O}[\Delta_nG^{n}_{{t_i}-s}(x,y)-\Delta_nG^{n,\tau}_{{t_i}-s+\tau}(x,y)]f(u^{n}(\eta_\tau(s)+\tau,\kappa_n(y)))\mathrm{d} y\mathrm{d} s,\\
Y^\tau_4&:=\int_0^{t_i}\int_{\mathcal O}G^n_{{t_i}-s}(x,y)[\sigma(u^n(s,\kappa_n(y)))-\sigma(u^n(\eta_\tau(s),\kappa_n(y)))]W(\mathrm{d} s,\mathrm{d} y),\\
Y^\tau_5&:=\int_0^{t_i}\int_{\mathcal O}[G^n_{{t_i}-s}(x,y)-G^{n,\tau}_{{t_i}-s+\tau}(x,y)]\sigma(u^{n}(\eta_\tau(s),\kappa_n(y)))W(\mathrm{d} s,\mathrm{d} y),\\
Y^\tau_6&:=\int_0^{t_i}\int_{\mathcal O}G^{n,\tau}_{{t_i}-s+\tau}(x,y)[\sigma(u^{n}(\eta_\tau(s),\kappa_n(y)))-\sigma(\tilde u^{n,\tau}(\eta_\tau(s),\kappa_n(y)))]W(\mathrm{d} s,\mathrm{d} y).
\end{align*}
Taking advantage of the following identity
$$(I+\tau A_n^2)^{-i}U^0-U^0=-\tau\sum_{k=1}^iA_n(I+\tau A_n^2)^{-k}A_nU^0\quad\forall ~i\in\mathbb{N}$$
and the fact $\tilde u^{n,\tau}(0,x)=\tilde u^n(0,x)$,
it can be verified that for $i\in\mathbb{Z}_{m}$,
\begin{align}\label{eq:fully-dis-u0}
&\quad\ \int_{\mathcal O}G^{n,\tau}_{t_i}(x,y)u_0(\kappa_n(y))\mathrm{d} y-\tilde u^n(0,x)\\\notag
&=-\int_0^{t_i}\int_{\mathcal{O}}\Delta_nG^{n,\tau}_{r+\tau}(x,z)\Delta_nu_0(\kappa_n(z))\ud z\ud r.
\end{align}
This together with \eqref{eq.u1tilde} and $\Delta_n u_0(\kappa_n(y))=\Delta_n u_0(y)$ yields
\begin{align*}
Y^\tau_1=\int_0^{t_i}\int_{\mathcal{O}}[\Delta_nG^{n,\tau}_{r+\tau}(x,z)- \Delta_n G^n_r(x,z)]\Delta_nu_0(\kappa_n(z))\ud z\ud r.
\end{align*}
According to the assumption  $u_0\in\mathcal C^2(\OO)$,
 Lemma \ref{eq:Gntau}, and
  \eqref{eq:u^nbound},
  $$\|Y_1^\tau\|_{L^p(\Omega)}+\|Y_3^\tau\|_{L^p(\Omega)}+\|Y_5^\tau\|_{L^p(\Omega)}\le C\tau^{\frac38-\epsilon}.$$
By the expression of $G_t^n$ (resp.\ $G_t^{n,\tau}$) and  \eqref{ex<x} (resp.\ \eqref{eq:smooth-dis}), for any $0<\epsilon\ll 1$,
\begin{align}\label{GnGn}
\vert G^n_{t}(x,y)\vert\le C_\epsilon t^{-\frac{1}{4}-\epsilon},\quad \vert\Delta_nG^n_{t}(x,y)\vert\le C_\epsilon t^{-\frac{3}{4}-\epsilon}
\quad\forall~t\in(0,T],~x,y\in\mathcal O,\\\label{GnGntau}
\vert G^{n,\tau}_{t_i}(x,y)\vert\le C_\epsilon {t_i}^{-\frac{1}{4}-\epsilon},\quad \vert\Delta_nG^{n,\tau}_{t_i}(x,y)\vert\le C_\epsilon t_i^{-\frac{3}{4}-\epsilon}
\quad\forall~i\in\mathbb{Z}_m,~x,y\in\mathcal O.
\end{align}
Hence using
 Lemma \ref{u2u3} and \eqref{eq:u^nbound} produces
   $\|Y_2^\tau\|^2_{L^p(\Omega)}+\|Y_4^\tau\|_{L^p(\Omega)}\le C\tau^{\frac38-\epsilon}$
   and
   $$\|Y_6^\tau\|_{L^p(\Omega)}\le \int_0^{t_i}(t_{i}-s+\tau)^{-\frac12-2\epsilon}\|u^{n}(\eta_\tau(s),\kappa_n(y))-\tilde u^{n,\tau}(\eta_\tau(s),\kappa_n(y))\|_{L^p(\Omega)}^2\ud s.$$
Gathering the above estimates on $\{Y_i^\tau\}_{i=1}^6$ and using the  singular Gronwall inequality (see e.g., \cite[Lemma 3.4]{GI98})  finally complete the proof.
\end{proof}

We close this part by giving the H\"older regularity of the fully discrete FDM.

\begin{lemma}\label{lemma:Holderuntau}
Let $u_0\in C^2(\OO)$. Then for any $\alpha\in(0,1)$, there exists some constant $C=C(\alpha,T,p)$ such that for any $0\le s<t\le T$ and $x,y\in\OO$,
\begin{align}\label{eq:Holderuntau}
\|u^{n,\tau}(t,x)-u^{n,\tau}(s,y)\|_{L^p(\Omega)}\le C(|t-s|^{\frac{\alpha}{4}}+|x-y|^\alpha).
\end{align}
\end{lemma}
\begin{proof}

We set $\mu\in\{0,1\}$, $\alpha\in(0,1-2\epsilon)$ and $0<\epsilon\ll 1$ throughout this proof.
Since $|\phi_{k,n}(x)-\phi_{k,n}(y)|\le C\min\{k|x-y|,1\}\le C(-\lambda_{k,n})^{\frac{\alpha}{2}}|x-y|^\alpha$, we have
\begin{align*}
&\quad\ \vert (-\Delta_n)^\mu G^n_{t_i-r+\tau}(x,z)-(-\Delta_n)^\mu G^n_{t_i-r+\tau}(y,z)\vert \\
&\le C\sum_{k=1}^{n-1}(-\lambda_{k,n})^{\mu}(1+\tau\lambda_{k,n}^2)^{-\lfloor\frac{t_i-r+\tau}{\tau}\rfloor}(-\lambda_{k,n})^{\frac{\alpha}{2}}|x-y|^\alpha\\
&\le C\big(\lfloor\frac{t_i-r+\tau}{\tau}\rfloor\tau\big)^{-\frac{1}{2}(\mu+\frac{\alpha}{2}+\frac{1}{2}+\epsilon)}|x-y|^\alpha
\le C(t_i-r)^{-\frac{1}{2}(\mu+\frac{\alpha}{2}+\frac{1}{2}+\epsilon)}|x-y|^\alpha,
\end{align*}
thanks to \eqref{eq:smooth-dis} with $\gamma=\mu+\frac{\alpha}{2}+\frac{1}{2}+\epsilon$.  Hence
$$ \int_0^{t_i}\int_\OO\vert (-\Delta_n)^\mu G^n_{t_i-r+\tau}(x,z)-(-\Delta_n)^\mu G^n_{t_i-r+\tau}(y,z)\vert^{2-\mu}\ud z\ud r\le C|x-y|^{\alpha(2-\mu)}.$$
By \eqref{1-ex} and $\ln(1+x)\le x$ for $x>0$, one obtains that for any $\alpha_1\in(0,1]$,
\begin{align*}
1-(1+\tau\lambda_{k,n}^2)^{-(j-i)}=1-e^{-(j-i)\ln(1+\tau\lambda_{k,n}^2)}\le (t_j-t_i)^{\alpha_1}\lambda_{k,n}^{2\alpha_1},
\end{align*}
which together with \eqref{eq:smooth-dis} indicates that for any $\alpha\in(0,1-2\epsilon)$ and $0<\epsilon\ll 1$,
\begin{align*}
&\quad\ \vert (-\Delta_n)^\mu G^n_{t_j-r+\tau}(x,z)-(-\Delta_n)^\mu G^n_{t_i-r+\tau}(x,z)\vert \\
&\le C\sum_{k=1}^{n-1}(-\lambda_{k,n})^{\mu}(1+\tau\lambda_{k,n}^2)^{-\lfloor \frac{t_i-r+\tau}{\tau}\rfloor}(t_j-t_i)^{\frac{\alpha}{4}}\lambda_{k,n}^{\frac{\alpha}{2}}\\
&\le C (t_i-r)^{-\frac{1}{2}(\mu+\frac{1}{2}+\epsilon+\frac{\alpha}{2})}(t_j-t_i)^{\frac{\alpha}{4}}
\end{align*}
for $0\le i< j\le m$.
On this basis, we arrive at
\begin{gather*}
\int_0^{t_i}\int_{\mathcal O}\vert (-\Delta_n)^\mu G^n_{t_j-r+\tau}(x,z)-(-\Delta_n)^\mu G^n_{t_i-r+\tau}(x,z)\vert ^{2-\mu}\ud z\ud r\le C\vert t_j-t_i\vert ^{{\frac{\alpha}{4}}(2-\mu)}.
\end{gather*}
Again by using \eqref{eq:smooth-dis} with $\gamma=\mu+\frac12+\epsilon$,
$$|(-\Delta_n)^\mu G^n_{t_j-r+\tau}(x,z)|\le C\sum_{k=1}^{n-1}(-\lambda_{k,n})^{\mu}(1+\tau\lambda_{k,n}^2)^{-\lfloor\frac{t_j-r+\tau}{\tau}\rfloor}\le C(t_j-r)^{-\frac{1}{2}(\mu+\frac12+\epsilon)}.$$
Consequently, it holds that for any $0\le i<j\le m$,
\begin{align*}
\int_{t_i}^{t_j}\int_{\mathcal O}\vert (-\Delta_n)^\mu G^n_{t_j-r+\tau}(x,z)\vert ^{2-\mu}\ud z\ud r
\le C\vert t_j-t_i\vert ^{\frac{\alpha}{4}(2-\mu)},
\end{align*}
since $1-\frac{1}{2}(\mu+\frac12+\epsilon)(2-\mu)\ge \frac{\alpha}{4}(2-\mu)$ for $\alpha\in(0,1-2\epsilon)$ and $\mu\in\{0,1\}$.

By means of \eqref{eq:fully-dis-u0} and $u_0\in\mathcal{C}^2(\OO)$, for any $0\le i<j\le m$,
\begin{align*}
&\quad\ \left|\int_{\mathcal O}G^{n,\tau}_{t_i}(x,z)u_0(\kappa_n(z))\mathrm{d} z-\int_{\mathcal O}G^{n,\tau}_{t_j}(y,z)u_0(\kappa_n(z))\mathrm{d} z\right|\\
&\le
|\tilde u^n(0,x)-\tilde u^n(0,y)|+
\int_0^{t_i}\int_{\mathcal{O}}|\Delta_nG^{n,\tau}_{r+\tau}(x,z)-\Delta_nG^{n,\tau}_{r+\tau}(y,z)||\Delta_nu_0(\kappa_n(z))|\ud z\ud r\\
&\quad+\int_{t_i}^{t_j}\int_{\mathcal{O}}|\Delta_nG^{n,\tau}_{r+\tau}(x,z)||\Delta_nu_0(\kappa_n(z))|\ud z\ud r\le C(|t_j-t_i|^{\frac{\alpha}{4}}+|x-y|^\alpha).
\end{align*}
In addition, by virtue of \eqref{eq:SD} and \eqref{eq:SS}, it follows from \eqref{eq:untau} and \eqref{eq:u^nbound-dis} that \eqref{eq:Holderuntau} holds for all $s=t_i$ and $t=t_j$ with $0\le i<j\le m$. Thanks to the triangle inequality, it suffices to prove  \eqref{eq:Holderuntau} for the following two cases.

\textbf{Case 1: $t=s$ and $x\neq y$.} By the definition of $u^{n,\tau}(t,x)$,
\begin{align*}
&\quad\ \|u^{n,\tau}(t,x)-u^{n,\tau}(t,y)\|_{L^p(\Omega)}\le \frac{\eta_\tau(t)+\tau-t}{\tau}\|u^{n,\tau}(\eta_\tau(t),x)-u^{n,\tau}(\eta_\tau(t),y)\|_{L^p(\Omega)}\\
&\quad+\frac{t-\eta_\tau(t)}{\tau}\|u^{n,\tau}(\eta_\tau(t)+\tau,x)-u^{n,\tau}(\eta_\tau(t)+\tau,y)\|_{L^p(\Omega)}\le C|x-y|^{\alpha}.
\end{align*}

\textbf{Case 2: $x=y$ and $t\ge s$.}  If $\eta_\tau(t)=\eta_\tau(s)$, then
\begin{align}\label{eq:Case2}\notag
&\quad\ \|u^{n,\tau}(t,x)-u^{n,\tau}(s,x)\|_{L^p(\Omega)}\\
&\le C\tau^{-1}(t-s)\|u^{n,\tau}(\eta_\tau(t)+\tau,x)-u^{n,\tau}(\eta_{\tau}(t),x)\|_{L^p(\Omega)}\le C(t-s)^{\frac{\alpha}{4}}.
\end{align}
If $\eta_\tau(t)\ge\eta_\tau(s)+\tau$, then
based on \eqref{eq:Case2},
\begin{align*}
&\quad\ \|u^{n,\tau}(t,x)-u^{n,\tau}(s,x)\|_{L^p(\Omega)}\\
&\le
\|u^{n,\tau}(t,x)-u^{n,\tau}(\eta_{\tau}(t),x)\|_{L^p(\Omega)}+\|u^{n,\tau}(\eta_{\tau}(t),x)-u^{n,\tau}(\eta_{\tau}(s)+\tau,x)\|_{L^p(\Omega)}\\
&\quad+\|u^{n,\tau}(\eta_{\tau}(s)+\tau,x)-u^{n,\tau}(s,x)\|_{L^p(\Omega)}\\
&\le C(t-\eta_{\tau}(t))^{\frac{\alpha}{4}}+C(\eta_{\tau}(t)-\eta_{\tau}(s)-\tau)^{\frac{\alpha}{4}}+C(\eta_{\tau}(s)+\tau-s)^{\frac{\alpha}{4}}\le C(t-s)^{\frac{\alpha}{4}}.
\end{align*}
The proof is completed.
\end{proof}

\begin{remark}\label{Rem1}
Based on the standard Picard argument, it can be verified that when the coefficients $f$ and $\sigma$ satisfy the global Lipschitz condition, the stochastic Cahn--Hilliard equation \eqref{CH} admits a unique mild solution $u=\{u(t,x),(t,x)\in[0,T]\times\OO\}$
satisfying
$\sup_{(t,x)\in[0,T]\times\mathcal{O}}\E[\vert u(t,x)\vert^p]\le C(p,T).$
Since the discussions in subsections \ref{S4.1} and \ref{S5.1} are mainly based on the properties of the Green function $G$ and the discrete Green functions $G^n$ and $G^{n,\tau}$, we  remark that
 Lemmas \ref{Holder-exact} and \ref{u2u3} as well as Propositions \ref{utilde-uh-1} and \ref{eq:un-untau} are valid as well when the coefficients $f$ and $\sigma$ satisfy the global Lipschitz condition.

\end{remark}

\subsection{Error estimate between $\tilde u^{n,\tau}$ and $u^{n,\tau}$}\label{S5.2}
This part presents the error estimate between the fully discrete numerical solution $u^{n,\tau}$ and the auxiliary process $\tilde u^{n,\tau}$. As in subsection \ref{S4.2}, this will be accomplished by estimating $E^{i}:=\tilde{U}^i-U^i$. 

\begin{proposition}\label{H-1E-dis}
Let $u_0\in\mathcal{C}^3(\OO)$ and $0<\epsilon\ll 1$. Then 
there exists some constant $C=C(T,\epsilon)$ such that for any $i\in\mathbb{Z}_{m}$,
\begin{align*}
\E\left[\|(-A_n)^{-\frac12}E^{i}\|_{l_n^2}^4\right]+\E\bigg[\Big|\tau\sum_{j=0}^{i-1} \|(-A_n)^{\frac12} E^{j+1}\|_{l_n^2}^2\Big|^2\bigg]
\le C\tau^{\frac32-4\epsilon}.
\end{align*}
\end{proposition}
\begin{proof}
By subtracting
\eqref{eq:Ui} from \eqref{eq:tildeUi}, 
\begin{align}\label{eq:Ei+1-Ei}
&\quad\ E^{i+1}-E^i+\tau A_n^2 E^{i+1}\\\notag
&=\tau A_n\{F_n(U(t_{i+1}))-F_n(U^{i+1})\}+\sqrt{n/\pi}\{\Sigma_n(\tilde U^i)-\Sigma_n(U^i)\}(\beta_{t_{i+1}}-\beta_{t_{i}}).
\end{align}
Then
applying $\langle \cdot,(-A_n)^{-1}E^{i+1}\rangle$ on both sides of \eqref{eq:Ei+1-Ei}, one has
\begin{align*}
 &\quad\ \langle E^{i+1}-E^i,(-A_n)^{-1}E^{i+1}\rangle +\tau \|(-A_n)^{\frac12} E^{i+1}\|^2\\
&=\tau\langle F_n(U^{i+1})-F_n(\tilde U^{i+1}),E^{i+1}\rangle +\tau\langle F_n(\tilde U^{i+1})-F_n(U(t_{i+1})),E^{i+1}\rangle \\
&\quad+\sqrt{n/\pi}\langle \{\Sigma_n(\tilde U^i)-\Sigma_n(U^i)\}(\beta_{t_{i+1}}-\beta_{t_{i}}),(-A_n)^{-1}(E^{i+1}-E^{i})\rangle\\
&\quad+\sqrt{n/\pi}\langle \{\Sigma_n(\tilde U^i)-\Sigma_n(U^i)\}(\beta_{t_{i+1}}-\beta_{t_{i}}),(-A_n)^{-1}E^{i}\rangle,\quad i\in\mathbb{Z}_{m-1}^0.
\end{align*}
Utilizing the identity \eqref{BE},
the Young inequality and \eqref{fbfa}, it holds that
\begin{align}\label{eq:U-U^j}
&\quad\ \frac{1}{2}\|(-A_n)^{-\frac12}E^{i+1}\|^2-\frac{1}{2}\|(-A_n)^{-\frac12}E^{i}\|^2
+\tau \|(-A_n)^{\frac12} E^{i+1}\|^2\\\notag
&\le \frac32\tau \|E^{i+1}\|^2
+\frac12\tau\|F_n(\tilde U^{i+1})-F_n(U(t_{i+1}))\|^2 \\\notag
&\quad+n/\pi\| (-A_n)^{-\frac12}\{\Sigma_n(\tilde U^i)-\Sigma_n(U^i)\}(\beta_{t_{i+1}}-\beta_{t_{i}})\|^2\\\notag
&\quad+\sqrt{n/\pi}\langle (-A_n)^{-\frac12}\{\Sigma_n(\tilde U^i)-\Sigma_n(U^i)\}(\beta_{t_{i+1}}-\beta_{t_{i}}),(-A_n)^{-\frac12}E^i\rangle.
\end{align}

\textit{Step 1:}
Multiplying \eqref{eq:U-U^j} by $\|(-A_n)^{\frac12}E^{i+1}\|^2$ and using the identity \eqref{BE} with $x=\|(-A_n)^{-\frac12}E^{i+1}\|^2$ and $y=\|(-A_n)^{-\frac12}E^{i}\|^2$, we obtain
\begin{align*}
&\quad\ \frac{1}{4}\|(-A_n)^{-\frac12}E^{i+1}\|^4-\frac{1}{4}\|(-A_n)^{-\frac12}E^{i}\|^4+\frac14\left(\|(-A_n)^{-\frac12}E^{i+1}\|^2-\|(-A_n)^{-\frac12}E^{i}\|^2\right)^2\\
&\quad 
+\tau \|(-A_n)^{\frac12} E^{i+1}\|^2\|(-A_n)^{-\frac12}E^{i+1}\|^2\\
&\le  \frac32\tau \|E^{i+1}\|^2\|(-A_n)^{-\frac12}E^{i+1}\|^2
+\frac12\tau\|F_n(\tilde U^{i+1})-F_n(U(t_{i+1}))\|^2\|(-A_n)^{-\frac12}E^{i+1}\|^2 \\\notag
&\quad+n/\pi\| (-A_n)^{-\frac12}\{\Sigma_n(\tilde U^i)-\Sigma_n(U^i)\}(\beta_{t_{i+1}}-\beta_{t_{i}})\|^2\|(-A_n)^{-\frac12}E^{i+1}\|^2\\\notag
&\quad+\sqrt{n/\pi}\langle (-A_n)^{-\frac12}\{\Sigma_n(\tilde U^i)-\Sigma_n(U^i)\}(\beta_{t_{i+1}}-\beta_{t_{i}}),(-A_n)^{-\frac12}E^i\rangle\|(-A_n)^{-\frac12}E^{i+1}\|^2\\
&\le \frac14\tau\|(-A_n)^{\frac12} E^{i+1}\|^2\|(-A_n)^{-\frac12} E^{i+1}\|^2+C\tau\|(-A_n)^{-\frac12} E^{i+1}\|^4\\
&\quad+C\tau\|F_n(\tilde U^{i+1})-F_n(U(t_{i+1}))\|^4\\
&\quad
+n/\pi\| (-A_n)^{-\frac12}\{\Sigma_n(\tilde U^i)-\Sigma_n(U^i)\}(\beta_{t_{i+1}}-\beta_{t_{i}})\|^2\|(-A_n)^{-\frac12}E^{i+1}\|^2\\
&\quad+\sqrt{n/\pi}\langle (-A_n)^{-\frac12}\{\Sigma_n(\tilde U^i)-\Sigma_n(U^i)\}(\beta_{t_{i+1}}-\beta_{t_{i}}),(-A_n)^{-\frac12}E^i\rangle\|(-A_n)^{-\frac12}E^{i}\|^2\\
&\quad+Cn\big|\langle (-A_n)^{-\frac12}\{\Sigma_n(\tilde U^i)-\Sigma_n(U^i)\}(\beta_{t_{i+1}}-\beta_{t_{i}}),(-A_n)^{-\frac12}E^i\rangle\big|^2\\
&\quad+\frac{1}{8}\left(\|(-A_n)^{-\frac12}E^{i+1}\|^2-\|(-A_n)^{-\frac12}E^{i}\|^2\right)^2,
\end{align*}
where in the last step we applied the following inequality  (with $\beta=\frac16$)
\begin{align}\label{eq:E2Yong}
 \|E^{i+1}\|^2\le \beta\|(-A_n)^{\frac12}E^{i+1}\|^2+\frac{1}{4\beta}\|(-A_n)^{-\frac12}E^{i+1}\|^2\quad\forall~\beta>0.
\end{align} 
Consequently, one has
\begin{align*}
&\quad\ \frac{1}{4}\|(-A_n)^{-\frac12}E^{i+1}\|^4-\frac{1}{4}\|(-A_n)^{-\frac12}E^{i}\|^4
+\frac{3 }{4}\tau\|(-A_n)^{\frac12} E^{i+1}\|^2\|(-A_n)^{-\frac12}E^{i+1}\|^2\\
&\le C(\varepsilon)\tau\|(-A_n)^{-\frac12} E^{i+1}\|^4+C(\varepsilon)\tau\|(-A_n)^{-\frac12} E^{i}\|^4+C\tau\|F_n(\tilde U^{i+1})-F_n(U(t_{i+1}))\|^4\\
&\quad+\sqrt{n/\pi}\langle (-A_n)^{-\frac12}\{\Sigma_n(\tilde U^i)-\Sigma_n(U^i)\}(\beta_{t_{i+1}}-\beta_{t_{i}}),(-A_n)^{-\frac12}E^i\rangle\|(-A_n)^{-\frac12}E^{i}\|^2\\
&\quad+\varepsilon n^2\tau^{-1}\| (-A_n)^{-\frac12}\{\Sigma_n(\tilde U^i)-\Sigma_n(U^i)\}(\beta_{t_{i+1}}-\beta_{t_{i}})\|^4,
\end{align*}
where $0<\varepsilon\ll 1$ is to be determined.
Since $E^0=0$, for any $i\in\mathbb{Z}_{m}$,
\begin{align*}
&\quad\ \frac{1}{4}\E\left[\|(-A_n)^{-\frac12}E^{i}\|_{l_n^2}^4\right]+ \frac{3 }{4}\tau\sum_{j=0}^{i-1}\E\left[\|(-A_n)^{\frac12} E^{j+1}\|_{l_n^2}^2\|(-A_n)^{-\frac12}E^{j+1}\|_{l_n^2}^2\right]\\\notag
&\le  C_\varepsilon\tau\sum_{j=0}^{i-1}\E\left[\|(-A_n)^{-\frac12} E^{j+1}\|_{l_n^2}^4\right]+C\tau\sum_{j=0}^{i-1}\E\left[\|F_n(\tilde U^{j+1})-F_n(U(t_{j+1}))\|_{l_n^2}^4\right]\\\notag
&\quad +\varepsilon \tau^{-1}\sum_{j=0}^{i-1}\E\left[\| (-A_n)^{-\frac12}\{\Sigma_n(\tilde U^j)-\Sigma_n(U^j)\}(\beta_{t_{j+1}}-\beta_{t_{j}})\|^4\right].
\end{align*}
Using the similar arguments for proving \eqref{ASigmaUF} and \eqref{Sigma-sigma}, it can be verified that
\begin{align}\label{eq:AnE}
\| (-A_n)^{-\frac12}\{\Sigma_n(\tilde U^j)-\Sigma_n(U^j)\}\|^2_{\mathrm{F}}\le C(\sigma)n^{-1}\|E^j\|^2\le C_1(\sigma)\|E^j\|_{l_n^2}^2.
\end{align}
Further, thanks to the Burkholder--Davis--Gundy inequality and H\"older inequality,
\begin{align}\label{eq:E2dis}
&\quad\ \E\left[\| (-A_n)^{-\frac12}\{\Sigma_n(\tilde U^j)-\Sigma_n(U^j)\}(\beta_{t_{j+1}}-\beta_{t_{j}})\|^4\right]\\\notag
&\le C\tau^2\E\left[\| (-A_n)^{-\frac12}\{\Sigma_n(\tilde U^j)-\Sigma_n(U^j)\}\|_{\mathrm{F}}^4\right]
\le C_2(\sigma)\tau^{2}\E\left[\|E^j\|_{l_n^2}^4\right].
\end{align}
Similarly to \eqref{F-F} and \eqref{Kt}, it follows from Proposition \ref{eq:un-untau} that for any $p\ge1$,
\begin{align}\label{Eq1}
&\quad\ \|F_n(\tilde U^{j+1})-F_n(U(t_{j+1}))\|^2_{L^{2p}(\Omega,l_n^2)}\\\nonumber
&\le C\big(1+\Vert \tilde U^{j+1}\Vert _{L^{8p}(\Omega;l_n^8)}^{4}+\Vert U(t_{j+1})\Vert _{L^{8p}(\Omega;l_n^8)}^{4}\big)\Vert \tilde U^{j+1}-U(t_{j+1})\Vert _{L^{4p}(\Omega;l^{4}_n)}^2\\\nonumber
&\le C\tau^{\frac34-2\epsilon}.
\end{align}
As a consequence, we derive that
 \begin{align*}
&\quad\ \frac{1}{4}\E\left[\|(-A_n)^{-\frac12}E^{i}\|_{l_n^2}^4\right]+ \frac{3 }{4}\tau\sum_{j=0}^{i-1}\E\left[\|(-A_n)^{\frac12} E^{j+1}\|_{l_n^2}^2\|(-A_n)^{-\frac12}E^{j+1}\|_{l_n^2}^2\right]\\\notag
&\le  C_\varepsilon\tau\sum_{j=0}^{i-1}\E\left[\|(-A_n)^{-\frac12} E^{j+1}\|_{l_n^2}^4\right]+\varepsilon C_2(\sigma)\tau\sum_{j=0}^{i-1}\E\left[\|E^{j}\|_{l_n^2}^4\right]+C\tau^{\frac32-4\epsilon}\\\notag
&\le  C_\varepsilon\tau\sum_{j=0}^{i-1}\E\left[\|(-A_n)^{-\frac12} E^{j+1}\|_{l_n^2}^4\right]\\\notag
&\quad+\varepsilon C_2(\sigma)\tau\sum_{j=0}^{i-1}\E\left[\|(-A_n)^{\frac12} E^{j+1}\|^2_{l_n^2}\|(-A_n)^{-\frac12}E^{j+1}\|_{l_n^2}^2\right]+C\tau^{\frac32-4\epsilon}.
\end{align*}
Hence by  choosing $\varepsilon$ small enough so that $C_2(\sigma)\varepsilon\le \frac12$, one could use the discrete Gronwall inequality to obtain
\begin{small}
\begin{equation}\label{eq:H1step-dis}
\E\left[\|(-A_n)^{-\frac12}E^{i}\|_{l_n^2}^4\right]+ \tau\sum_{j=0}^{i-1}\E\left[\|(-A_n)^{\frac12} E^{j+1}\|_{l_n^2}^2\|(-A_n)^{-\frac12}E^{j+1}\|_{l_n^2}^2\right]\le C\tau^{\frac32-4\epsilon}.
\end{equation}
\end{small}

\textit{Step 2:}
Based on \eqref{eq:U-U^j}, we utilize $E^0=0$ and \eqref{eq:E2Yong} with $\beta=\frac13$ to deduce
\begin{align}\label{eq:H-1dis}
&\quad\ \frac{1}{2}\|(-A_n)^{-\frac12}E^{i}\|_{l_n^2}^2+\frac{1}{2}\tau\sum_{j=0}^{i-1} \|(-A_n)^{\frac12} E^{j+1}\|_{l_n^2}^2\\\notag
&\le C\tau\sum_{j=0}^{i-1} \|(-A_n)^{-\frac12} E^{j+1}\|_{l_n^2}^2+
\tau\sum_{j=0}^{i-1}\|F_n(\tilde U^{j+1})-F_n(U(t_{j+1}))\|_{l_n^2}^2 \\\notag
&\quad+\frac{n}{\pi}\sum_{j=0}^{i-1}\| (-A_n)^{-\frac12}\{\Sigma_n(\tilde U^j)-\Sigma_n(U^j)\}(\beta_{t_{j+1}}-\beta_{t_{j}})\|_{l_n^2}^2\\\notag
&\quad+\sqrt{\frac{\pi}{n}}\sum_{j=0}^{i-1}\langle \{\Sigma_n(\tilde U^j)-\Sigma_n(U^j)\}(\beta_{t_{j+1}}-\beta_{t_{j}}),(-A_n)^{-1}E^j\rangle.
\end{align}
In order to estimate the last term, let us introduce 
\begin{align*}
M^i:=\sqrt{\frac{\pi}{n}}\int_0^{t_i}\langle (-A_n)^{-1}E^{\lfloor \frac s\tau \rfloor},\{\Sigma_n(\tilde U^{\lfloor \frac s\tau \rfloor})-\Sigma_n(U^{\lfloor \frac s\tau \rfloor})\}\ud\beta_s\rangle,\quad i\in\mathbb{Z}_{m}^{0}.
\end{align*}
Note that $\{M^i\}_{i\in\mathbb{Z}_{m}^0}$ is a discrete martingale and satisfies
\begin{align*}
\E\left[|M^i|^2\right]&\le Cn^{-1}\int_0^{t_i}\E\left[\|(E^{\lfloor \frac s\tau \rfloor})^\top (-A_n)^{-1}\{\Sigma_n(\tilde U^{\lfloor \frac s\tau \rfloor})-\Sigma_n(U^{\lfloor \frac s\tau \rfloor})\}\|^2\right]\ud s\\
&\le Cn^{-2}\int_0^{t_i}\E\left[\|(-A_n)^{-\frac12}E^{\lfloor \frac s\tau \rfloor}\|^2\| E^{\lfloor \frac s\tau \rfloor}\|^2\right]\ud s\\
&\le \frac34\tau\sum_{j=0}^{i-1}\E\left[\|(-A_n)^{\frac12}E^{j}\|_{l_n^2}^2\|(-A_n)^{-\frac12} E^{j}\|_{l_n^2}^2\right]+C\tau\sum_{j=0}^{i-1}\E\left[\|(-A_n)^{-\frac12}E^{j}\|_{l_n^2}^4\right],
\end{align*}
due to \eqref{eq:AnE} and \eqref{eq:E2Yong}. 
Besides, taking \eqref{eq:E2dis} into account, 
\begin{align*}
&\quad\ \E\bigg[\Big|\sum_{j=0}^{i-1}\frac{n}{\pi}\| (-A_n)^{-\frac12}\{\Sigma_n(\tilde U^j)-\Sigma_n(U^j)\}(\beta_{t_{j+1}}-\beta_{t_{j}})\|_{l_n^2}^2\Big|^2\bigg]
\\
&\le C\tau\sum_{j=0}^{i-1}\E\left[\|E^j\|_{l_n^2}^4\right]\le C\tau\sum_{j=0}^{i-1}\E\left[\|(-A_n)^{\frac12} E^{j}\|_{l_n^2}^2\|(-A_n)^{-\frac12}E^{j}\|_{l_n^2}^2\right].
\end{align*}
Taking second order moments on both sides of \eqref{eq:H-1dis}, then \eqref{Eq1}--\eqref{eq:H1step-dis} yield
\begin{align*}
&\quad\ \frac{1}{4}\E\left[\|(-A_n)^{-\frac12}E^{i}\|_{l_n^2}^4\right]+\frac{1}{4}\E\bigg[\Big|\tau\sum_{j=0}^{i-1} \|(-A_n)^{\frac12} E^{j+1}\|_{l_n^2}^2\Big|^2\bigg]\\
&\le C\tau\sum_{j=0}^{i-1}\E\left[\|(-A_n)^{-\frac12} E^{j+1}\|_{l_n^2}^4\right]+
C\tau^{\frac32-4\epsilon}.
\end{align*}
 Finally, applying the discrete Gronwall inequality leads to the desired result.
\end{proof}

\begin{theorem}\label{eq:untau-untau}
Let $u_0\in\mathcal{C}^3(\OO)$, $0<\epsilon\ll1$, and $\zeta\in[1,2)$. Then there exists some constant $C=C(\zeta,T,\epsilon)$ such that for any $i\in\mathbb{Z}_m$ and $x\in\OO$,
\begin{align*}
\E\left[\vert u^{n,\tau}(t_i,x)-u^{n}(t_i,x)\vert ^{\zeta}\right]\le C\tau^{(\frac{3}{8}-\epsilon)\zeta}.
\end{align*}
\end{theorem}
\begin{proof}

For $i\in\mathbb{Z}_{m}$ and $j=0,1,\ldots,i-1$, we denote
\begin{align*}
&K_\mu^{i,j}:=(-A_n)^{1+\mu}(I+\tau A_n^2)^{-(i-j)}\{F_n(U(t_{j+1}))-F_n(U^{j+1})\},\\
&L_\mu^{i,j}:=\sqrt{n/\pi}(-A_n)^\mu(I+\tau A_n^2)^{-(i-j)}\{\Sigma_n(\tilde U^j)-\Sigma_n(U^j)\}.
\end{align*}
It then follows from \eqref{eq:Ei+1-Ei} and $E^0=0$ that for any $\mu\in[0,1]$,
\begin{align*}
(-A_n)^\mu E^{i}
=-\sum_{j=0}^{i-1}\tau K_\mu^{i,j}+\sum_{j=0}^{i-1}L_\mu^{i,j}(\beta_{t_{j+1}}-\beta_{t_{j}}).
\end{align*}
Repeating the proof of \eqref{J2mu}, it can be shown that for any $\mu\in[0,\frac12]$ and $0<\varepsilon\ll 1$,
\begin{align*}
\|L_\mu^{i,j}\|_{\mathrm{F}}^2&\le n\sum_{l=1}^{n-1}(-\lambda_{l,n})^{2\mu}(1+\tau \lambda_{l,n}^2)^{-2(i-j)}\|\{\Sigma_n(\tilde U^j)-\Sigma_n(U^j)\}e_l\|^2\\
&\le Ct_{i-j}^{-\frac14-\mu-\epsilon}\|\tilde U^j-U^j\|^2=Ct_{i-j}^{-\frac14-\mu-\varepsilon}\|E^j\|^2,
\end{align*}
where we used \eqref{eq:smooth-dis} with $\gamma=\frac{1}{2}+2\mu+2\varepsilon$.
Using the Burkholder inequality and
gathering the above estimates give that for any $p\ge1/2$,
\begin{equation}\label{eq:EnH1-dis}
\E[\|(-A_n)^\mu E^{i}\|_{l_n^2}^{2p}]\le C\E[|\sum_{j=0}^{i-1}\tau \|K_\mu^{i,j}\|_{l_n^2}|^{2p}]+C\E[|\sum_{j=0}^{i-1}\tau t_{i-j}^{-\frac14-\mu-\varepsilon}\|E^j\|_{l_n^2}^2|^{p}].
\end{equation}
\textit{Step 1: In this step, we take $\mu=0$ and estimate $\|E^i\|_{l_n^2}$.}

By Lemma \ref{FaFb} and \eqref{eq:smooth-dis}, for any $j=0,1,\ldots,i-1$,
\begin{align*}
\|K_0^{i,j}\|_{l_n^2}&\le\| (-A_n)^{\frac12}(I+\tau A_n^2)^{-(i-j)}(-A_n)^{\frac12}\{F_n(\tilde U^{j+1})-F_n(U^{j+1})\}\|_{l_n^2}\\
&\quad+\| (-A_n)(I+\tau A_n^2)^{-(i-j)}\{F_n(U(t_{j+1}))-F_n(\tilde U^{j+1})\}\|_{l_n^2}\\
&\le Ct_{i-j}^{-\frac14}\big(1+\Vert (-A_n)^{\frac{1}{2}}\tilde U^{j+1}\Vert ^2_{l_n^2}+\Vert (-A_n)^{\frac{1}{2}}U^{j+1}\Vert ^2_{l_n^2}\big)\Vert (-A_n)^{\frac{1}{2}}E^{j+1}\Vert _{l_n^2}\\
&\quad+Ct_{i-j}^{-\frac12}\big(1+\Vert U(t_{j+1})\Vert ^2_{l_n^\infty}+\Vert \tilde U^{j+1}\Vert ^2_{l_n^\infty}\big)\Vert U(t_{j+1})-\tilde U^{j+1}\Vert _{l_n^2},
\end{align*}
where the second term on the right hand side is handled in the same way as in \eqref{F-F} with $q=1$. 
Then
one can apply Proposition \ref{H-1E-dis} to derive that
\begin{align*}
&\quad\ \E\bigg[\Big|\sum_{j=0}^{i-1}\tau \|K_0^{i,j}\|_{l_n^2}\Big|^{2}\bigg]\\
&\le C\tau^{\frac34-2\epsilon}\sum_{j=0}^{i-1}\tau t_{i-j}^{-\frac12}\big(1+\Vert (-A_n)^{\frac{1}{2}}\tilde U^{j+1}\Vert ^4_{L^8(\Omega;l_n^2)}+\Vert (-A_n)^{\frac{1}{2}}U^{j+1}\Vert ^4_{L^8(\Omega;l_n^2)}\big)\\
& + C\Big|\sum_{j=0}^{i-1}\tau t_{i-j}^{-\frac12}\big(1+\Vert U(t_{j+1})\Vert ^2_{L^8(\Omega;l_n^\infty)}+\Vert \tilde U^{j+1}\Vert ^2_{L^8(\Omega;l_n^\infty)}\big)\Vert U(t_{j+1})-\tilde U^{j+1}\Vert _{L^4(\Omega;l_n^2)}\Big|^{2},
\end{align*}
which can be further bounded by $ C\tau^{\frac34-2\epsilon}$, due to Proposition \ref{eq:un-untau}, \eqref{eq:u^nbound}, \eqref{eq:u^nbound-dis} and Lemma \ref{lem:tildeU}.
Plugging this estimate into \eqref{eq:EnH1-dis} with $\mu=0$ and $p=1$ yields
 $$\E[\| E^{i}\|_{l_n^2}^{2}]
\le C\tau^{\frac34-2\epsilon}
+C\sum_{j=0}^{i-1}\tau t_{i-j}^{-\frac14-\varepsilon}\E[\|E^j\|_{l_n^2}^2],$$
which along with the discrete Gronwall inequality  gives
$ \E[\| E^{i}\|_{l_n^2}^{2}]\le C\tau^{\frac34-2\epsilon}$. Hence taking Proposition \ref{eq:un-untau} into account yields
\begin{align}\label{eq:Uti-Ui-l2} \E[\| U(t_i)-U^i\|_{l_n^2}^{2}]\le C\tau^{\frac34-2\epsilon}\quad\forall~i\in\mathbb{Z}^0_{m}.
\end{align}

\textit{Step 2: In this step, we take $\mu=\frac12$ and estimate $\|(-A_n)^{\frac12}E^i\|_{l_n^2}$.}

In view of \eqref{eq:smooth-dis} and a similar argument of \eqref{F-F} (with $q=1$),
\begin{align*}
\|K_{\frac12}^{i,j}\|_{l_n^2}\le Ct_{i-j}^{-\frac34}(1+\|U(t_{j+1})\|_{l_n^\infty}^2+\|U^{j+1}\|_{l_n^\infty}^2)\|U(t_{j+1})-U^{j+1}\|_{l_n^2}.
\end{align*}
By the H\"older inequality, \eqref{eq:u^nbound} and \eqref{eq:u^nbound-dis}, for any $q\in[1,2)$, 
\begin{align*}
\|K_{\frac12}^{i,j}\|_{L^q(\Omega;l_n^2)}\le C\|U(t_{j+1})-U^{j+1}\|_{L^2(\Omega;l_n^2)}\le  C\tau^{\frac38-\epsilon},
\end{align*}
in view of \eqref{eq:Uti-Ui-l2}.
This in combination with \eqref{eq:EnH1-dis} indicates that for any $p\in[\frac12,1)$,
\begin{align*}
\E\left[\|(-A_n)^{\frac12} E^{i}\|_{l_n^2}^{2p}\right]&\le C\Big|\sum_{j=0}^{i-1}\tau \|K_{\frac12}^{i,j}\|_{L^{2p}(\Omega;l_n^2)}\Big|^{2p}+\E\bigg[\Big|\sum_{j=0}^{i-1}\tau t_{i-j}^{-\frac34-\varepsilon}\|E^j\|_{l_n^2}^2\Big|^{p}\bigg]\\
&\le C\tau^{(\frac34-2\epsilon)p}+\bigg(\sum_{j=0}^{i-1}\tau t_{i-j}^{-\frac34-\varepsilon}\E\left[\|E^j\|_{l_n^2}^2\right]\bigg)^{p}\le C\tau^{(\frac34-2\epsilon)p}.
\end{align*}
In this way,
we obtain the required result,
thanks to \eqref{l2H1} and Proposition \ref{eq:un-untau}.
\end{proof}

\section{Convergence of density}\label{S6}

Given $d\in\mathbb{N}_+$, for two $\mathbb{R}^{d}$-valued random variables $X,Y$, we write $\mathrm{d}_{\mathrm{TV}}(X,Y)$ to indicate the total variation distance between $X$ and $Y$, i.e.,
$$\mathrm{d}_{\mathrm{TV}}(X,Y)=2\sup_{A\in\mathscr B(\mathbb{R}^d)}\{\vert \mathbb P(X\in A)-\mathbb P(Y\in A)\vert \}=\sup_{\phi\in\Phi}\vert \E[\phi(X)]-\E[\phi(Y)]\vert ,$$
where $\Phi$ is the set of continuous functions $\phi:\mathbb{R}^d\rightarrow\mathbb{R}$ which are bounded by $1$, and $\mathscr B(\mathbb{R}^d)$ is the Borel $\sigma$-algebra of $\mathbb{R}^d$. Furthermore, if $\{X_n\}_{n\ge1}$ and $X_\infty$ have the densities $\{p_{X_n}\}_{n\ge 1}$ and $p_{X_\infty}$ respectively, then 
\begin{align}\label{pX-PX}
\mathrm{d}_{\mathrm{TV}}(X_n,X_\infty)
=\Vert p_{X_n}-p_{X_\infty}\Vert _{L^1(\mathbb{R}^d)}.
\end{align}

We now present a criterion for reducing the  total variation distance of  random variables to that of their localizations, which will be applied to prove the density convergence  in $L^1(\mathbb{R})$ for the numerical discretizations.
\begin{proposition}\label{prop:localization}
Let $\mathbb{T}:=\Pi_{i=1}^{d_1}[a_i,b_i]$ ($a_i<b_i$) be an interval or a rectangle in $\mathbb{R}^{d_1}$ and
 $\mathbf{X}=\{\mathbf{X}(\mathrm{t}),\mathrm{t}\in\mathbb{T} \}$  an $\mathbb{R}^{d}$-valued random field defined on $(\Omega,\mathscr{F},\mathbb{P})$ with continuous trajectories a.s.  For every $R\ge1$, denote 
$\mathbf{\Omega}_R:=\big\{\omega\in\Omega:\sup_{\mathrm{t}\in\mathbb{T}}\Vert \mathbf{X}(t,\omega)\Vert \le R\big\}.$ Assume that the following conditions \textup{(C1)--(C4)} hold.

\begin{itemize}
\item[(C1)] Assume that
$\{(\mathbf{\Omega}_R,\mathbf{X}_R)\}_{R\ge1}$ is a localization of $\mathbf{X}$, i.e., 
$\mathbf{X}_R=\mathbf{X}$ on $\mathbf{\Omega}_R\subset \Omega$ for every $R\ge1$ and $\lim_{R\to \infty}\mathbb{P}(\mathbf{\Omega}_R)=1$.

\item[(C2)] Given a discretization parameter  $N\in\mathbb{N}_+$, denote by $\mathbf{X}^N$ (resp.\ $\mathbf{X}^N_R$) the a numerical approximation of $\mathbf{X}$ (resp.\ $\mathbf{X}_R$).
Assume that for sufficiently large $N$ and $R$,
$\mathbf{X}_R^N=\mathbf{X}^N$ on $\mathbf{\Omega}_{R,N}:
=\big\{\omega\in\Omega:\sup_{\mathrm{t}\in\mathbb{T}}\Vert \mathbf{X}^N(t,\omega)\Vert \le R\big\}$.

\item[(C3)] There exists $\upsilon_1>0$ and $q>0$ such that $$\E\left[\|\mathbf{X}^N(\mathrm{t})-\mathbf{X}(\mathrm{t})\|^q\right]\le CN^{-q\upsilon_1}\quad \forall~\mathrm{t}\in\mathbb{T}.$$

\item[(C4)] There exists $\upsilon_2>0$ such that for any $p\ge1$, there exists $C=C(p)$ independent of $N$ such that for any $\mathrm{t}_1,\mathrm{t}_2\in\mathbb{T}$,
$$\E\left[\|\mathbf{X}^N(\mathrm{t}_1)-\mathbf{X}^N(\mathrm{t}_2)\|^p\right]+\E\left[\|\mathbf{X}(\mathrm{t}_1)-\mathbf{X}(\mathrm{t}_2)\|^p\right]\le C\|\mathrm{t}_1-\mathrm{t}_2\|^{p\upsilon_2}.$$
\end{itemize}
Then it holds that
\begin{equation}\label{eq:RN}\limsup_{N\to\infty}\mathrm{d}_{\mathrm{TV}}(\mathbf{X}^N(\mathrm{t}),\mathbf{X}(\mathrm{t}))\le \limsup_{R\to \infty}\limsup_{N\to\infty}\mathrm{d}_{\mathrm{TV}}(\mathbf{X}^N_R(\mathrm{t}),\mathbf{X}_R(\mathrm{t})).
\end{equation}
\end{proposition}
\begin{proof}
\textit{Step 1.} By virtue of (C1) and (C2), for any $\phi\in\Phi$ and $R\ge1$,
\begin{align*}
&\quad\big\vert \E[\phi(\mathbf{X}^N(\mathrm{t}))]-\E[\phi(\mathbf{X}(\mathrm{t}))]\big\vert \\
&\le\big\vert \E[\phi(\mathbf{X}^N(\mathrm{t}))
\mathbf{1}_{\mathbf{\Omega}_{R,N}^\textrm{c}}]-\E[\phi(\mathbf{X}(\mathrm{t}))\mathbf{1}_{\mathbf{\Omega}_R^\textrm{c}}]\big\vert +\big\vert \E[\phi(\mathbf{X}^N_R(\mathrm{t}))]-\E[\phi(\mathbf{X}_R(\mathrm{t}))]\big\vert \\
&\quad+\big\vert \E[\phi(\mathbf{X}^N_R(\mathrm{t}))
\mathbf{1}_{\mathbf{\Omega}_{R,n}^\textrm{c}}]-\E[\phi(\mathbf{X}_R(\mathrm{t}))\mathbf{1}_{\mathbf{\Omega}_R^\textrm{c}}]\big\vert \\
&\le2\mathbb P(\mathbf{\Omega}_R^\textrm{c})+\big\vert \E[\phi(\mathbf{X}^N_R(\mathrm{t}))]-\E[\phi(\mathbf{X}_R(\mathrm{t}))]\big\vert +2\mathbb P(\mathbf{\Omega}_{R,N}^\textrm{c}),
\end{align*}
where $\mathbf{1}_A$ denotes the indicator function on the set $A$.
Since $\sup_{\mathrm{t}\in\mathbb{T}}\Vert \mathbf{X}^N(\mathrm{t})\Vert \le \sup_{\mathrm{t}\in\mathbb{T}}\Vert \mathbf{X}^N(\mathrm{t})-\mathbf{X}(\mathrm{t})\Vert +\sup_{\mathrm{t}\in\mathbb{T}}\Vert \mathbf{X}(\mathrm{t})\Vert $, we have $$\mathbb P(\mathbf{\Omega}_{R,N}^\textrm{c})
\le\mathbb P(\mathbf{\Omega}_{R-1}^\textrm{c})+\mathbb P(\sup_{\mathrm{t}\in\mathbb{T}}\Vert \mathbf{X}^N(\mathrm{t})-\mathbf{X}(\mathrm{t})\Vert \ge1).$$
Taking supremum over $\phi\in\Phi$, we obtain
\begin{align}\label{eq:TV}
\mathrm{d}_{\mathrm{TV}}(\mathbf{X}^N(\mathrm{t}),\mathbf{X}(\mathrm{t}))&\le 4\mathbb P(\mathbf{\Omega}_{R-1}^\textrm{c})+2\mathbb P(\sup_{\mathrm{t}\in\mathbb{T}}\Vert \mathbf{X}^N(\mathrm{t})-\mathbf{X}(\mathrm{t})\Vert \ge1) \\\notag
&\quad+\mathrm{d}_{\mathrm{TV}}(\mathbf{X}^N_R(\mathrm{t}),\mathbf{X}_R(\mathrm{t})).
\end{align}

\textit{Step 2.} Let $\delta\in(0,q\upsilon_1/d_1)$ be arbitrarily fixed. 
For $k\in \mathbb{N}$ and $i\in\{1,\ldots,d_1\}$, denote
$\mathrm{s}_i^k:=a_i+k(b_i-a_i)N^{-\delta}$. Then for any $\mathrm{t}=(\mathrm{t}_1,\ldots,\mathrm{t}_{d_1})\in\mathbb{T}$ and  $i\in\{1,\ldots,d_1\}$, there is a unique integer $0\le k_i(\mathrm{t})\le \lfloor N^\delta\rfloor$ such that
 $\mathrm{t}_i\in[\mathrm{s}_i^{k_i(\mathrm{t})},\mathrm{s}_i^{k_i(\mathrm{t})+1}\wedge b_i)$.
Hence
\begin{align*}
\Vert \mathbf{X}^N(\mathrm{t})-\mathbf{X}(\mathrm{t})\Vert^{q}&\le 3^q\Vert \mathbf{X}^N(\mathrm{t})-\mathbf{X}^N(\mathrm{s}_1^{k_1(\mathrm{t})},\ldots,\mathrm{s}_{d_1}^{k_{d_1}(\mathrm{t})})\Vert ^{q}\\
&\quad+ 3^q\Vert \mathbf{X}^N(\mathrm{s}_1^{k_1(\mathrm{t})},\ldots,\mathrm{s}_{d_1}^{k_{d_1}(\mathrm{t})})-\mathbf{X}(\mathrm{s}_1^{k_1(\mathrm{t})},\ldots,\mathrm{s}_{d_1}^{k_{d_1}(\mathrm{t})})\Vert ^{q}\\
&\quad+ 3^q\Vert \mathbf{X}(\mathrm{t})-\mathbf{X}(\mathrm{s}_1^{k_1(t)},\ldots,\mathrm{s}_{d_1}^{k_{d_1}(t)})\Vert ^{q}
\end{align*}
for $q>0$.
As a result, 
\begin{align}\label{XN-X>1}
\mathbb P\left(\sup_{\mathrm{t}\in\mathbb{T}}\Vert \mathbf{X}^N(\mathrm{t})-\mathbf{X}(\mathrm{t})\Vert \ge1\right)\le\sum_{i=1}^3\mathbb P\left(R_i^{N}\ge3^{-(q+1)}\right),
\end{align}
where $R_1^N:=\sum_{\mathrm{s}\in\mathbb{T}_\delta}\Vert \mathbf{X}^N(\mathrm{s})-\mathbf{X}(\mathrm{s})\Vert ^{q}$, and
\begin{gather*}
R_2^N:=\sup_{\mathrm{s}\in\mathbb{T}_\delta}\sup_{\{\mathrm{t}:\|\mathrm{s}-\mathrm{t}\|\le c(\mathbb{T})N^{-\delta}\}}\Vert \mathbf{X}(\mathrm{s})-\mathbf{X}(\mathrm{t})\Vert ^{q},\\
R_3^N:=\sup_{\mathrm{s}\in\mathbb{T}_\delta}\sup_{\{\mathrm{t}:\|\mathrm{s}-\mathrm{t}\|\le c(\mathbb{T})N^{-\delta}\}}\Vert \mathbf{X}^N(\mathrm{s})-\mathbf{X}^N(\mathrm{t})\Vert ^{q}
\end{gather*}
with $c(\mathbb{T}):=(\sum_{i=1}^{d_1}|a_i-b_i|^2)^{\frac12}$ and 
$$\mathbb{T}_{\delta}:=\{\mathrm{t}\in\mathbb{T}~ |\text{ for any }i=1,\ldots,d_1 , \mathrm{t}_{i}=\mathrm{s}_i^{k_i} \text{ for some } 0\le k_i\le  \lfloor N^\delta\rfloor \}.$$
The Markov inequality, (C3) and $q\upsilon_1>d_1\delta$ reveal
\begin{align*}
\mathbb P\big(R_1^{N}\ge3^{-(q+1)}\big)&\le 3^{q+1}\sum_{\mathrm{s}\in\mathbb{T}_\delta}\E\left[\Vert \mathbf{X}^N(\mathrm{s})-\mathbf{X}(\mathrm{s})\Vert ^{q}\right]
\le C N^{-q\upsilon_1+d_1\delta}\to 0,\quad\text{as } N\to\infty.
\end{align*}
By (C4) and  \cite[Theorem C.6]{KD14}, one has that for sufficiently large $p\ge1$,
$$\mathbb{E}\left[\bigg|\sup_{\mathrm{t}_1,\mathrm{t}_2\in\mathbb{T},\mathrm{t}_1\neq\mathrm{t}_2}
\frac{\Vert\mathbf{X}(\mathrm{t}_1)-\mathbf{X}(\mathrm{t}_2)\Vert }{\Vert \mathrm{t}_1-\mathrm{t}_2\Vert ^{\upsilon_2(1-\frac{2d}{p\upsilon_2})}}\bigg|^p\right]\le C_1.$$
Thus the Markov inequality gives that for sufficiently large $p\ge\max\{q,2d/{v_2}\}$, 
\begin{align*} 
\mathbb P\big(R_2^N\ge3^{-(q+1)}\big)&\le 3^{\frac{p}{q}(q+1)}\mathbb{E}\bigg[\sup_{\mathrm{s}\in\mathbb{T}_\delta}\sup_{\{\mathrm{t}:\|\mathrm{s}-\mathrm{t}\|\le c(\mathbb{T})N^{-\delta}\}}\Vert \mathbf{X}(\mathrm{s})-\mathbf{X}(\mathrm{t})\Vert ^{p}\bigg]\\
&\le CN^{-\delta\upsilon_2(1-\frac{2d}{p\upsilon_2})p}\to 0,\quad\text{as } N\to\infty.
\end{align*}
Analogously, the H\"older regularity assumption of $\mathbf{X}^N$ also yields $$\mathbb P(R_3^N\ge3^{-(q+1)})\to0,\quad\text{as } N\to\infty.$$
Plugging the above estimates on $\{R_j^N\}_{j=1,2,3}$ into \eqref{XN-X>1}, we obtain 
\begin{equation}\label{eq:XN-X}\lim_{N\to\infty}\mathbb P\big(\sup_{\mathrm{t}\in\mathbb{T}}\Vert \mathbf{X}^N(\mathrm{t})-\mathbf{X}(\mathrm{t})\Vert \ge1\big)=0.
\end{equation}

\textit{Step 3.}
Letting $N\rightarrow \infty$ on both sides of \eqref{eq:TV}, it follows from \eqref{eq:XN-X} that
\begin{align*}
\limsup_{N\to\infty}\mathrm{d}_{\mathrm{TV}}(\mathbf{X}^N(\mathrm{t}),\mathbf{X}(\mathrm{t}))\le 4\mathbb P(\mathbf{\Omega}_{R-1}^\textrm{c})+\limsup_{N\to\infty}\mathrm{d}_{\mathrm{TV}}(\mathbf{X}^N_R(\mathrm{t}),\mathbf{X}_R(\mathrm{t})),
\end{align*}
which together with $\lim_{R\to \infty}\mathbb{P}(\mathbf{\Omega}_R)=1$ in (C1) leads to \eqref{eq:RN}.
\end{proof}

In order to apply Proposition \ref{prop:localization} with $\mathbf {X}=u$ and $\mathbb{T}=[0,T]\times\OO$, we first construct the localization of $u$. Denote 
$\Omega_R:=\big\{\omega\in\Omega:\sup_{(t,x)\in[0,T]\times\mathcal{O}}\vert u(t,x,\omega)\vert \le R\big\}.$ 
Set $f_R(x)=K_R(x)f(x)$ for $x\in\mathbb{R}$, where the cut-off function $K_R$ is defined in \eqref{KR}.
 Then we consider the following localized Cahn--Hilliard equation
\begin{align}\label{ur}
\partial_t u_R+\Delta^2u_R=\Delta f_R(u_R)+\sigma(u_R)\dot{W},\quad R\ge1
\end{align}
with $u_R(0,\cdot)=u_0$ and  DBCs. By the fact that $u$ has a.s.\ continuous trajectories and the local property of stochastic integrals, one has
\begin{align}\label{Dlocu}
u=u_R ~\textrm{on} ~\Omega_R \quad \text{a.s.},\qquad \textrm{and} ~\lim_{R\rightarrow\infty}\mathbb P(\Omega_R)=1.
\end{align}
For $n\ge2$, consider the spatial FDM numerical solution $u_R^n$ and the fully discrete FDM numerical solution $u_R^{n,\tau}$ of \eqref{ur}, i.e., $u_R^n$ and $u_R^{n,\tau} $ respectively solve \eqref{unR0} and \eqref{eq:untau} with $f$ replaced by $f_R$.
Similarly to \eqref{Dlocu}, by setting 
$\Omega_{R,n}:=\big\{\omega\in\Omega: \sup_{(t,x)\in[0,T]\times\mathcal{O}}\vert u^n(t,x,\omega)\vert \le R\big\}$ and
$\Omega_{R,n,m}:=\big\{\omega\in\Omega: \sup_{(t,x)\in [0,T]\times\mathcal {O}}\vert u^{n,\tau}(t,x,\omega)\vert \le R\big\}$, we have  
\begin{gather}\label{eq:Dlocun}
u^n=u_R^n ~\textrm{on} ~\Omega_{R,n} \quad \text{a.s.},\qquad \textrm{and} ~\lim_{R\rightarrow\infty}\mathbb P(\Omega_{R,n})=1,\\\label{eq:Dlocunm}
u^{n,\tau}=u_R^{n,\tau} ~\textrm{on} ~\Omega_{R,n,m} \quad \text{a.s.},\qquad \textrm{and} ~\lim_{R\rightarrow\infty}\mathbb P(\Omega_{R,n,m})=1.
\end{gather}

The following uniform non-degeneracy condition is instrumental for the convergence of density of the numerical solution.
\begin{assumption}\label{A3}
There exists some $\sigma_0>0$ such that $\vert \sigma(x)\vert >\sigma_0$ for any $x\in\mathbb{R}$.
\end{assumption}
 Since for a fixed $R\ge1$, $f_R$ in  \eqref{ur} is infinitely differentiable with bounded derivatives of any order,  Proposition \ref{dTVR} is a direct consequence of Proposition \ref{density-Lip}.

\begin{proposition}\label{dTVR}
Let $u_0\in\mathcal{C}^3(\OO)$ and $\sigma$ be twice differentiable with bounded first and second order derivatives. Then for any $x\in\mathcal O$, 
\begin{gather*}
\lim_{n\rightarrow \infty}\mathrm{d}_{\mathrm{TV}}(u_R(T,x),u^n_R(T,x))=0,\quad\text{for any fixed } R\ge1,\\
\lim_{\tau\rightarrow \infty}\mathrm{d}_{\mathrm{TV}}(u^n_R(T,x),u^{n,\tau}_R(T,x))=0,\quad\text{for any fixed } R\ge1 \text{ and } n\ge 2.
\end{gather*}
\end{proposition}

\begin{proposition}\label{thm:main}
Let $u_0\in\mathcal{C}^3(\OO)$, $\sigma$ be twice differentiable with bounded first and second order derivatives, and
 Assumption \ref{A3} hold. Then
\begin{gather}\label{eq:thm1}
\lim_{n\rightarrow \infty}\mathrm{d}_{\mathrm{TV}}(u^n(T,x),u(T,x))=0,\quad x\in\OO,\\
\label{eq:thm2}
\lim_{n\rightarrow \infty}\lim_{\tau\rightarrow0}\mathrm{d}_{\mathrm{TV}}(u^{n,\tau}(T,x),u(T,x))=0,\quad x\in\OO.
\end{gather}
\end{proposition}
\begin{proof}
($i$) Let $\mathbf {X}=u$, $\mathbb{T}=[0,T]\times\OO$, $N=n$, $\mathbf {X}_R=u_R$, $\mathbf {X}^N=u_n$, $\mathbf {X}^N_R=u^n_R$  in Proposition \ref{prop:localization}. Note that (C1) and (C2) follow from \eqref{Dlocu} and \eqref{eq:Dlocun} respectively. In addition, (C3) can be ensured by Theorem \ref{main-strong}, and 
 (C4) is a consequence of Lemmas \ref{Holder-exact} and \ref{u2u3}.
Hence an application of Proposition \ref{prop:localization}, together with Proposition \ref{dTVR}, proves 
the first assertion \eqref{eq:thm1}.

($ii$) 
For fixed $n$, let $\mathbf {X}=u^n$, $\mathbb{T}=[0,T]\times\OO$,  $N=m=T/\tau$, $\mathbf {X}_R=u^n_R$, $\mathbf {X}^N=u^{n,\tau}$, $\mathbf {X}^N_R=u^{n,\tau}_R$ in Proposition \ref{prop:localization}.  Then (C1) and (C2) come from \eqref{eq:Dlocun} and \eqref{eq:Dlocunm} respectively.  Theorem \ref{eq:untau-untau} implies (C3), and 
 (C4) is a consequence of Lemmas  \ref{u2u3} and \ref{lemma:Holderuntau}.
Therefore, using Propositions \ref{prop:localization}--\ref{dTVR} yields that for fixed $n$,
\begin{align*}
\lim_{\tau\rightarrow0}\mathrm{d}_{\mathrm{TV}}(u^{n,\tau}(T,x),u^{n}(T,x))=0,\quad x\in\OO.
\end{align*}
This together with the triangle inequality and \eqref{eq:thm1} yields \eqref{eq:thm2}.
\end{proof}

Recall that for any $t \in (0,T]$ and $x\in\OO$,  $p_{t,x}$ is the density of 
the exact solution $u(t,x)$ to \eqref{CH}. In Theorems \ref{Density-n} and \ref{Density-n-dis}, we have also shown that for $k \in \mathbb{Z}_{n-1}$, both the spatial FDM numerical solution $\{u^n(t,kh)\}_{t\in(0,T]}$ and the fully discrete FDM numerical solution $\{u^{n,\tau}(t_i,kh)\}_{i\in\mathbb{Z}_m}$ admit densities, which are denoted by $\{p_{t,kh}^n\}_{t\in(0,T]}$ and $\{p_{t_i,kh}^{n,\tau}\}_{i\in\mathbb{Z}_m}$, respectively.
In view of \eqref{pX-PX} and Proposition \ref{thm:main}, we have the following convergence of density in $L^1(\mathbb{R})$ of the numerical solutions.
\begin{theorem}\label{thm:density convergece}
Under the assumptions of Proposition \ref{thm:main}, for any $k \in \mathbb{Z}_{n-1}$,
\begin{gather*}
\lim_{n\rightarrow\infty}\int_\mathbb{R}\vert p_{T,kh}^n(\xi)-p_{T,kh}(\xi)\vert \ud \xi= 0,\\
\lim_{n\rightarrow\infty}\lim_{\tau\rightarrow0}\int_\mathbb{R}\vert p_{T,kh}^{n,\tau}(\xi)-p_{T,kh}(\xi)\vert \ud \xi= 0.
\end{gather*}
\end{theorem}

\appendix

\section{Convergence of density: Lipschitz case}\label{secA1}

\setcounter{equation}{0}
\setcounter{subsection}{0}
\renewcommand{\theequation}{A.\arabic{equation}}

This section is devoted to studying the density convergence of the spatial and fully discrete FDMs for the stochastic Cahn--Hilliard equation with Lipschitz nonlinearities.  This is motivated by proving Proposition \ref{dTVR}.
\textit{We adopt a slight abuse of notation in the appendix. The coefficient $f$ in \eqref{CH} is no longer $f(x)=x^3-x$, but a general Lipschitz continuous function.} More precisely, we will work under the following assumption.
\begin{assumption}\label{Asp1}
 $f$ and $\sigma$  are twice differentiable with bounded derivatives of first and second order.
\end{assumption}

The main approach is based on the Malliavin calculus.
Let us introduce some notations in the context of the Malliavin calculus with respect to the space-time white noise (see e.g., \cite{DN06}). 
The isonormal Gaussian family $\{W(h),h\in\mathbb{H}\}$ corresponding to $\mathbb{H}:=L^2([0,T]\times\mathcal{O})$ is  given by the Wiener integral
$W(h)=\int_{0}^T\int_\mathcal{O} h(s,y)W(\mathrm{d} s,\mathrm{d} y).$ 
Denote by
$\mathcal{S}$ the class of smooth real-valued random variables of the form 
\begin{equation}\label{smoothfunctional}
X= \varphi(W(h_1),\ldots,W(h_n)),
\end{equation}
where $\varphi\in\mathcal{C}_p^\infty(\mathbb{R}^n),$ $h_i\in \mathbb{H},\, i=1,\ldots,n, \,n\ge 1.$
Here $\mathcal{C}_p^\infty(\mathbb{R}^n)$ is the space of all $\mathbb{R}$-valued smooth functions on $\mathbb{R}^n$ whose partial derivatives have at most polynomial growth.
The Malliavin derivative of $X\in\mathcal S$ of the form \eqref{smoothfunctional} is an $\mathbb{H}$-valued random variable given by 
$DX=\sum_{i=1}^n\partial_i\varphi(W(h_1),\ldots,W(h_n))h_i,$
which is also a random field $DX=\{D_{\theta,\xi}X, (\theta,\xi)\in[0,T]\times\mathcal{O}\}$ with
$D_{\theta,\xi}X\!=\sum_{i=1}^n\partial_i\varphi(W(h_1),\ldots,W(h_n))h_i(\theta,\xi)$
for almost everywhere $(\theta,\xi,\omega)\in[0,T]\times\mathcal{O}\times\Omega$.
For any $p\ge 1$, we denote the domain of $D$ in $L^p(\Omega;\mathbb{R})$ by $\mathbb{D}^{1,p}$, meaning that $\mathbb{D}^{1,p}$ is the closure of $\mathcal{S}$ with respect to the norm
$$\Vert X\Vert _{1,p}=\left(\mathbb{E}\left[\vert X\vert ^p+\Vert DX\Vert _\mathbb{H}^p\right]\right)^{\frac{1}{p}}.$$

We define the iteration of the operator $D$ in such a way that for $X\in\mathcal S$, the iterated derivative $D^k X$ is an $\mathbb{H}^{\otimes k}$-valued
random variable. More precisely,  for $k\in\mathbb N_+$, $D^k X=\{D_{r_1,\theta_1}\cdots D_{r_k,\theta_k}X,(r_i,\theta_i)\in[0,T]\times\mathcal{O}\}$  is a measurable function on the product space $([0,T]\times\mathcal{O})^k \times \Omega$. 
Then for  $p\ge1$, $k\in\mathbb N$, denote by $\mathbb{D}^{k,p}$ the completion of $\mathcal{S}$ with respect to the norm
$$\Vert X\Vert _{k,p}=\big(\mathbb{E}\big[\vert X\vert ^p+\sum_{j=1}^{k}\Vert D^jX\Vert _{\mathbb{H}^{\otimes j}}^p\big]\big)^{\frac{1}{p}}.$$
Define $\mathbb{D}^{k,\infty}:=\bigcap_{p\ge 1} \mathbb{D}^{k,p}$ and $\mathbb{D}^{\infty}:=\bigcap_{k\ge 1} \mathbb{D}^{k,\infty}$ to be  topological projective limits. 
The following proposition allows one to obtain the convergence of density of a sequence of random variables from the convergence in $\mathbb D^{1,2}$.
\begin{proposition}\textup{\cite[Theorem 4.2]{NP13}}\label{TVconvergence}
Let $\{X_N\}_{N\ge1}$ be a sequence in $\mathbb D^{1,2}$ such that each $X_N$ admits a density. Let $X_\infty\in\mathbb D^{2,4}$ and let $0<\alpha\le 2$ be such that $\mathbb{E}[\Vert DX_\infty\Vert _{\mathbb{H}}^{-\alpha}]<\infty$. If $X_N\rightarrow X_\infty$ in $\mathbb D^{1,2}$, then there exists a constant $c>0$ depending only on $X_\infty$ such that for any $N\ge1$, 
$$\mathrm{d}_{\mathrm{TV}}(X_N,X_\infty)\le c\Vert X_N-X_\infty\Vert _{1,2}^{\frac{\alpha}{\alpha+2}}.$$
\end{proposition}

From \cite[Proposition 3.1]{CH20} or \cite[Lemma 3.2]{CC01}, one can see that if $f$ and $\sigma$ in \eqref{CH} is continuously differentiable with bounded derivatives, then for any $(t,x)\in[0,T]\times\mathcal{O}$, $u(t,x)\in\mathbb{D}^{1,2}$ and satisfies
\begin{align*}
D_{r,z}u(t,x)&=G_{t-r}(x,z)\sigma(u(r,z))+\int_r^t\int_{\mathcal O}\Delta G_{t-s}(x,y)f^\prime(u(s,y))D_{r,z}u(s,y)\mathrm{d} y\mathrm{d} s\\
&\quad+\int_r^t\int_{\mathcal O}G_{t-s}(x,y)\sigma^\prime(u(s,y))D_{r,z}u(s,y)W(\mathrm{d} s,\mathrm{d} y),
\end{align*}
if $r\le t$, and $D_{r,z}u(t,x)=0$, if $r>t$.
Further, we study the regularity of the exact solution $u(t,x)$ to  \eqref{CH} in the Malliavin--Sobolev space $\mathbb{D}^{k,p}$.

\begin{lemma}\label{lem:chimp0}
Given $k\in\mathbb N_+$, let $f$ and $\sigma$ be $k$th differentiable with 
bounded derivatives up to order $k$. Then  $u(t,x)\in \mathbb{D}^{k,\infty}$ for any $(t,x)\in[0,T]\times\mathcal{O}$. Moreover, for any $p\ge1$, there exists $C=C(k,p,T)$ such that for any $(t,x)\in[0,T]\times\mathcal{O}$,
$$\Vert u(t,x)\Vert _{k,p}\le C.$$
\end{lemma}
\begin{proof}
Define the Picard approximation by $w^0(t,x)=u_0(x)$, and for $i\in\mathbb N$,
\begin{align*}
w^{i+1}(t,x)&=\mathbb{G}_tu_0(x)+(\Delta G)*f(w^{i})(t,x)+G\diamond \sigma(w^{i})(t,x),\quad (t,x)\in[0,T]\times\mathcal{O}; 
\end{align*}
see \eqref{eq:SDC} and \eqref{eq:SSC} for more details.
Fix $(t,x)\in[0,T]\times\mathcal{O}$.
In view of \cite[Lemma 1.5.3]{DN06},
the proof of $u(t,x)\in\mathbb{D}^{k,\infty}$ boils down to proving that 
\begin{itemize}
\item[($i$)]  $\{w^i(t,x)\}_{i\ge1}$ converges to $u(t,x)$ in $L^p(\Omega)$ for every $p\ge1$;

\item[($ii$)]   for any  $p\ge1$,
$\sup_{i\ge0}\Vert w^i(t,x)\Vert _{k,p}<\infty.$
\end{itemize}

Property ($i$) and property ($ii$) with $k=1$ and $p=2$
can be obtained in the same way as in \cite[Lemma 3.2]{CC01} (the sequence $\{w^i(t,x)\}_{i\ge1}$ corresponds to $\{u_{n,k}(t,x)\}_{k\ge1}$ in \cite{CC01}). 
The proof of property ($ii$) with general $k,p\ge1$ is omitted since it is standard and similar to those for other kinds of SPDEs with Lipschitz continuous coefficients; see \cite[Proposition 4.3]{BP98} for the case of stochastic heat equations and \cite[Theorem 1]{QS04} for the case of stochastic wave equations.
\end{proof}

Similar to properties ($i$) and ($ii$), the standard Picard approximation also shows that 
under Assumption \ref{Asp1},  $u^n(t,x)\in\mathbb{D}^{1,2}$ for any $(t,x)\in[0,T]\times\mathcal{O}$ and $u^{n,\tau}(t_i,x)\in\mathbb{D}^{1,2}$ for any $i\in\mathbb{Z}_m^0$ and $x\in\mathcal{O}$.

\begin{lemma}\label{eq.strong-trun}
Let  $u_0\in\mathcal{C}^3(\mathcal{O})$ and Assumption \ref{Asp1} hold. Then for every $p\ge1$ and $0<\epsilon\ll1$, there exist some constants $C_1=C(p,T)$ and $C_2=C(p,T,\epsilon)$ such that for any $(t,x)\in[0,T]\times\mathcal{O}$ and $i\in\mathbb{Z}_m$,
\begin{gather*}
\|u^{n}(t,x)-u(t,x)\|_{L^p(\Omega)}\le C_1n^{-1},\\
\|u^{n,\tau}(t_i,x)-u^n(t_i,x)\|_{L^p(\Omega)}\le C_2\tau^{\frac38-\epsilon}.
\end{gather*} 
\end{lemma}
\begin{proof}
Notice that
$u^n-u=(u^n-\tilde u^{n})+(\tilde u^{n}- u)
$ and $u^{n,\tau}-u^n=(u^{n,\tau}-\tilde u^{n,\tau})+(\tilde u^{n,\tau}-u^n).$
The errors $\|\tilde u^{n}(t,x)- u(t,x)\|_{L^p(\Omega)}$ and $\|\tilde u^{n,\tau}(t_i,x)-u^n(t_i,x)\|_{L^p(\Omega)}$ have been tackled in Propositions \ref{utilde-uh-1} and \ref{eq:un-untau} (see Remark \ref{Rem1}), respectively. Since $f$ and $\sigma$ are globally Lipschitz, the estimates of $\|u^n(t,x)-\tilde u^n(t,x)\|_{L^p(\Omega)}$ and $\|u^{n,\tau}(t_i,x)-\tilde u^{n,\tau}(t_i,x)\|_{L^p(\Omega)}$ are standard by using \eqref{GnGn}, \eqref{GnGntau} and the singular Gronwall inequality.
 \end{proof}

\begin{proposition}\label{eq.strong-trun-D}
Let  $u_0\in\mathcal{C}^3(\mathcal{O})$ and Assumption \ref{Asp1} hold.  Then there exists some constant $C$ such that for any $(t,x)\in[0,T]\times\mathcal{O}$,
\begin{align*}
\|u^n(t,x)-u(t,x)\|_{1,2}&\le Cn^{-1}.
\end{align*}
\end{proposition}

The proof of Proposition \ref{eq.strong-trun-D} is standard and thus is omitted.
In order to apply Proposition \ref{TVconvergence} with $X_\infty=u(T, x)$,
we further investigate the inverse moment estimate of  $\Vert Du(t, x)\Vert _{\mathbb{H}}$.

\begin{lemma}\label{lem:chimp}
Under Assumptions \ref{A3}--\ref{Asp1}, for any $x\in\OO$, there is $\rho\in(0,1]$ such that
\begin{align*}
\mathbb{E}\big[\Vert Du(T,x)\Vert ^{-2\rho}_{\mathbb{H}}\big]\le C(\rho,T).
\end{align*}
\end{lemma}
\begin{proof}
We need to use \cite[Proposition 3.2]{CH20}, which 
is summarized as follows: under Assumption \ref{A3},
if $x_i\in\mathcal{O}$, $i=1,\ldots,d$, are distinct points, then for some $p_0>0$, there exists $\varepsilon_0=\varepsilon_0(p_0)$ such that for all $\varepsilon\in (0,\varepsilon_0)$, 
\begin{align}\label{C(t)}
\sup_{\xi\in\mathbb{R}^d,\Vert \xi\Vert =1}\mathbb P\left(\xi^\top \mathbb C(t)\xi\le \varepsilon\right)\le \varepsilon^{p_0},
\end{align}
where $\mathbb C(t):=(\langle Du(t,x_i),Du(t,x_j)\rangle_{\mathbb{H}})_{1\le i,j\le d}$ denotes the Malliavin covariance matrix of the random vector $(u(t,x_1),\ldots,u(t,x_d))$ (the notation $u$ corresponds to $X_R$ in \cite{CH20}).
As a consequence of \eqref{C(t)} with $d=1$ and $t=T$, we have that for all $0<\varepsilon\le \varepsilon_0$, 
$\mathbb P(\Vert Du(T,x)\Vert ^2_{\mathbb{H}}\le\varepsilon)\le \varepsilon^{p_0},$
which implies that for any $ \rho< p_0$,
\begin{align*}
\sum_{n=1}^\infty n^{\rho-1}\mathbb P\left(\Vert Du(T,x)\Vert ^{-2}_{\mathbb{H}}\ge n\right)\le \sum_{n=1}^{\lfloor\varepsilon_0^{-1}\rfloor} n^{\rho-1}+\sum_{n= \lfloor\varepsilon_0^{-1}\rfloor+1}^\infty n^{\rho-1}n^{-p_0}\le C(\rho,\varepsilon_0).
\end{align*}
Then we have that for $0< \rho< \min\{p_0,1\}$ and $Z:=\Vert Du(T,x)\Vert ^{-2}_{\mathbb{H}}$,
\begin{align*}
\mathbb E\left[Z^{\rho}\right]&\le
1+\sum_{n=1}^\infty (n+1)^{\rho}\mathbb P(n\le Z< n+1)\le 2+\sum_{n=1}^\infty \left((n+1)^\rho-n^\rho\right)\mathbb P(Z\ge n)\\
&\le 2+\rho\sum_{n=1}^\infty n^{\rho-1}\mathbb P(Z\ge n)\le C(\rho,\varepsilon_0).
\end{align*}
 The proof is completed. 
\end{proof}

In view of the Bouleau and Hirsch’s criterion (see e.g., \cite[Theorem 2.1.2]{DN06}), Lemmas \ref{lem:chimp0} and \ref{lem:chimp} imply that under Assumptions \ref{A3} and \ref{Asp1}, for any $(t,x)\in[0,T]\times\OO$, the exact solution $u(t,x)$ to \eqref{CH} admits a density.
We are ready to give the main result of the appendix.

\begin{proposition}\label{density-Lip}
Under Assumptions \ref{A3}--\ref{Asp1}, if $u_0\in\mathcal{C}^3(\mathcal{O})$, then for any $x\in\OO$,
\begin{align}\label{eq:density-Lip}
\lim_{n\rightarrow \infty}\mathrm{d}_{\mathrm{TV}}(u(T,x),u^n(T,x))=0,
\end{align}
and for any fixed $n\ge2$,
\begin{align}\label{eq:density-Lip-fully}
\lim_{\tau\rightarrow 0}\mathrm{d}_{\mathrm{TV}}(u^n(T,x),u^{n,\tau}(T,x))=0.
\end{align}
\end{proposition}\begin{proof}
 ($i$) Lemma \ref{lem:chimp0}, Proposition \ref{eq.strong-trun-D}, and
 Lemma \ref{lem:chimp} indicate that the conditions of 
Proposition \ref{TVconvergence}  are fulfilled with $\alpha=2\rho$, $X_n=u^n(T,x)$ and $X_\infty=u(T,x)$. Thus \eqref{eq:density-Lip} can be obtained from Proposition \ref{TVconvergence}.

($ii$) In a similar manner, \eqref{eq:density-Lip-fully} can be proved applying Proposition \ref{TVconvergence} with $X_\infty=u^n(T,x)$ and $X_m=u^{n,\tau}(T,x)$, and we omit the details.
\end{proof}

In Sections \ref{S5}--\ref{S6},
 we give the strong convergence orders and density convergence of
 the spatial and fully discrete FDMs applied to \eqref{CH} for \textit{Case 1: $f(x)=x^3-x$}.
In the appendix, 
 we also present 
 that the above results also hold for \textit{Case 2: $f$ is twice differentiable with bounded derivatives of first and second order}.
 Although the results are the same in both cases, the main techniques are essentially different.
We take the spatial FDM as instance to point out some differences between these two cases below.
 \begin{itemize}
 \item[$(1)$] In both  \textit{Case 1} and \textit{Case 2},
  the strong convergence analysis of the spatial FDM relies on the introduction of the auxiliary process $\tilde{u}$  and the error estimate between $u$ and $\tilde{u}$ in Proposition \ref{utilde-uh-1} (see Theorem \ref{main-strong} and Lemma \ref{eq.strong-trun} for \textit{Case 1} and \textit{Case 2}, respectively). 
  However, we emphasis that the introduction of the auxiliary process $\tilde{u}$ is mainly used to deal with \textit{Case 1}, and is not necessary for \textit{Case 2}. Alternatively, Lemma \ref{eq.strong-trun} can be proved based on the following decomposition
{\small  
  \begin{align*}
u^n(t,x)-u(t,x)&=\int_0^t\int_{\mathcal O}[\Delta_nG^n_{t-s}(x,y)-\Delta G_{t-s}(x,y)]f(u^n(s,\kappa_n(y)))\mathrm{d} y\mathrm{d} s\\
&\quad+\int_0^t\int_{\mathcal O}\Delta G_{t-s}(x,y)[f(u^n(s,\kappa_n(y)))-f(u^n(s,y))]\mathrm{d} y\mathrm{d} s\\
&\quad+\int_0^t\int_{\mathcal O}\Delta G_{t-s}(x,y)[f(u^n(s,y))-f(u(s,y))]\mathrm{d} y\mathrm{d} s\\
&\quad+\int_0^t\int_{\mathcal O}[G^n_{t-s}(x,y)-G_{t-s}(x,y)]\sigma(u^n(s,\kappa_n(y)))W(\mathrm{d} s,\mathrm{d} y)\\
&\quad+\int_0^t\int_{\mathcal O}G_{t-s}(x,y)[\sigma(u^n(s,\kappa_n(y)))-\sigma(u^n(s,y))]W(\mathrm{d} s,\mathrm{d} y)\\
&\quad+\int_0^t\int_{\mathcal O} G_{t-s}(x,y)[\sigma(u^n(s,y))-\sigma(u(s,y))]W(\mathrm{d} s,\mathrm{d} y).
\end{align*}}However, the above way of decomposition does not work for the proof of Theorem \ref{main-strong}, due to the absence of the Lipschitz property of $f$ in \textit{Case 1}.

 \item[$(2)$] In \textit{Case 2}, the key to deriving the density convergence of the numerical solution is the application of Proposition \ref{TVconvergence}, whose prerequisite involves the Malliavin differentiability of the exact solution. However, in \textit{Case 1}, we are only aware that the exact solution $u(t,x)$ is
 locally Malliavin differentiable, and it is still unclear to us whether $u(t,x)$ belongs to $\mathbb D^{1,2}$ or not. This brings difficulty in applying
  Proposition \ref{TVconvergence} to prove the density convergence for \textit{Case 1}. Instead,  we propose a novel localization argument,  which enables us to convert the proof of the  density convergence in \textit{Case 1} into the strong convergence analysis of the numerical method in \textit{Case 1} and the density convergence of the numerical solution in \textit{Case 2}.
 \end{itemize}

\bibliographystyle{plain}
\bibliography{mybibfile}

\end{document}